\documentclass[letterpaper,11pt]{article}

\usepackage[usenames,dvipsnames]{color}
\usepackage{changes}

\renewcommand{\added}[1]{#1}

\usepackage[margin=1in]{geometry}  % set the margins to 1in on all sides

\usepackage{booktabs}

\usepackage[%dvips,
            CJKbookmarks=true,
            bookmarksnumbered=true,
            bookmarksopen=true,
            colorlinks=true,
            citecolor=red,
            linkcolor=blue,
            anchorcolor=red,
            urlcolor=blue
            ]{hyperref}

\usepackage{subfigure}
\usepackage{amsmath,amssymb}
\usepackage{amsthm,enumerate}
\usepackage{versions}
\excludeversion{forme}
\usepackage[english]{babel}
\usepackage{graphicx}
\usepackage{prettyref}
\newrefformat{eq}{(\ref{#1})}
\newrefformat{thm}{Theorem~\ref{#1}}
\newrefformat{th}{Theorem~\ref{#1}}
\newrefformat{chap}{Chapter~\ref{#1}}
\newrefformat{sec}{Section~\ref{#1}}
\newrefformat{ssec}{Section~\ref{#1}}
\newrefformat{algo}{Algorithm~\ref{#1}}
\newrefformat{proc}{Procedure~\ref{#1}}
\newrefformat{fig}{Fig.~\ref{#1}}
\newrefformat{tab}{Table~\ref{#1}}
\newrefformat{rmk}{Remark~\ref{#1}}
\newrefformat{rem}{Remark~\ref{#1}}
\newrefformat{clm}{Claim~\ref{#1}}
\newrefformat{def}{Definition~\ref{#1}}
\newrefformat{cor}{Corollary~\ref{#1}}
\newrefformat{lmm}{Lemma~\ref{#1}}
\newrefformat{prop}{Proposition~\ref{#1}}
\newrefformat{pr}{Proposition~\ref{#1}}
\newrefformat{app}{Appendix~\ref{#1}}
\newrefformat{apx}{Appendix~\ref{#1}}
\newrefformat{ex}{Example~\ref{#1}}
\newrefformat{exer}{Exercise~\ref{#1}}
\newrefformat{soln}{Solution~\ref{#1}}

\usepackage{float}
\usepackage[ruled, vlined, lined]{algorithm2e}

\newcommand{\MLE}{\text{MLE}}

\newcommand{\Expect}{\mathbb{E}}

\newcommand{\Prob}{\mathbb{P}}
\newcommand{\prob}[1]{\mathbb{P}\left[#1\right]}

\newcommand{\diff}{{\rm d}}
\newcommand{\eg}{e.g.\xspace}

\newcommand{\iid}{i.i.d.\xspace}

\newcommand{\iiddistr}{{\stackrel{\text{\iid}}{\sim}}}

\newcommand{\stepa}[1]{\overset{\rm (a)}{#1}}
\newcommand{\stepb}[1]{\overset{\rm (b)}{#1}}
\newcommand{\stepc}[1]{\overset{\rm (c)}{#1}}
\newcommand{\stepd}[1]{\overset{\rm (d)}{#1}}

\newcommand{\TV}{{\sf TV}}
\newcommand{\KL}{{\sf KL}}

\newcommand{\bfM}{{\mathbf{M}}}
\newcommand{\bfm}{{\mathbf{m}}}

\newcommand{\supp}{{\mathrm{supp}}}

\newcommand{\pth}[1]{\left( #1 \right)}
\newcommand{\qth}[1]{\left[ #1 \right]}
\newcommand{\sth}[1]{\left\{ #1 \right\}}
\newcommand{\abs}[1]{\left| #1 \right|}

\newcommand{\naturals}{\mathbb{N}}
\newcommand{\integers}{\mathbb{Z}}

\newcommand{\reals}{{\mathbb{R}}} %% mine
\newcommand{\rank}{\mathrm{rank}} %% mine

\newcommand{\Var}{\mathrm{Var}}
\newcommand{\bfj}{\mathbf{j}}
\newcommand{\symmetric}{\mathbb{S}}

\newcommand{\lnorm}[2]{\left\|{#1} \right\|_{{#2}}}
\newcommand{\norm}[1]{\left\|{#1} \right\|}
\newcommand{\Fnorm}[1]{\lnorm{#1}{\rm F}}

\newcommand{\inprod}[2]{\ensuremath{\langle #1 , \, #2 \rangle}}
\newcommand{\iprod}[2]{\left \langle #1, #2 \right\rangle}
\newcommand{\Iprod}[2]{\langle #1, #2 \rangle}
\newcommand{\indc}[1]{{\mathbf{1}_{\left\{{#1}\right\}}}}

\def\bone{{\mathbf 1}}
\newcommand{\Norm}[1]{\|{#1} \|}

\newcommand{\calA}{{\mathcal{A}}}

\newcommand{\calF}{{\mathcal{F}}}
\newcommand{\calG}{{\mathcal{G}}}

\newcommand{\calL}{{\mathcal{L}}}
\newcommand{\calM}{{\mathcal{M}}}
\newcommand{\calN}{{\mathcal{N}}}

\newcommand{\calP}{{\mathcal{P}}}

\newcommand{\calS}{{\mathcal{S}}}

\newcommand{\calW}{{\mathcal{W}}}
\newcommand{\calX}{{\mathcal{X}}}

\newcommand{\Wsliced}{W^{\sf sliced}}

\numberwithin{equation}{section}
\theoremstyle{plain}

\newtheorem{thm}{Theorem}[section]
\newtheorem{lem}[thm]{Lemma}
\newtheorem{theorem}[thm]{Theorem}
\newtheorem{lemma}[thm]{Lemma}

\theoremstyle{definition}

\newtheorem{rem}{Remark}

\newcommand{\mc}{\mathcal}
\newcommand{\mb}{\mathbb}

\def\({\left(}
\def\){\right)}

\DeclareMathOperator*{\argmin}{argmin}
\DeclareMathOperator*{\argmax}{argmax}
\DeclareMathOperator*{\trace}{tr}

\newcommand{\DMM}{\mathsf{DMM}}
\newcommand\numberthis{\addtocounter{equation}{1}\tag{\theequation}}

\title{Optimal estimation in the high-dimensional Gaussian mixture model}
\title{Optimal estimation of high-dimensional Gaussian location mixtures}
\author{Natalie Doss, Yihong Wu, Pengkun Yang, and Harrison H.~Zhou
\thanks{
N.\ Doss, Y.\ Wu, and H.~H.\ Zhou are with Department of Statistics and Data Science, Yale University, New Haven, CT, USA, 
\texttt{\{natalie.doss,yihong.wu,huibin.zhou\}@yale.edu}. P.\ Yang is with the Center for Statistical Science at Tsinghua University, Beijing, China, \texttt{yangpengkun@tsinghua.edu.cn}.
Y.~Wu is supported in part by the NSF Grants CCF-1527105, CCF-1900507, an NSF CAREER award CCF-1651588, and an Alfred Sloan fellowship.
H.~H.~Zhou is supported in part by NSF grants DMS 1811740, DMS 1918925, and NIH grant 1P50MH115716.
}}
\date{\today}

\begin{document}

\maketitle

\begin{abstract}

This paper studies the optimal rate of estimation in a finite Gaussian location mixture model in high dimensions without separation conditions. 
We assume that the number of components $k$ is bounded and that the centers lie in a ball of bounded radius, while allowing the dimension $d$ to be as large as the sample size $n$.
Extending the one-dimensional result of Heinrich and Kahn \cite{HK2015}, we show that the minimax rate of estimating the mixing distribution in Wasserstein distance 
is $\Theta((d/n)^{1/4} + n^{-1/(4k-2)})$, achieved by an estimator computable in time $O(nd^2+n^{5/4})$.
Furthermore, we show that the mixture density can be estimated at the optimal parametric rate $\Theta(\sqrt{d/n})$ in Hellinger distance and provide a computationally efficient algorithm to achieve this rate in the special case of $k=2$. 
%; however, no computationally efficient algorithm is known to achieve the optimal rate except for the special case of $k=2$. 

%, which can be attained an estimator in polynomial time for the special case of $k=2$; however, no computationally efficient algorithm is known to achieve the optimal parametric rate for $k\geq 3$.
%While this rate can be achieved by a computationally efficient estimator for the special case of $k=2$, no such algorithm is known for $k\geq 3$.
%For density estimation we show that the optimal Hellinger rate is $\Theta(\sqrt{d/n})$; however, no computationally efficient algorithm is known to achieve the optimal parametric rate.

Both the theoretical and methodological development rely on a careful application of the method of moments. Central to our results is the observation that the information geometry of finite Gaussian mixtures is characterized by the moment tensors of the mixing distribution, whose low-rank structure can be exploited to obtain a sharp local entropy bound. 

%Both the theoretical and methodological development rely on a careful application of method of moments, and central to our results is the idea that the moments tensor characterizes the information geometry of finite Gaussian mixtures and provides a good parametrization for determining the sharp metric entropy structure. 
\end{abstract}

\tableofcontents

%!TEX root = gmm.tex

\section{Introduction} \label{sec:intro}

Mixture models are useful tools for dealing with heterogeneous data. A mixture model posits that the data are generated from a collection of sub-populations, each governed by a different distribution. The Gaussian mixture model is one of the most widely studied mixture models because of its simplicity and wide applicability; however, optimal rates of both parameter and density estimation in this model are not well understood in high dimensions.
%when the mixing distribution centers are both high dimensional and allowed to be arbitrarily close to each other. 
Consider the $k$-component Gaussian location mixture model 
%(which we shall refer as $k$-GM) 
in $d$ dimensions:
\begin{equation} 
\begin{aligned}  \label{eq:gmm_model}
X_1,\ldots,X_n \iiddistr \sum_{j = 1}^k w_j N(\mu_j,\sigma^2 I_d),
\end{aligned}
\end{equation}
where $\mu_j\in\reals^d$ and $w_j\geq 0$ are the center and the weight of the $j$th component, respectively, with $\sum_{j=1}^k w_j=1$.
Here the scale parameter $\sigma^2$ and \added{$k$ as an upper bound on the number of components}
%the number of components $k$} 
are assumed to be known; 
for simplicity, we assume that $\sigma^2 = 1$.  
Equivalently, we can view the Gaussian location mixture \prettyref{eq:gmm_model} as the convolution
\begin{equation} 
\begin{aligned}  \label{eq:gmm_model1}
%X_1,\ldots,X_n \iiddistr 
P_\Gamma \triangleq \Gamma * N(0,I_d)
\end{aligned}
\end{equation}
between the standard normal distribution and the \emph{mixing distribution}
\begin{equation}
\Gamma = \sum_{j = 1}^k w_j \delta_{\mu_j},
\label{eq:Gamma}
\end{equation}
which is a $k$-atomic distribution on $\reals^d$.

%Extending the previous work of \cite{Chen_1995,HK2015,WY18} for Gaussian mixture models to high dimensions, the goal of the present paper is to understand the fundamental limit of parameter estimation 
%(in the sense of mixing distributions) 
%and density estimation.
 %%(as opposed to classification).
%%with the emphasis on high dimensions.
%%, as opposed to classification. 
%To this end, the most interesting regime is where the centers lie in a ball of bounded radius, in which case 
%consistent clustering is impossible but consistent estimation of the mixing distribution and the mixture density is nevertheless possible.
%We will make no other assumptions such as separation or non-colinearity between centers or lower bounds on the weights that are prevelant in the literature of Gaussian mixtures, as none of them is statistically necessary for estimation. 
%
%\nb{In~\prettyref{eq:gmm_model}, the most interesting regime is one of arbitrary separation, i.e., one in which the model centers may overlap and the model weights may not be bounded away from zero. In this case, clustering is impossible, but consistent estimation of the mixing distribution and mixture density is nevertheless possible. A line of research begun by~\cite{Chen_1995} and culminating in the works of~\cite{HK2015,WY18} provide optimal convergence rates and fast algorithms in this setting when $d=1$. The goal of this paper is to understand the fundamental limits of parameter and density estimation in Gaussian mixtures of both arbitrary separation and arbitrary dimension, i.e., to extend the above works to general $d$.}

For the purpose of estimation, the most interesting regime is one in which the centers lie in a ball of bounded radius and are allowed to overlap arbitrarily. In this case, consistent clustering is impossible but the mixing distribution and the mixture density can nonetheless be accurately estimated. In this setting, a line of research including~\cite{Chen_1995,HK2015,WY18}
%{begun by~\cite{Chen_1995} and culminating in the works of ~\cite{HK2015},~\cite{WY18}}
obtained optimal convergence rates and practical algorithms for one-dimensional Gaussian mixtures.
%Extending these works to high dimensions, the goal of this paper is to further the statistical and algorithmic understanding of parameter and density estimation in Gaussian mixtures in the same assumption-free framework, without imposing other conditions (such as separation or non-colinearity between centers, lower bounds on the weights) that are prevalent in the literature of Gaussian mixtures, as none of them is statistically necessary for estimation. 
The goal of this paper is to extend these works to high dimension. That is, we seek to further the statistical and algorithmic understanding of parameter and density estimation in high-dimensional Gaussian mixtures in an assumption-free framework, without imposing conditions such as separation or non-collinearity between centers or lower bounds on the weights that are prevalent in the literature on Gaussian mixtures but that are not statistically necessary for estimation.

\subsection{Main results}
\label{sec:main}

We start by defining the relevant parameter space. 
Let $\mc G_{k, d}$ denote the collection of $k$-atomic distributions supported on a ball of radius $R$ in $d$ dimensions,\footnote{If the mixing distributions have unbounded support, the minimax risk under the Wasserstein distance is infinite (see \cite[Sec.~4.4]{WY18}).} i.e., 
\begin{equation}
\calG_{k,d} \triangleq 
\sth{ \Gamma = \sum_{j =1}^k w_j \delta_{\mu_j}: \mu_j \in \reals^d, \norm{\mu_j}_2 \le R, w_j\geq 0,\sum_{j=1}^k w_j=1},
\label{eq:Gkd}
\end{equation}
where $\norm{\cdot}_2$ denotes the Euclidean norm. Throughout the paper, $R$ is assumed to be an absolute constant. 
%This is the most interesting regime for parameter and density estimation (as opposed to clustering). 
The corresponding collection of $k$-Gaussian mixtures ($k$-GMs) is denoted by  
\begin{equation}
\mc P_{k, d} = \{ P_{\Gamma} : \Gamma \in \mc G_{k, d} \}, \qquad P_{\Gamma} = \Gamma * N(0,I_d).
\label{eq:Pkd}
\end{equation}
Let $\phi_d(x) = (2\pi)^{-d/2} e^{-\|x\|_2^2/2}$ denote the standard normal density in $d$ dimensions. Then the density of $P_\Gamma$ is given by
\begin{equation}
p_\Gamma(x) = \sum_{j = 1}^k w_j  \phi_d\(x - \mu_j\).
\label{eq:pGamma}
\end{equation}
%which is a parametric distribution with the $2k - 1$ parameters $\mu_1, \hdots, \mu_k$ and $w_1, \hdots, w_k$. In many scenarios, it is possible to estimate the individual weights and location parameters with small error. Under our assumptions, we may be able to estimate some of the components or weights with large error but still have small error in the estimation of the entire mixing distribution. Therefore, it is natural for us to frame the model as in~\prettyref{eq:gmm_model}. 

%We first discuss the problem of parameter estimation in the sense of mixing distributions. This view allows for the development of a meaningful statistical theory in an assumption-free framework~\cite{HK2015,WY18} as the mixture model is always uniquely identified by the mixing distribution. 

We first discuss the problem of parameter estimation. The distribution \prettyref{eq:gmm_model} has $kd + k - 1$ parameters: $\mu_1, \hdots, \mu_k\in\reals^d$ and $w_1, \hdots, w_{k}$ that sum up to one. Without extra assumptions such as separation between centers or a lower bound on the weights, estimating individual parameters is clearly impossible; nevertheless, estimation of the mixing distribution $\Gamma=\sum w_i \delta_{\mu_i}$ is always well-defined. Reframing the parameter estimation problem in terms of estimating the mixing distribution allows for the development of a meaningful statistical theory in an assumption-free framework~\cite{HK2015,WY18} since the mixture model is uniquely identified through the mixing distribution.

For mixture models and deconvolution problems, the Wasserstein distance is a natural and commonly-used loss function (\cite{Chen_1995,Nguyen2013,Ho_Nguyen_2016_annals,Ho_Nguyen_2016_ejs,HK2015,WY18}). For $q \geq 1$, the $q$-Wasserstein distance (with respect to the Euclidean distance) is defined as
\begin{equation}
\label{eq:wq-def}
W_q(\Gamma, \Gamma') 
\triangleq  \pth{\inf \mb E \norm{U - U'}_2^q}^{\frac{1}{q}},
\end{equation}
where the infimum is taken over all couplings of $\Gamma$ and $\Gamma'$, i.e., joint distributions of random vectors $U$ and $U'$ with marginals $\Gamma$ and $\Gamma'$, respectively. We will mostly be concerned with the case of $q=1$, although the $W_2$-distance will make a brief appearance in the proofs. In one dimension, the $W_1$-distance coincides with the $L_1$-distance between the cumulative distribution functions~\cite{villani.topics}. For multivariate distributions, there is no closed-form expression, and the $W_1$-distance can be computed by linear programming. In the widely-studied case of the symmetric $2$-GM in which 
\begin{equation}
P_\mu = \frac{1}{2} N(\mu, I_d)+\frac{1}{2} N(-\mu, I_d),
\label{eq:sym2GM}
\end{equation}
the mixing distribution is $\Gamma_\mu = \frac{1}{2}(\delta_{-\mu}+\delta_{\mu})$, and the Wasserstein distance coincides with the commonly-used loss function 
$W_1(\Gamma_\mu,\Gamma_{\mu'}) = \min\{\|\mu-\mu'\|_2,\|\mu+\mu'\|_2\}$.
In this paper we do not postulate any separation conditions or any lower bound on the mixing weights; nevertheless, given such assumptions, statistical guarantees in $W_1$-distance can be translated into those for the individual parameters (\cite[Lemma 1]{WY18}).

%It is invariant under permutation of the distributions it compares, so it sidesteps the unidentifiability of mixing distributions. 
% unidentifiability of estimating mixing distributions (permutation invariance) and allows one to develop a meaningful statistical theory in an assumption-free framework. 
%Note that if $\Gamma$ and $\Gamma'$ have the same number of atoms and those atoms are well separated and the weights are bounded away from zero, estimation in $W_1$ distance translates to recovery of the individual parameters up to permutation; see \cite[Lemma  1]{WY18}. 
%\footnote{This can be seen directly from the definition of $W_1$ from~\prettyref{eq:wq-def}; simply take the minimum over all couplings and note that any coupling in this case is a distribution on at most four points. Then $W_1(\Gamma_{\mu},\Gamma_{\mu'})=\min_{0\le w\le 1}w\Norm{\mu-\mu'}_2+(1-w)\Norm{\mu+\mu'}_2.$}
%i.e., $W_1(\Gamma_\theta,\Gamma_{\hat\theta})=\min\{\|\theta-\hat\theta\|,\|\theta+\hat\theta\|\}$.
% \nbwu{Pengkun pls doublecheck}

For general $k$-GMs in one dimension where $k\geq 2$ is a constant, the minimax $W_1$-rate of estimating the mixing distribution is $n^{-1/(4k-2)}$, achieved by a minimum $W_1$-distance estimator \cite{HK2015} or the Denoised Method of Moments (DMM) approach \cite{WY18}.
 %\added{or even by the maximum likelihood estimator, though the MLE only achieves the rate up to undesirable logarithmic factors. The latter result can be seen by combining the $\sqrt{\log n/n}$ rate in Hellinger distance of~\cite[Lemma $2.1$ combined with Theorem $2.1$]{Ho_Nguyen_2016_annals} or~\cite[Theorem $4.1$]{Ho_Nguyen_2016_ejs} with a one-dimensional variant of~\prettyref{thm:mixing_est_density}}. 
%\nbwu{I did not get the logic here. The MLE result from \cite{Ho_Nguyen_2016_annals,Ho_Nguyen_2016_ejs} or other papers has extra logs such as the aforementioned $\sqrt{\log n/n}$, so applying \prettyref{thm:mixing_est_density} will inherit the same logs, and fails to achieve $n^{-1/(4k-2)}$. BTW, does the referee asks us to add this sentence?}
This is the worst-case rate in the absence of any separation assumptions. In the case where the centers can be grouped into $k_0$ clusters each separated by a constant, the optimal rate improves to $n^{-1/(4(k - k_0) + 2)}$, which reduces to the parametric rate $n^{-1/2}$ in the fully separated case.

Given the one-dimensional result, it is reasonable to expect that the $d$-dimensional rate is given by $(d/n)^{1/(4k-2)}$. This conjecture turns out to be incorrect, as the following result shows.
 %result of this paper on the minimax rate for estimating $\Gamma$ in setting \prettyref{eq:gmm_model} is the following theorem.

\begin{thm}[Estimating the mixing distribution] \label{thm:mixing_est}
Let \added{$k\ge 2$ and} $P_\Gamma$ be the $k$-GM defined in \prettyref{eq:gmm_model1}. Given $n$ i.i.d.~observations from $P_\Gamma$, the minimax risk of estimating $\Gamma$ over the class $\mc G_{k, d}$ satisfies
%\footnote{The notation $a \asymp_k b$ means that $c_k a \leq b \leq C_k a$ for some constants $c_k, C_k$ depending only on $k$. }
\begin{equation}
\inf_{\hat \Gamma} \sup_{\Gamma \in \mc G_{k, d}} \mb E_\Gamma W_1(\hat \Gamma, \Gamma) 
~ \asymp_k ~ \(\frac{d}{n} \)^{1/4} \wedge 1 + \(\frac{1}{n}\)^{1/(4k-2)}, \label{eq:mixing_est_avg_risk}
\end{equation}
where the notation $\asymp_k$ means that both sides agree up to constant factors depending only on $k$.
Furthermore, if $n\geq d$, there exists an estimator $\hat \Gamma$, computable in $O(nd^2) + O_k(n^{5/4})$ time, and a positive constant \added{C}, such that 
for any $\Gamma \in \mc G_{k, d}$ and any $0<\delta<\frac{1}{2}$, with probability at least $1 - \delta$, 
\begin{equation}
W_1(\hat \Gamma, \Gamma)  \le \added{C} \( \added{\sqrt{k}}\(\frac{d}{n} \)^{1/4} + \added{k^5}\(\frac{1}{n}\)^{1/(4k-2)}\sqrt{\log \frac{1}{\delta}} \). 
%\sup_{\Gamma \in \mc G_{k, d}} \Prob_\Gamma \sth{W_1(\hat \Gamma, \Gamma)  \ge C_k \( \(\frac{d}{n} \)^{1/4} + \(\frac{1}{n}\)^{1/(4k-2)}\sqrt{\log \frac{1}{\delta}} \)} \leq \delta. 
\label{eq:mixing_est_tail_bound}
\end{equation}
\end{thm}

We now explain the intuition behind the minimax rate \prettyref{eq:mixing_est_avg_risk}. The atoms $\mu_1, \hdots, \mu_k$ of $\Gamma$ span a subspace $V$ in $\mb R^d$ of dimension at most $k$. We can identify $\Gamma$ with this subspace and its projection therein, which is a $k$-atomic mixing distribution in $k$ dimensions. 
This decomposition motivates a two-stage procedure which achieves the optimal rate \prettyref{eq:mixing_est_avg_risk}: 
\begin{itemize}
	\item First, estimate the subspace $V$ by principal component analysis (PCA), then project the $d$-dimensional data onto the learned subspace. 
	Since we do not impose any spectral gap assumptions, standard perturbation theory cannot be directly applied; instead, one needs to control the Wasserstein loss incurred by the subspace estimation error, which turns out to be $(d/n)^{1/4}$.
	
	%The subspace estimation is done by principal component analysis (PCA). 
	%Since we do not impose any spectral gap assumptions, standard perturbation theory cannot be directly applied. We instead control the Wasserstein loss between $\Gamma$ and its projection onto the estimated subspace. It 

	\item Having reduced the problem to $k$ dimensions, a relevant notion is the \emph{sliced Wasserstein distance} \added{~\cite{Paty_Cuturi_2019,Deshpande_etal_2019,Niles-Weed_Rigollet_2019}}, which measures the distance of multivariate distributions by the maximal $W_1$-distance of their one-dimensional projections. We show that for $k$-atomic distributions in $\reals^k$, the ordinary and the sliced Wasserstein distance are comparable up to constant factors depending only on $k$. 
	This allows us to construct an estimator for a $k$-dimensional mixing distribution whose one-dimensional projections are simultaneously close to their estimates. 
	We shall see that the resulting error is $n^{-1/(4k-2)}$, exactly as in the one-dimensional case. 
\end{itemize}

Overall, optimal estimation in the general case is as hard as the special cases of $d$-dimensional symmetric $2$-GM \cite{Wu_Zhou_2019} and $1$-dimensional $k$-GM \cite{HK2015,WY18}.
From~\prettyref{eq:mixing_est_avg_risk}, we see that there is a threshold $d^* = n^{(2k-3)/(2k-1)}$ (e.g.,~$d^* =n^{1/3}$ for $k=2$). For $d > d^*$, the rate is governed by the subspace estimation error; otherwise, the rate is dominated by the error of estimating the low-dimensional mixing distribution.
Note that Theorem~\ref{thm:mixing_est} pertains to the optimal rate in the worst case. A faster rate is expected when the components are better separated, such as a parametric rate when the centers are separated by a constant, which, in one dimension, can be adaptively achieved by the estimators in \cite{HK2015,WY18}. However, adaptation to the separation between components in $d$ dimensions remains an open problem; see the discussion in Section~\ref{sec:discussion}. 

We note that the idea of using linear projections to reduce a multivariate Gaussian mixture to a univariate one has been previously explored 
in the context of parameter and density estimation (e.g.,~\cite{Moitra_Valiant_2010,Hardt_Price_2015,Acharya_etal_2014,Li_Schmidt_2017,WY18}); nevertheless, none of these results achieves the precision needed for attaining the optimal rate in \prettyref{thm:mixing_est}.
In particular, to avoid the unnecessary logarithmic factors, we use the denoised method of moments (DMM) algorithm introduced in \cite{WY18} to simultaneously estimate many one-dimensional projections, which is amenable to sharp analysis via chaining techniques.

%In addition to a minimax rate of convergence, we seek an algorithm providing an estimator that achieves this rate. There are two broad categories of approaches to estimation for mixtures of Gaussians: method-of-moments and maximum likelihood. This paper first provides a polynomial time algorithm based on the former approach. It is derived from the denoised-method-of-moments (DMM) method of~\cite{WY18}, which does the usual method-of-moments procedure, but projects the moments to a true moment space. 
%
%
%
%
%There is a vast literature on the method of moments for Gaussian mixtures; see~\cite{WY18},~\cite{Lindsay_1989}, and references therein for an overview. Several works in the theoretical computer science literature have provided algorithms for estimation of multivariate location-scale Gaussian mixtures under fairly weak separation conditions; see, e.g.,~\cite{Kalai_Moitra_Valiant_2010},~\cite{Moitra_Valiant_2010},~\cite{Hsu_Kakade_2013}, and~\cite{Hardt_Price_2015}. All of these use moments-based procedures to estimate the parameters of the model. For instance,~\cite{Hardt_Price_2015} consider a two-component Gaussian mixture and rely on random projections, followed by a method-of-moments algorithm on the univariate data. Their method is difficult to extend to mixtures with more than two components.~\cite{WY18} also provide a random projection-based approach to the multivariate mixture estimation problem, but the resulting estimator attains only a suboptimal rate of convergence. 

Next we discuss the optimal rate of density estimation for high-dimensional Gaussian mixtures, measured in the Hellinger distance. For distributions $P$ and $Q$, let $p$ and $q$ denote their respective densities with respect to some dominating measure $\mu$. The squared Hellinger distance between $P$ and $Q$ is $H^2(P, Q) \triangleq \int \( \sqrt{p(x)} - \sqrt{q(x)}\)^2 \mu(dx)$. In this work, we focus on \emph{proper learning}, in which the estimated density is required to be a $k$-GM. While there is no difference in the minimax rates for proper and improper density estimators, computationally the former is more challenging as it is not straightforward to find the best $k$-GM approximation to an improper estimate.

\begin{thm}[Density estimation] \label{thm:density_est}
Let $P_\Gamma$ be as in~\prettyref{eq:gmm_model1}. Then the minimax risk of estimating $P_\Gamma$ over the class $\mc P_{k, d}$ satisfies:
\begin{align}
\inf_{\hat P} \sup_{\Gamma \in \mc G_{k, d}} \mb E_{\Gamma} H(\hat P, P_\Gamma) \asymp_k \sqrt{\frac{d}{n}} \wedge 1. \label{eq:result_density}
\end{align}
%where $C_k$ is a positive constant.
Furthermore, there exists a proper density estimate $P_{\hat\Gamma}$ 
%$\hat P \in \calP_{k,d}$ 
and a positive constant $C$, such that for any $\Gamma \in \mc G_{k, d}$ and any $0<\delta<\frac{1}{2}$, with probability at least $1 - \delta$, 
\begin{equation}
%\sup_{\Gamma \in \mc G_{k, d}} \Prob_\Gamma \sth{H(P_{\hat\Gamma}, P_\Gamma) \geq C_k \sqrt{\frac{d}{n} \log \frac{1}{\delta}}} \leq \delta.
H(P_{\hat\Gamma}, P_\Gamma) \leq \added{C \sqrt{\frac{k^4d + (2k)^{2k+2}}{n} \log \frac{1}{\delta}}.}
\label{eq:density_est-highprob}
\end{equation}
%holds with probability at least $1-\delta$.
\end{thm}

%To the best of our knowledge, this result, following a long line of research on density estimation in Gaussian mixture models, is the first that shows the sharp parameteric rate without any logarithmic slack. The rate in~\prettyref{eq:result_density} follows from a local cover number calculation in the relevant parameter space. 

%To the best of our knowledge,~\prettyref{thm:density_est}, following a long line of research, is the first result that establishes the sharp rate without logarithmic factors. 
\prettyref{thm:density_est}, which follows a long line of research, is the first result we know of that establishes the sharp rate without logarithmic factors.
%\prettyref{thm:density_est} shows that the parametric rate (i.e., the ratio between the number of parameters and the sample size) can be achieved for Gaussian mixtures.
 %which is the typical behavior of smooth parametric families. For the $k$-GM, 
The parametric rate $O_k(\sqrt{d/n})$ can be anticipated by noting that the model \prettyref{eq:gmm_model} is a smooth parametric family with $k(d+1)-1$ parameters. 
Justifying this heuristic, however, is not trivial, especially in high dimensions. To this end, we apply the Le Cam-Birg\'e construction of estimators from pairwise tests, which, as opposed to the analysis of the maximum likelihood estimator (MLE) based on bracketing entropy~\cite{VdV_Wellner_1996,maugis2011non,Ghosal_VdV_2001,Ho_Nguyen_2016_annals}, relies on bounding the local Hellinger entropy without brackets. 
By doing so, we also avoid the logarithmic slack in the existing result for the MLE; see~\prettyref{sec:discussion} for more discussion.
%\added{The best bracketing entropy result in the current literature is that of~\cite{maugis2011non}, which yields a Hellinger rate of $\sqrt{d \log(nd)/n}$ for the MLE; removing this logarithmic slack is open. See~\prettyref{sec:discussion} for more discussion.}

%, i.e., the minimum number of Hellinger-balls of radius $\delta$ that cover the Hellinger-ball of radius $\epsilon$ centered at the true density.
The celebrated result of Le Cam-Birg\'e \cite{LeCam73,Birge83,birge1986estimating} shows that if the local covering number 
(the minimum number of Hellinger-balls of radius $\delta$ that cover any Hellinger-ball of radius $\epsilon$) 
 is at most $(\frac{\epsilon}{\delta})^{O(D)}$, then there exists a density estimate that achieves a squared Hellinger risk $O\(\frac{D}{n}\)$. Here the crucial parameter $D$ is known as the \emph{doubling dimension} (or the Le Cam dimension \cite{vdv02}), which serves as the effective number of parameters. In order to apply the theory of Le Cam-Birg\'e, we need to show that the doubling dimension of Gaussian mixtures is at most $O_k(d)$.

%\nb{\prettyref{thm:density_est} shows that the sharp parametric rate (i.e., the ratio between the number of parameters and the sample size) can be achieved for Gaussian mixtures. This can be anticipated by noting that~\prettyref{eq:gmm_model1} is a smooth parametric family with $k(d+1)-1$ parameters. But justifying the rate~\prettyref{eq:result_density} is not trivial, especially in high dimensions. 
%
%To justify~\prettyref{eq:result_density}, we use the Le Cam-Birg\'e construction of estimators from pairwise tests, which relies on bounding the \emph{local Hellinger covering number}, i.e., the minimum number of Hellinger balls of radius $\delta$ that cover the Hellinger ball of radius $\epsilon$ centered at the true density. The celebrated result of Le Cam-Birg\'e \cite{LeCam73,Birge83,birge1986estimating} shows that if the local entropy is at most $(\frac{\epsilon}{\delta})^{O(D)}$, then there exists a density estimate that achieves the squared Hellinger risk of $O(\frac{D}{n})$. The crucial parameter $D$ is known as the \emph{doubling dimension} (or the \emph{Le Cam dimension}~\cite{vdv02}), which is the effective number of parameters. In order to apply the theory of Le Cam-Birg\'e, we need to show that the doubling dimension of Gaussian mixtures is at most $O_k(d)$.}

%To explain the role of moment tensors in characterizing the information geometry of Gaussian mixtures, 

Bounding the local entropy requires a sharp characterization of the information geometry of Gaussian mixtures, for which the \emph{moment tensors} play a crucial role. To explain this, we begin with an abstract setting: 
Consider a parametric model $\{P_\theta:\theta\in\Theta\}$, where the parameter space $\Theta$ is a subset of the $D$-dimensional Euclidean space. We say a parameterization is \emph{good} if the Hellinger distance satisfies the following \emph{dimension-free} bound:
%Suppose that the Hellinger distance of this model satisfies the following sandwich bound:
\begin{equation}
C_0 \|\theta-\theta'\| \leq H(P_\theta,P_{\theta'}) \leq C_1 \|\theta-\theta'\|,
\label{eq:good-param}
\end{equation}
for some norm $\|\cdot\|$ and constants $C_0,C_1$.
%We refer to \prettyref{eq:good-param} as a good parametrization. 
The two-sided bound \prettyref{eq:good-param} leads to the desired result on the local entropy in the following way.
First, given any $P_\theta$ in an $\epsilon$-Hellinger neighborhood of the true density $P_{\theta_*}$, the lower bound in \prettyref{eq:good-param} localizes the parameter $\theta$ in an $O(\epsilon)$-neighborhood (in $\|\cdot\|$-norm) of the true parameter $\theta_*$, which, thanks to the finite dimensionality, can be covered by at most $(\frac{\epsilon}{\delta})^{O(D)}$ $\delta$-balls. Then the upper bound in \prettyref{eq:good-param} shows that this covering constitutes an $O(\delta)$-covering for the Hellinger ball.
%In this argument we see the crucial role of \emph{both} directions of \prettyref{eq:good-param}.

While satisfied by many parametric families, notably the Gaussian location model, 
\prettyref{eq:good-param} fails for their mixtures if we adopt the natural parametrization (in terms of the centers and weights), as shown by the simple counterexample of the
%Unfortunately, for mixture models, \prettyref{eq:good-param} fails for the natural parametrization (in terms of the centers and weights). 
%To see this, consider the simple example of 
symmetric 2-GM where $P_\theta=\frac{1}{2} N(-\theta,1)+\frac{1}{2} N(\theta,1)$, with $|\theta|\leq 1$.
 %where $\theta\in\Theta=[-1,1]$.
%; this is arguably the simplest non-trivial mixture model. 
Indeed, it is easy to show that \cite{Wu_Zhou_2019}: 
\[
|\theta-\theta'|^2 \lesssim H(P_\theta,P_{\theta'}) \lesssim |\theta-\theta'|,
\]
which is tight since the lower and upper bound are achieved when $\theta'\to\theta$ and for $\theta=0$ and say $\theta=0.1$, respectively.
The behavior of the lower bound can be attributed to the zero Fisher information at $\theta=0$. 
The importance of a two-sided comparison result like \prettyref{eq:good-param} and the difficulty in Gaussian mixtures
were recognized by \cite{gassiat2014local,gassiat2012consistent} in their study of the local entropy of mixture models. See \prettyref{sec:density_estimation} for detailed discussion.

It turns out that for Gaussian mixture model \prettyref{eq:gmm_model1}, a good parametrization satisfying~\prettyref{eq:good-param} is provided by the moment tensors. The degree-$\ell$ moment tensor of the mixing distribution $\Gamma$ is the symmetric tensor 
\begin{equation}
M_\ell(\Gamma) \triangleq \Expect_{U\sim\Gamma}[U^{\otimes \ell}] = \sum_{j=1}^k w_j \mu_j^{\otimes \ell}.
\label{eq:momenttensor-intro}
\end{equation} 
It can be shown that any $k$-atomic distribution is uniquely determined by its first $2k-1$ moment tensors $\bfM_{2k-1}(\Gamma)=[M_1(\Gamma),\ldots,M_{2k-1}(\Gamma)]$.
Consequently, moment tensors provides a valid parametrization of the $k$-GM in the sense that $\bfM_{2k-1}(\Gamma) = \bfM_{2k-1}(\Gamma')$ if and only if $P_\Gamma=P_{\Gamma'}$. At the heart of our proof of \prettyref{thm:density_est}
is the following robust version of this identifiability result:
 %by showing that the Hellinger distance of the $k$-GM satisfies
\begin{equation}
\label{eq:H-M-intro}
H^2(P_\Gamma, P_{\Gamma'}) \asymp_k \Fnorm{\bfM_{2k-1}(\Gamma)-\bfM_{2k-1}(\Gamma')}^2
%\max_{\ell \le 2k-1} \Fnorm{M_{\ell}(\Gamma)-M_{\ell}(\Gamma')}^2.
\end{equation}
which shows that the Hellinger distance between $k$-GMs are characterized by the Euclidean distance of their moment tensors up to dimension-free constant factors. Furthermore, the same result also holds for the Kullback-Leibler (KL) and the $\chi^2$ divergences. See \prettyref{sec:moment-tensor} for details.

Note that moment tensors appear to be a gross overparameterization of $\mc G_{k, d}$ since the original number of parameters is only $k d+ k-1$ as compared to the size $d^{\Theta(k)}$ of moment tensors. The key observation is that the moment tensors \prettyref{eq:momenttensor-intro} for $k$-atomic distributions are naturally low rank, so that the effective dimension remains $\Theta(kd)$. This observation underlies tensor decomposition methods for learning mixture models \cite{Anandkumar_etal_2012,Hsu_Kakade_2013}; here we use it for the information-theoretic purpose of bounding the local metric entropy of Gaussian mixtures. 

Results similar to~\prettyref{eq:H-M-intro} were previously shown in~\cite{bandeira2017optimal} for the problem of multiple-reference alignment, a special case of Gaussian mixtures with mixing distribution being uniform over the cyclic shifts of a given vector. The crucial difference is that the characterization~\prettyref{eq:H-M-intro} involves moments tensors of degree at most $2k-1$, while~\cite[Theorem 9]{bandeira2017optimal} involves all moments.

\added{The Le Cam-Birg\'e construction used to show \prettyref{thm:density_est} does not result in a computationally efficient estimator. In~\prettyref{ssec:density_est_efficient}, we provide a variant of the algorithm in~\prettyref{sec:mixing_estimation} that runs in $n^{O(k)}$ time which achieves the suboptimal rate of $O_k((d/n)^{1/4})$ for general $k$-GM (\prettyref{thm:density_est_efficient}). Though not optimal, this result nonetheless improves the state of the art of~\cite{Acharya_etal_2014} by logarithmic factors.
Furthermore, in the special case of $2$-GM, a slightly modified estimator is shown to achieve the optimal rate of $O(\sqrt{d/n})$ (\prettyref{thm:rate_density_2gm}). 
%(see Theorems \ref{thm:density_est_efficient} and \ref{thm:rate_density_2gm}). 
Finding a polynomial-time algorithm achieving the optimal rate in~\prettyref{thm:density_est} for all $k$ is an open problem.
}

\subsection{Related work}
\label{sec:related}

There is a vast literature on Gaussian mixtures; see~\cite{Lindsay_1989,HK2015,WY18} and the references therein for an overview. In one dimension, fast algorithms and optimal rates of convergence have already been achieved for both parameter and density estimation by, e.g.,~\cite{WY18}. We therefore focus the following discussion on multivariate Gaussian mixtures, in both low and high dimensions. 

\paragraph{Parameter estimation}

For statistical rates,~\cite[Theorem 1.1]{Ho_Nguyen_2016_annals} \added{and~\cite[Theorem $4.3$]{Ho_Nguyen_2016_ejs}} obtained convergence rates for mixing distribution estimation in Wasserstein distances for low-dimensional location-scale Gaussian mixtures, both over- and exact-fitted. Their rates for over-fitted mixtures are determined by algebraic dependencies among a set of polynomial equations whose order depends on the level of overfitting \added{and identifiability of the model}; the rates are potentially much slower than $n^{-1/2}$. The estimator analyzed in~\cite{Ho_Nguyen_2016_annals,Ho_Nguyen_2016_ejs} is the MLE, which involves non-convex optimization and is typically approximated by the Expectation-Maximization (EM) algorithm.

In the computer science literature, a long line of research starting with~\cite{Dasgupta_1999} has developed fast algorithms for individual parameter estimation in multivariate Gaussian mixtures under fairly weak separation conditions, see, e.g.,~\cite{Vempala_Wang_2004,Arora_Kannan_2005,Belkin_Sinha_2009,Kalai_Moitra_Valiant_2010,Moitra_Valiant_2010,Hsu_Kakade_2013,Hardt_Price_2015,Hopkins_Li_2018}. Since these works focus on individual parameter estimation, some separation assumption on the mixing distribution is necessary.

\paragraph{Density estimation}

Computational issues aside, there are several recent works addressing the minimax rate of density estimation for Gaussian mixtures. 
%The Hellinger MLE result in Ho and Nguyen annals is actually part of the proof of Theorem $1.1$. 
%In low dimensions,~\cite[Theorem $1.1$]{Ho_Nguyen_2016_annals} \added{and~\cite[Theorem $4.3$]{Ho_Nguyen_2016_ejs}} obtained an $O(\sqrt{\log n/n})$-Hellinger guarantee for the MLE for Gaussian location-scale and location mixtures, respectively. 
In low dimensions, an $O(\sqrt{\log n/n})$-Hellinger guarantee for the MLE is obtained for finite Gaussian mixtures \cite{Ho_Nguyen_2016_annals,Ho_Nguyen_2016_ejs}. 
The near-optimal rate for high-dimensional location-scale mixtures was obtained recently in~\cite{Ashtiani_etal_2018}. This work also provides a total variation guarantee of $\tilde O(\sqrt{kd/n})$ for location mixtures, where $\tilde O$ hides polylogarithmic factors, as compared to the sharp result in~\prettyref{thm:density_est}. 
%This result is the best rate for a high-dimensional Gaussian location mixture density estimation until the sharp result found in \prettyref{thm:density_est}. 
The algorithm in~\cite{Ashtiani_etal_2018} runs in time that is exponential in $d$.

To our knowledge, there is no polynomial-time algorithm that achieves the sharp density estimation guarantee in~\prettyref{thm:density_est} (or the slightly suboptimal rate in \cite{Ashtiani_etal_2018}), even for constant $k$. The works of~\cite{Kalai_Moitra_Valiant_2010,Moitra_Valiant_2010} showed that their polynomial-time parameter estimation algorithms also provide density estimators without separation conditions, but the resulting rates of convergence are far from optimal.~\cite{Feldman_etal_2006,Acharya_etal_2014,Li_Schmidt_2017} provided polynomial-time algorithms for density estimation with improved statistical performance. 
In particular, \cite{Acharya_etal_2014} obtained an algorithm that runs in time $\tilde O_k(n^2d + d^2(n/d)^{3k^2/4})$
%\replaced{$\tilde O_k(n^2d + d^2(n/d)^{3k^2/4})$}{$\tilde O(n^2d + (d/n)^{k^2})$} 
and achieves a total variation error of $\tilde O((d/n)^{1/4})$. 
%(Here and below $\tilde O(\cdot)$ hides logarithmic factors.) % Theorem 11 of NeurIPS paper. 
% \nb{This is suspicious: Since the result is in TV, it is trivial if $d>n$, so $d\leq n$. But what does $(d/n)^{k^2}$ mean then...} \added{Fixed now. Also note: this is my translation of their Theorem 8, but note that they typically simplify it to be $\tilde O_{k, \epsilon}(d^3)$, where the $\epsilon$ depends on $n$.} 
The running time was further improved in \cite{Li_Schmidt_2017}, which achieves the rate $\tilde O((d/n)^{1/6})$ for $2$-GM.

\paragraph{Nonparametric mixtures}

The above-mentioned works all focus on finite mixtures, which is also the scenario considered in this paper. A related strain of research (e.g.,~\cite{Genovese_Wasserman_2000,Ghosal_VdV_2001,Zhang_2009,Saha_2017})
 studies the so-called \emph{nonparametric mixture model}, in which the mixing distribution $\Gamma$ may be an arbitrary probability measure.

In this case, the nonparametric maximum likelihood estimator (known as the NPMLE) entails solving a convex (but infinite-dimensional) optimization problem, which, in principle, can be solved by discretization \cite{koenker2014convex}. For statistical rates, it is known that in one dimension, the optimal $L_2$-rate for density estimation is $\Theta((\log n)^{1/4}/\sqrt{n})$  and the Hellinger rate is at least $\Omega(\sqrt{\log n/n})$ \cite{Ibragimov_2001,Kim_2014}, which shows that the parametric rate \prettyref{eq:result_density} is only achievable for finite mixture models.
%which coincides with that of the class of analytic densities \cite{Ibragimov_2001}.
%which coincides with that of the class of analytic densities \cite{Ibragimov_2001}.
For the NPMLE, \cite{Zhang_2009} proved the Hellinger rate of $O(\log n/\sqrt{n})$ in one dimension; this was extended to the multivariate case by \cite{Saha_2017}. 
In particular,~\cite[Theorem 2.3]{Saha_2017} obtained a Hellinger rate of 
$ C_d \sqrt{k (\log n)^{d+1}/n}$ for the NPMLE when the true model is a $k$-GM. 
In high dimensions, this is highly suboptimal compared to the parametric rate in \prettyref{eq:result_density}, although the dependency on $k$ is optimal.

\subsection{Organization}
\label{sec:org}

The rest of the paper is organized as follows. \prettyref{sec:mixing_estimation} presents an efficient algorithm for estimating the mixing distribution and provides the theoretical justification for~\prettyref{thm:mixing_est}. \prettyref{sec:density_estimation} introduces the necessary background on moment tensors and proves the optimal rate of density estimation in \prettyref{thm:density_est}. \prettyref{sec:numerical} provides simulations that support the theoretical results. \prettyref{sec:discussion} provides further discussion on the connections between this work and the Gaussian mixture literature.

\section{Notation} \label{sec:notation}

%We now introduce some notation that will be used throughout the paper. 
Let $[n]\triangleq \{1, \hdots, n\}$. 
Let $S^{d-1}$ and $\Delta^{d-1}$ denote the unit sphere and the probability simplex in $\reals^d$, respectively.
Let $e_j$ be the vector with a $1$ in the $j$th coordinate and zeros elsewhere. 
For a matrix $A$, let $\norm{A}_2 = \sup_{x: \norm{x}_2 = 1} \norm{Ax}_2$ and $\norm{A}_F^2 = \trace(A^{\top}A)$. 
For two positive sequences $\{a_{n}\}, \{b_{n}\}$, we write $a_n \lesssim b_n$ or $a_{n} = O(b_{n})$ if there exists a constant $C$ such that $a_n\le C b_n$ and we write $a_{n} \lesssim_k b_{n}$ and $a_{n} = O_k(b_{n})$ to emphasize that $C$ may depend on a parameter $k$. 
%\added{Since we use $\Gamma$ mixing distributions, we denote the Gamma function by $\mathbf{\Gamma}(t) \triangleq \int_0^\infty x^{t-1} e^{-x} \, dx$ for positive real $t$.} \nbwu{Can you please change $G(t)$ to be $\mathbf{\Gamma}(t)$?}

For $\epsilon > 0$, an $\epsilon$-covering of a set $A$ with respect to a metric $\rho$ is a set $\calN$ such that for all $a \in A$, there exists $b \in \calN$ such that $\rho(a, b) \le \epsilon$; denote by $N(\epsilon, A, \rho)$ the minimum cardinality of $\epsilon$-covering sets of $A$.
An $\epsilon$-packing in $A$ with respect to the metric $\rho$ 
is a set $\mc M \subset A$ such that $\rho(a,b) > \epsilon$ for any distinct $a,b$ in $\calM$; denote by $M(\epsilon, A, \rho)$ the largest cardinality of $\epsilon$-packing sets in $A$.

For distributions $P$ and $Q$, let $p$ and $q$ denote their relative densities with respect to some dominating measure $\mu$, respectively. The total variation distance is defined as $\TV(P, Q) \triangleq  \frac{1}{2} \int |p(x) - q(x) | \mu(dx)$.
If $P\ll Q$, the KL divergence and the $\chi^2$-divergence are defined as 
$\KL\(P || Q \) \triangleq  \int p(x) \log \frac{p(x)}{q(x)} \mu(dx)$ and 
$\chi^2\(P\| Q\) \triangleq  \int \frac{(p(x) - q(x))^2}{q(x)} \mu(dx)$, respectively.
Let $\supp(P)$ denote the support set of a distribution $P$.
Let $\calL(U)$ denote the distribution of a random variable $U$.
For a one-dimensional distribution $\nu$, denote the $r$th moment of $\nu$ by $m_r(\nu) \triangleq \mb E_{U\sim \nu}[U^r]$
and its moment vector ${\bf m}_{r}(\nu) \triangleq (m_1(\nu),\ldots,m_{r}(\nu))$. Given a $d$-dimensional distribution $\Gamma$, for each $\theta\in\reals^d$, we denote 
\begin{equation}
\Gamma_\theta \triangleq \mc L(\theta^\top U), \quad U \sim \Gamma; 
\label{eq:gamma_theta}
\end{equation}
in other words, $\Gamma_\theta$ is the pushforward of $\Gamma$ by the projection $u\mapsto\theta^\top u$; 
in particular, 
the $i$th marginal of $\Gamma$ is denoted by $\Gamma_i \triangleq \Gamma_{e_i}$, with $e_i$ being the $i$th coordinate vector.
 Similarly, for $V\in\reals^{d\times k}$, denote 
 %we use $\Gamma_V$ to indicate the distribution $\mc L(V^\top U)$ where $U \sim \Gamma$.
\begin{equation}
\Gamma_V \triangleq \mc L(V^\top U), \quad U \sim \Gamma. 
\label{eq:gamma_V}
\end{equation}

\section{Mixing distribution estimation} \label{sec:mixing_estimation}

In this section we present the algorithm that achieves the optimal rate for estimating the mixing distribution in \prettyref{thm:mixing_est}.
The procedure is described in Sections \ref{sec:dtok} and \ref{sec:kto1}. The proof of correctness is given in Sections \ref{ssec:proof_main}, with supporting lemmas proved in \prettyref{app:proofs_supporting}.
Throughout this section we assume that $n \geq d$.

\subsection{Dimension reduction via PCA}
	\label{sec:dtok}
\added{
In this section we assume $d> k$ and reduces the dimension from $d$ to $k$.
For $d\le k$, we will directly apply the procedure in \prettyref{sec:kto1}.
}
%We start by introducing some helpful notation. 
Recall that the atoms $\Gamma$ are $\mu_1, \hdots, \mu_k$; they span a subspace of $\mb R^d$ of dimension at most $k$. 
Therefore, there exists $V=[v_1,\dots,v_k]$ consisting of orthonormal columns, such that 
for each $j = 1, \hdots, k$, we have
$\mu_j = V\psi_j$, where $\psi_j=V^\top \mu_j \in \reals^k$ encodes the coefficients of $\mu_j$ in the basis vectors in $V$. 
Therefore, we can identify a $k$-atomic distribution $\Gamma$ on $\reals^d$ 
with a pair $\(V, \gamma \)$, where $\gamma = \sum_{j \in [k]} w_j \delta_{\psi_j}$ is a $k$-atomic distribution on $\mb R^{k}$. 
This perspective motivates the following two-step procedure.
First, we estimate the subspace $V$ using PCA, relying on the fact that the covariance matrix satisfies $\Expect[XX^\top] = I_d + \sum_{j=1}^k w_j \mu_j\mu_j^\top$.
We then project the data onto the estimated subspace, reducing the dimension from $d$ to $k$, and apply an estimator of $k$-GM in $k$ dimensions. 
The precise execution of this idea is described below.

%An optimal procedure is to first project the data onto a subspace that is ``close enough'' to the space spanned by the $v_j$, then to estimate the mixing distribution $\gamma$ using the low-dimensional data. We would like to project the data onto the subspace spanned by the columns of $V$, but since we cannot access it, we use the noisy version. 

For simplicity, consider a sample of $2n$ observations $X_1,\ldots,X_{2n} \iiddistr P_\Gamma$.
%The procedure for estimating $\Gamma$ as follows,
We construct an estimator $\hat \Gamma$ of $\Gamma$ in the following way:
\begin{enumerate}[(a)]
	\item Estimate the subspace $V$ using the first half of the sample. 
	%Given $\{ X_1, \ldots,X_n \}$, let 
	%\begin{equation}
		%\hat \Sigma = \frac{1}{n} \sum_{i=1}^n X_iX_i^\top - I_d.
	%\label{eq:hatSigma}
	%\end{equation}
	Let $\hat V=[\hat v_1,\dots,\hat v_k]\in\reals^{d\times k}$ be the matrix whose columns are the top $k$ orthonormal left singular vector of $[X_1,\ldots,X_n]$.
	%Given $\{ X_1, \ldots,X_n \}$, let 
	%\begin{equation}
		%\hat \Sigma = \frac{1}{n} \sum_{i=1}^n X_iX_i^\top - I_d.
	%\label{eq:hatSigma}
	%\end{equation}
	%Let $\hat V=[\hat v_1,\dots,\hat v_k]\in\reals^{d\times k}$ be the matrix whose columns are the top $k$ orthonormal eigenvectors of $\hat \Sigma$.
	
\item Project the second half of the sample onto the learned subspace $\hat V$: 
\begin{equation}
x_i \triangleq \hat V^{\top} X_{i+n}, \quad i =1,\ldots,n. 
\label{eq:xi}
\end{equation}
Thanks to independence, conditioned on $\hat V$, $x_1,\ldots,x_n$ are \iid~observations  from a $k$-GM in $k$ dimensions, with mixing distribution
\begin{equation}
\gamma \triangleq \Gamma_{\hat V} = \sum_{j=1}^k w_j \delta_{\hat V^\top \mu_j}
\label{eq:gamma}
\end{equation}
obtained by projecting the original $d$-dimensional mixing distribution $\Gamma$ onto $\hat V$.

\item To estimate $\hat \gamma$, we apply a multivariate version of the denoised method of moments to $x_1,\ldots,x_n$ to obtain a $k$-atomic distribution on $\reals^k$:
\begin{equation}
\hat \gamma = \sum_{j = 1}^k \hat w_j \delta_{\hat \psi_j}.
\label{eq:hat-gamma}
\end{equation}
This procedure is explained next and detailed in \prettyref{algo:gmm_est}.

\item Lastly, we report
\begin{equation}
\hat \Gamma = \hat\gamma_{\hat V^\top} = \sum_{j = 1}^k \hat w_j \delta_{\hat V \hat \psi_j} \label{eq:estimator}
\end{equation}
as the final estimate of $\Gamma$.
\end{enumerate}

%\added{We now briefly discuss possible ways to improve on the dimension reduction step above. The independence of $\hat V$ and $\{X_{i+n} \}_{i \in [n]}$ in~\prettyref{eq:xi} is due to sample splitting, which we have presented here as dividing the full sample of $2n$ points into two equal parts. In fact, the sample sizes do not need to be equal and can be optimized with respect to $d$ and $k$ to balance the two terms of error, depending on which step is harder. But the improvement in the final estimation rate from optimizing the sample splitting is at most a constant factor. Since the main focus of this paper is on the theoretical minimax rate within constant factors, we choose equal splits for the ease of presentation. } 

A slightly better dimension reduction can be achieved by first centering the data by subtracting the sample mean, then projecting to a subspace of dimension $k-1$ rather than $k$, and finally  adding back the sample mean after obtaining the final estimator. As this only affect constant factors, we forgo the centering step in this section. 
 %To simplify the presentation in this section. 
Later in \prettyref{ssec:density_est_efficient}, it turns out that centering is important for achieving the optimal density estimation for $2$-GM (see \prettyref{thm:rate_density_2gm}).

The usefulness of dimension reduction has long been recognized in the literature of mixture models \cite{Dasgupta_1999,Vempala_Wang_2004,Kannan_Salmasian_2005,Achlioptas_McSherry,HP15,Loffler_etal_2019}, where the mixture data is projected to a good low-dimensional subspace (by either random projection or spectral methods) and parameter estimation or clustering are carried out subsequently in the low-dimensional mixture model. For such methods, the error bound typically depends on those of these two steps, 
%the subspace estimation error and the error of the clustering method performed on the low-dimensional data, 
analogously to the analysis of mixing distribution estimation in~\prettyref{thm:mixing_est}. 
%In particular, our error bound for the projection part (\prettyref{lem:perturbation_psd}) is closely connected to~\cite[Theorem $3$]{Vempala_Wang_2004}.

%\nbwu{For reference, below are original:
%
%\added{The idea of reducing mixture data to a good low-dimensional subspace and then performing an algorithm on the low-dimensional data is the basis of much of the spectral clustering literature \cite{Dasgupta_1999,Vempala_Wang_2004,Kannan_Salmasian_2005,Achlioptas_McSherry,HP15,Loffler_etal_2019}. For such methods, the error depends on the error in the subspace estimation and the error of the clustering method performed on the low-dimensional data, analogously to the analysis of mixing distribution estimation in~\prettyref{thm:mixing_est}. Our statement on the error from the projection,~\prettyref{lem:perturbation_psd}, is closely connected to~\cite[Theorem $3$]{Vempala_Wang_2004}.}
%}

%Note that~\cite{Vempala_Wang_2004} use as their best subspace the top singular values of the data matrix $X$ directly, which is slightly different from our approach.} Specifically, they use the top $r$ singular values, where $r \ge C \max(k, \log (n/n_{\min})$). It is important to note that since the above-mentioned works focus on clustering, they make stronger separation assumptions than we do, and thus their subspace estimation error rates may be faster.  But actually, ~\cite[Theorem 3]{Vempala_Wang_2004} translates to an individual location estimation error rate of $\(\frac{d^3}{n}\)^{1/4}$, which is much worse than ours. This is because their analysis isn't sharp.

\subsection{Estimating the mixing distribution in low dimensions}
\label{sec:kto1}
We now explain how we estimate a $k$-GM in $k$ dimensions from \iid~observations. 
%Recall that the optimal $W_1$-rate for $k$-atomic one-dimensional mixing distribution is $O(n^{-\frac{1}{2k-1}})$. Our goal is to extend this result to $k$ dimensions.
As mentioned in \prettyref{sec:intro}, the idea is to use many projections to reduce the problem to one dimension. We first present a conceptually simple estimator $\hat\gamma^{\circ}$ with an optimal statistical performance but unfavorable run time $n^{O(k)}$. We then describe an improved estimator $\hat\gamma$ that retains the statistical optimality and can be executed in time $n^{O(1)}$. \added{These procedures are also applicable to estimating a $k$-GM in $d< k$ dimensions using fewer projections.}

To make precise the reduction to one dimension, a relevant metric is the \emph{sliced Wasserstein distance} \added{~\cite{Paty_Cuturi_2019,Deshpande_etal_2019,Niles-Weed_Rigollet_2019}}, which measures the distance of two $d$-dimensional distributions by the maximal $W_1$-distance of their one-dimensional projections:
\begin{equation}
\Wsliced_1(\Gamma, \Gamma') \triangleq \sup_{\theta \in S^{d-1}} W_1(\Gamma_\theta, \Gamma'_\theta).
\label{eq:W1sliced}
\end{equation}
Here we recall that $\Gamma_\theta$ defined in \prettyref{eq:gamma_theta} denotes the projection, or pushforward, of $\Gamma$ onto the direction $\theta$. 
\added{A related definition was introduced earlier by~\cite{rabin2011wasserstein}, where the supremum over $\theta$ in~\prettyref{eq:W1sliced} is replaced by the average.} 
%\added{Several related definitions exist in the literature. For example, a similar notion was introduced earlier by~\cite{rabin2011wasserstein}, where the supremum over $\theta$ in~\prettyref{eq:W1sliced} is replaced by the average; a min-max type quantity (as opposed to $\Wsliced_1$ which is a max-min) is studied in \cite{Paty_Cuturi_2019}.} 
Computing the sliced Wasserstein distance can be difficult and in practice is handled by gradient descent heuristics \cite{Deshpande_etal_2019}; 
we will, however, only rely on its theoretical properties.
The following result, which is proved in \prettyref{app:proofs_supporting}, shows that for low-dimensional distributions with few atoms, the full Wasserstein distance and the sliced one are comparable up to constant factors. \added{Related results are obtained in \cite{Paty_Cuturi_2019,Bayraktar_Guo_2019}. For instance, \cite[Theorem 2.1(ii)]{Bayraktar_Guo_2019} showed that 
$W_1(\Gamma, \Gamma') \le C_d \cdot \Wsliced_1(\Gamma, \Gamma')$ holds for all distributions $\Gamma,\Gamma'$ for some non-explicit constant $C_d$.}

%\added{To our knowledge,~\prettyref{lem:w1_dim_reduction} is novel, but related characterizations of the connection between Wasserstein and sliced Wasserstein distances are obtained in~\cite[Theorem $2.1$]{Bayraktar_Guo_2019} and~\cite[Proposition $2$]{Paty_Cuturi_2019}.}
\begin{lem}[Sliced Wasserstein distance] \label{lem:w1_dim_reduction}
For any $k$-atomic distributions $\Gamma, \Gamma'$ on $\mb R^d$, 
\begin{align*}
%\sup_{\theta \in S^{d-1}} W_1(\Gamma_\theta, \Gamma'_\theta) \le W_1(\Gamma, \Gamma') \le k^2 \sqrt{d} \sup_{\theta \in S^{d-1}} W_1(\Gamma_\theta, \Gamma'_\theta).
\Wsliced_1(\Gamma, \Gamma') \le W_1(\Gamma, \Gamma') \le k^2 \sqrt{d} \cdot \Wsliced_1(\Gamma, \Gamma').
\end{align*}
\end{lem}

%Since we have used PCA to reduced the dimension from $d$ to $k$ and 
Having obtained via PCA the reduced sample $x_1,\ldots,x_n \sim \gamma * N(0,I_k)$ in \prettyref{eq:xi}, \prettyref{lem:w1_dim_reduction} suggests the following ``meta-procedure'':
%The main point of \prettyref{lem:w1_dim_reduction} is that for low-dimensional distributions with few atoms, the full Wasserstein distance is preserved by one-dimensional projection up to constant factors. This suggests the following ``meta-procedure'':
Suppose we have an algorithm (call it a 1-D algorithm) that estimates the mixing distribution based on $n$ \iid~observations drawn from a $k$-GM in one dimension.
Then
\begin{enumerate}
	\item For each $\theta \in S^{k-1}$, since 
	$\iprod{\theta}{x_i} \iiddistr \gamma_\theta *N(0,1)$, we can apply the 1-D algorithm to obtain an estimate $\widehat{\gamma_\theta} \in \calG_{k,1}$;
	
	\item We obtain an estimate of the multivariate distribution by minimizing a proxy of the sliced Wasserstein distance: 
	%\nb{previously you had, "by minimizing a proxy of the sliced Wasserstein distance." I assume you meant that in practice we minimize $\max_{\theta \in \mc N} W_1(\gamma_\theta', \widehat{\gamma_\theta})$, but since here we are just writing a summary and using the real sliced $W_1$, I got rid of the "a proxy of" part. }
	%let $\hat\gamma^{\circ }$ be an $\epsilon$-minimizer of 
	%Let $\hat\gamma^{\circ }$ be an minimizer of 
	\begin{equation}
\hat\gamma^{\circ }=	\argmin_{\gamma'\in\calG_{k,k}} \sup_{\theta\in S^{k-1}}W_1(\gamma'_\theta, \widehat{\gamma_\theta}).
%\min_{\gamma'\in\calG_{k,k}} \sup_{\theta\in S^{k-1}}W_1(\hat\gamma_\theta, \widehat{\gamma_\theta}).
	\label{eq:ideal-estimator}
	\end{equation}
	
\end{enumerate}
Then by \prettyref{lem:w1_dim_reduction} (with $d=k$) and the optimality of $\hat\gamma^{\circ}$, we have
\begin{align*}
W_1(\hat\gamma^{\circ},\gamma)
\lesssim_k & ~ \Wsliced_1(\hat\gamma^{\circ},\gamma) = \sup_{\theta \in S^{k-1}} W_1(\hat\gamma^{\circ}_\theta,\gamma_\theta) \\
\leq & ~ 	\sup_{\theta \in S^{k-1}} W_1(\widehat{\gamma_\theta},\gamma_\theta) + \sup_{\theta \in S^{k-1}} W_1(\widehat{\gamma_\theta},\hat\gamma^{\circ}_\theta) \\
{\leq} & ~ 	2 \sup_{\theta \in S^{k-1}} W_1(\widehat{\gamma_\theta},\gamma_\theta).\numberthis \label{eq:ideal-estimator-bound}
\end{align*}
Recall that the optimal $W_1$-rate for $k$-atomic one-dimensional mixing distribution is $O(n^{-\frac{1}{4k-2}})$. 
Suppose there is a 1-D algorithm that achieves the optimal rate \emph{simultaneously} for all projections, in the sense that
\begin{equation}
\Expect\qth{\sup_{\theta\in S^{k-1}} W_1(\widehat{\gamma_\theta},\gamma_\theta)} \lesssim_k n^{-\frac{1}{4k-2}}.
\label{eq:optimal-allprojection}
\end{equation}
This immediately implies the desired 
\begin{equation}
\Expect[W_1(\hat\gamma^{\circ},\gamma)] \lesssim_k n^{-\frac{1}{4k-2}}.
\label{eq:optima-kdim}
\end{equation}
%which concludes the optimality of $\hat \gamma^{\circ}$.
However, it is unclear how to solve the min-max problem in \prettyref{eq:ideal-estimator} where the feasible sets for $\gamma$ and $\theta$ are both non-convex. 
The remaining tasks are two-fold:
(a) provide a 1-D algorithm that achieves \prettyref{eq:optimal-allprojection};
(b) replace $\hat{\gamma}^{\circ}$ by a computationally feasible version.

%where
%(a) follows from \prettyref{lem:w1_dim_reduction} with $d=k$;
%(b) follows from the zeroth-order $\epsilon$-optimality.
%\nbwu{More specifically, let $\calN$ be an $\epsilon$-covering (in $\ell_2$) of $S^{k-1}$, and let $\calG$ be a $\epsilon$-covering (in $W_1$) of $\calG_{k,k}$. 
%Let $\hat \gamma=\argmin_{\gamma'\in\calG} \sup_{\theta\in \calN}W_1(\hat\gamma_\theta, \widehat{\gamma_\theta})$.
%Then 
%\begin{align*}
%W_1(\hat\gamma,\gamma)
%\lesssim_k & ~ \Wsliced_1(\hat\gamma,\gamma) = \sup_{\theta \in S^{k-1}} W_1(\hat\gamma_\theta,\gamma_\theta) \\
%\leq & ~\sup_{\theta \in \calN} W_1(\hat\gamma_\theta,\gamma_\theta) + 2 R \epsilon \\
%\leq & ~ 	\sup_{\theta \in \calN} W_1(\widehat{\gamma_\theta},\gamma_\theta) + \sup_{\theta \in \calN} W_1(\widehat{\gamma_\theta},\hat\gamma_\theta) \\
%\stepb{\leq} & ~ 	2 \sup_{\theta \in \calN} W_1(\widehat{\gamma_\theta},\gamma_\theta) + \epsilon,
%\end{align*}
%This concludes the statistical guarantee.
%However, the time complexity is not good: $|\calW| \leq (\frac{1}{\epsilon})^{Ck}$ and $|\calN| \leq (\frac{1}{\epsilon})^{Ck^2}$, so this gives an $n^{O(k)}$ result.
%}

\paragraph{Achieving \prettyref{eq:optimal-allprojection} by denoised method of moments}
In principle, any estimator for a one-dimensional mixing distribution with exponential concentration can be used as a black box; this achieves \prettyref{eq:optimal-allprojection} up to logarithmic factors by a standard covering and union bound argument. In order to attain the sharp rate in \prettyref{eq:optimal-allprojection}, we consider the Denoised Method of Moments (DMM) algorithm introduced in \cite{WY18}, which allows us to use the chaining technique to obtain a tight control of the fluctuation over the sphere (see \prettyref{lem:w1_sup_bound}).

DMM is an optimization-based approach that introduces a denoising step before solving the method of moments equations. 
For location mixtures, it provides an exact solver to the  
non-convex optimization problem arising in 
generalized method of moments~\cite{Hansen1982}.
For Gaussian location mixtures with unit variance, the DMM algorithm proceeds as follows:
\begin{enumerate}[(a)]
    \item 
		Given $Y_1,\ldots,Y_n \iiddistr \nu * N(0,1)$ for some $k$-atomic distribution $\nu$ supported on $[-R,R]$, we first estimate the moment vector  ${\bf m}_{2k-1}(\nu) = (m_1(\nu),\ldots,m_{2k-1}(\nu))$ by their unique unbiased estimator $\tilde m = (\tilde m_1,\dots,\tilde m_{2k-1})$,		
		where $\tilde{m}_r= \frac{1}{n} \sum_{i = 1}^n H_r(Y_i)$, and $H_r$ is the degree-$r$ Hermite polynomial defined via
		\begin{align}
H_r(x) \triangleq  r! \sum_{i = 0}^{\lfloor{r/2}\rfloor} \frac{(-1/2)^i}{i!(r - 2i)!} x^{r - 2i}. \label{eq:hermites_def}
\end{align}
Then $\Expect[\tilde m_r]=m_r(\nu)$ for all $r$. This step is common to all approaches based on the method of moments.
		
    \item 
		In general the unbiased estimate $\tilde m$ is not a valid moment vector, in which case the method-of-moment-equation lacks a meaningful solution. 
		The key idea of the DMM method is to denoise $\tilde m$ by its projection onto the space of moments:		
        \begin{equation}
            \label{eq:project}
\hat m    \triangleq        \argmin\{\Norm{\tilde m - m}: m\in \calM_r\},
        \end{equation} 
        where the moment space 
				\begin{equation}
\calM_r\triangleq\{\bfm_r(\pi):\pi \text{ supported on }[-R,R]\}
\label{eq:momentspace}
\end{equation}
consists of the first $r$ moments of all probability measures on $[-R,R]$.
The moment space is a convex set and characterized by positive semidefinite constraints (of the associated Hankel matrix); we refer the reader to the monograph \cite{ST1943} or \cite[Sec.~2.1]{WY18} for details. 
This means that the optimization problem~\prettyref{eq:project} can be solved efficiently as a semidefinite program (SDP); see~\cite[Algorithm 1]{WY18}.

    \item Use Gauss quadrature to find the unique $k$-atomic distribution $\hat \nu$ such that ${\bf m}_{2k-1}(\hat \nu)=\hat m$.
		We denote the final output $\hat\nu$ by $\DMM(Y_1,\ldots,Y_n)$. 
\end{enumerate}

The following result shows the DMM estimator achieves the optimal rate in \prettyref{eq:optimal-allprojection} simultaneously for all one-dimensional projections. (For a single $\theta$, this is shown in \cite[Theorem 1]{WY18}.) %\nb{But that concentration was slightly different, no? We also show this here, in the second statement in Lemma A.3, combined with the standard reduction of moments to Wasserstein.}
\begin{lem} \label{lem:w1_sup_bound}
For each $\theta \in S^{k-1}$, let $\widehat{\gamma_\theta} = \DMM(\Iprod{\theta}{x_1},\ldots,\Iprod{\theta}{x_n})$ where $x_1,\ldots,x_n \iiddistr \gamma * N(0,I_k)$ as in \prettyref{eq:xi}.
There is a positive constant \added{$C$} such that, for any $\delta\in(0,\frac{1}{2})$, with probability at least $1 - \delta$, 
\begin{align*}
\max_{\theta \in S^{k-1}} W_1(\widehat{\gamma_\theta}, \gamma_\theta) \le \added{C k^{7/2}} n^{-1/(4k-2)} \sqrt{\log \frac{1}{\delta}}.
%\( \frac{\(\log(1/\delta)\)^{2k-1}}{n}\)^{1/(4k-2)}.
\end{align*}
%\nbwu{Pengkun: There is a small problem in the proof. The statement as currently stated cannot hold for all $\delta\in(0,1)$ so it need to be updated, either assume $\delta \leq \delta_k$, or replace RHS by $\sqrt{\log \frac{C_k}{\delta}}$ (which is the same as saying $C_k + \sqrt{\log \frac{1}{\delta}})$ for instance. Please double check (a) how this propagates elsewhere (b) if there is similar problem in other lemma/theorems, e.g. \prettyref{lmm:DMM-kdim} -- essentially all results states in the format of ``w.p. $1-\delta$...''.}
%\nbpy{I agree it cannot hold for all $\delta<1$. 
%But I think it suffices to assume $\delta<\frac{1}{2}$ (or replace RHS by $\sqrt{\log \frac{2}{\delta}}$) to avoid further complication, since we have a factor $C_k$ in the front anyway.
%To be more precise, when $\delta<\frac{1}{2}$, we have $C_k(C_k'+\sqrt{\log \frac{1}{\delta}})\le C_kC_k'(1+\sqrt{\log \frac{1}{\delta}})\le 3C_kC_k'\sqrt{\log \frac{1}{\delta}}$.}
\end{lem}

\paragraph{Solving \prettyref{eq:ideal-estimator} efficiently using marginal estimates}
%\paragraph{Improving the Solving \prettyref{eq:ideal-estimator} efficiently}

We first note that in order to achieve the optimal rate in \prettyref{eq:optima-kdim}, 
it is sufficient to consider any approximate minimizer of \prettyref{eq:ideal-estimator} up to an additive error of $\epsilon$, as long as $\epsilon = O(n^{-\frac{1}{4k-2}})$.
%an $\epsilon$-net (in $W_1$) of $\calG_{k,k}$ and an $\epsilon$-net (in $\ell_2$) of the sphere, one can 
Therefore, to find an $\epsilon$-optimizer, it suffices to maximize over $\theta$ in an $\epsilon$-net (in $\ell_2$) of the sphere, which has cardinality $(\frac{1}{\epsilon})^k=n^{O(1)}$, and, likewise, minimize $\gamma$ over an $\epsilon$-net (in $W_1$) of $\calG_{k,k}$. The $W_1$-net can be constructed by combining an $\epsilon$-net (in $\ell_2$) for each of the $k$ centers and an $\epsilon$-net (in $\ell_1$) for the weights, resulting in a set of cardinality $(\frac{1}{\epsilon})^{O(k^2)}=n^{O(k)}$. This na\"ive discretization scheme leads to an estimator of $\gamma$ with optimal rate but time complexity $n^{O(k)}$. We next improve it to $n^{O(1)}$.

The key idea is to first estimate the marginals of $\gamma$, which narrows down its support set. It is clear that a $k$-atomic joint distribution is not determined by its marginal distributions, as shown by the example of $\frac{1}{2}\delta_{(-1,-1)}+\frac{1}{2}\delta_{(1,1)}$ and $\frac{1}{2}\delta_{(-1,1)}+\frac{1}{2}\delta_{(1,-1)}$, which have identical marginal distributions. Nevertheless, the support of the joint distribution must be a $k$-subset of the Cartesian product of the marginal support sets. This suggests that we can select the atoms from this Cartesian product and weights by fitting all one-dimensional projections, as in~\prettyref{eq:ideal-estimator}. 

Specifically, 
for each $j\in[k]$, we estimate the $j$th marginal distribution of $\gamma$ by $\widehat{\gamma_j}$, obtained by applying the DMM algorithm on the coordinate projections $\Iprod{e_j}{x_1}, \ldots, \Iprod{e_j}{x_n}$. Consider the Cartesian product of the support of each estimated marginal as the candidate set of atoms:
\[
%\calA\triangleq \{\psi:e_i^\top\psi\in \supp(\widehat{\gamma_i}), i=1,\dots,k\};
\calA\triangleq \supp(\widehat{\gamma_1}) \times \cdots \times \supp(\widehat{\gamma_k}).
\]
 Throughout this section, let 
\begin{equation}
\epsilon_{n, k}\triangleq n^{-\frac{1}{4k-2}},
\label{eq:gridsize-mixing}
\end{equation}
and fix an $(\epsilon_{n, k},\norm{\cdot}_2)$-covering $\calN$ for the unit sphere $S^{k-1}$
and an $(\epsilon_{n, k},\norm{\cdot}_1)$-covering $\calW$ for the probability simplex $\Delta^{k-1}$, 
such that\footnote{This is possible by, e.g.,  \cite[Prop.~2.1]{RV09} and \cite[Lemma A.4]{Ghosal_VdV_2001} for the sphere and probability simplex, respectively.}
\begin{equation}
\max\{|\calN|, |\calW|\} \lesssim \pth{\frac{C}{\epsilon_{n, k}}}^{k-1}.
\label{eq:netsize}
\end{equation}
Define the following set of candidate $k$-atomic distributions on $\mb R^{k}$:
\begin{equation}
\mc S \triangleq \sth{\sum_{j \in [k]} w_{j} \delta_{\psi_{j}}: (w_{1}, \hdots, w_{k})\in \calW, \psi_{j}\in \calA}.
\label{eq:candidates}
\end{equation}
Note that $\mc S$ is a random set which depends on the sample; furthermore, each $\psi_j \in \mc A$ has coordinates lying in $[-R, R]$ by virtue of the DMM algorithm.

The next lemma shows that with high probability there exists a good approximation of $\gamma$ in the set $\calS$.
 %which \prettyref{algo:gmm_est} aims to find.
%Thuswhich boils down to the statistical guarantee of \prettyref{algo:gmm_est} in $k$ dimensions. 
%\prettyref{algo:gmm_est} is a grid search over discretized mixing distributions.
%\prettyref{lem:existence_solution} guarantees that with high probability there exists a good solution in the candidate set $\calS$ formed in \prettyref{eq:candidates}.

\begin{lem} \label{lem:existence_solution}
%Let $\gamma \in \mc G_{k, k}$. 
Let $\mc S$ be given in~\prettyref{eq:candidates}. \added{There is a positive constant $C$} such that, for any $\delta\in(0,\frac{1}{2})$, with probability $1 - \delta$,
\begin{align}
\min_{\gamma'\in\calS}W_1\(\gamma', \gamma \) \le 
% C_k \(\frac{\log(1/\delta)}{n}\)^{1/(4k-2)}. 
%C_k n^{-1/(4k-2)} \sqrt{\log \frac{1}{\delta}}.
\added{C  k^5} n^{-1/(4k-2)} \sqrt{\log \frac{1}{\delta}}.
\label{eq:existence_solution}
\end{align}
% \nbwu{Pengkun: This statement needs correction. See proof.}
\end{lem}

%To find a good approximation of $\gamma$ in $\calS$, in~\prettyref{algo:gmm_est} we assess each mixing distribution in $\calS$ by comparing its projection with $\widehat{\gamma_\theta}$ for all $\theta\in \calN$, and select one as the final estimator $\hat\gamma$. \prettyref{lem:w1_sup_bound} shows that with high probability, $\widehat{\gamma_\theta}$ are close to $\gamma_\theta$ uniformly for all $\theta$ on the sphere.

%In~\prettyref{algo:gmm_est}, we require the nets $\mc N(\epsilon_{n, k}, \Delta_k, \norm{\cdot}_1)$ and $\mc N(\epsilon_{n, k}, S^{k-1}, \norm{\cdot}_2)$ to be nets of size $(C/\epsilon_{n, k})^{k}$. We cannot use a net of exactly minimal size since this is not actually computable, but we cannot use a net of arbitrary size if we wish to obtain the time guarantee mentioned in Remark~\prettyref{rem:algorithm_time}. For more discussion on generating a small-enough net, see~\prettyref{sec:numerical}.

We conclude this subsection with \prettyref{algo:gmm_est}, which provides a full description of an estimator for $k$-atomic mixing distributions in $k$ dimensions.
The following result shows its optimality under the $W_1$ loss:
\begin{lemma}
\label{lmm:DMM-kdim}	
\added{There is a positive constant $C$} such that the following holds. Let $x_1,\ldots,x_n \iiddistr \gamma * N(0,I_k)$ for some $\gamma \in G_{k,k}$. Then \prettyref{algo:gmm_est} produces an estimator $\hat \gamma \in \calG_{k,k}$ such that, for any $\delta\in(0,\frac{1}{2})$, with probability $1-\delta$, 
\begin{equation}
W_1(\gamma,\hat\gamma) \leq \added{C k^5} n^{-1/(4k-2)} \sqrt{\log \frac{1}{\delta}}.
\label{eq:DMM-kdim}
\end{equation}
\end{lemma}

\begin{center}
\begin{algorithm}[ht] \label{algo:gmm_est}
        \SetAlgoLined
\caption{Parameter estimation for $k$-GM in $k$ dimensions}
\KwIn{
Dataset $\{ x_i \}_{i \in [n]}$ with each point in $\mb R^k$, order $k$, radius $R$. \\
}
\KwOut{
Estimate $\hat \gamma$ of $k$-atomic distribution in $k$ dimensions.
}
{\bf For $j =1,\ldots,k$:} \\
{
	%Do one-dimensional DMM on the dataset $\{ e_j^{\top}x_i \}_{i\in[n]}$  \;
	%Obtain one-dimensional distribution $\widehat{\gamma_j}$\;
~~~~Compute the marginal estimate $\widehat{\gamma_j}= \DMM(\{e_j^{\top}x_i\}_{i\in[n]})$  \;
}

%{\bf Form candidate distributions:} \\
{
	Form the set $\mc S$ of $k$-atomic candidate distributions on $\mb R^{k}$ as in~\prettyref{eq:candidates} \;

{\bf For each $\theta \in \mc N$:} \\
{
	%Compute the one-dimensional DMM estimate using the data $\{\theta^{\top}x_i\}_{i\in[n]}$  \;
	%Obtain one-dimensional $k$-atomic distribution $\widehat{\gamma_\theta}$ \;
	~~~~Estimate the projection by $\widehat{\gamma_\theta}= \DMM(\{\theta^{\top}x_i\}_{i\in[n]})$  \;
	%Compute the one-dimensional $k$-atomic distribution $\widehat{\gamma_\theta}$ \;
}
{\bf For each candidate distribution $\gamma'\in \calS$ and each direction $\theta \in \mc N$: } \\
{
	~~~~Compute $W_1(\gamma'_\theta, \widehat{\gamma_\theta})$  \;
	Report 
	\begin{equation}
	\hat \gamma = \arg\min_{\gamma'\in\calS}\max_{\theta\in \mc N}W_1(\gamma'_\theta, \widehat{\gamma_\theta}).
	\label{eq:hatgamma-kdim}
	\end{equation}	
}	
}
\end{algorithm} 
\end{center}

\begin{rem} \label{rem:algorithm_time}
The total time complexity to compute the estimator \prettyref{eq:estimator} is
$O(nd^2) + O_k(n^{5/4})$. Indeed, the time complexity of computing the sample covariance matrix is $O(nd^2)$, and the time complexity of performing the eigendecomposition is $O(d^3)$, which is dominated by $O(nd^2)$ since $d \le n$. By \prettyref{eq:netsize}, both $\calW$ and $\calN$ have cardinality at most $(C/\epsilon_{n, k})^{k-1}=O_k(n^{1/4})$.
Each one-dimensional DMM estimate takes $O_k(n)$ time to compute \cite[Theorem 1]{WY18}. Thus computing the one-dimensional estimator $\widehat{\gamma_\theta}$ for all $\theta = e_i$ and $\theta \in \calN$ takes time $O_k(n^{5/4})$. 
Since both $\gamma'_\theta$ and $\widehat{\gamma_\theta}$ are $k$-atomic distributions by definition, their $W_1$ distance can be computed in $O_k(1)$ time. Finally, $|\calA| = k^k$, \added{and to form $S$ we select all sets of $k$ atoms from $\calA$, so $|\mc S| \leq |\calW| {k^k \choose k} = O_k\(n^{1/4}\)$}. Thus searching over $\calS \times \calN$ takes time at most $O_k(n^{1/4}) * O_k(n^{1/4}) = O_k \(n^{1/2}\)$. Therefore, the overall time complexity of~\prettyref{algo:gmm_est} is $O_k(n^{5/4})$.
\end{rem}

\begin{forme}
Note that in general, $|\mc A| = {k^{ld} \choose k}$, where $ld$ is the latent dimension we project to. It's $k$ in the above, but e.g., if we centered first, it would be $k-1$. 
\begin{align*}
|\mc A| &\le k^k. \\
|\calW| &\le \( \frac{C}{\epsilon_{n, k}} \)^{k-1} = C^{k-1} * n^{(k-1)/(4k-2)} \le O_k(n^{1/4}). \\
|\mc S| &= {k^k \choose k} * |\mc W| \le {k^k \choose k} * C^{k-1} * O(n^{1/4}) = O_k(n^{1/4}). \\
|\mc N| &\le \( \frac{C}{\epsilon_{n, k}} \)^{k-1} = C^k * O(n^{1/4}) \le O_k(n^{1/4}).
\end{align*}
Now the DMM algorithm of~\cite{WY18} takes time $O(kn)$. So Algorithm $2$ takes:
\begin{align*}
\text{Step 1:  } & k * O(kn).  \\
\text{Step 2:  } & C^k * O(n^{1/4}) * O(kn) = k * C^k * O(n^{5/4}) = O_k(n^{5/4}). \\
\text{Step 3: } & |\mc S| * N(\epsilon_{n, k}, S^{k-1}, \norm{\cdot}_2) = {k^k \choose k} * C^{k-1} * O(n^{1/4}) * C^k * O(n^{1/4}) = C^{2k-1} * {k^k \choose k} * O(n^{1/2}).
\end{align*}

\begin{itemize}
	\item Step $1$: To compute the $\widehat{\gamma_i}$ for $\mc S_c$.
	\item Sept $2$: To compute the $\widehat{\gamma_\theta}$ for $\theta \in \calN$.
	\item Step $3$: To compute $W_1(\hat \gamma_\theta, \widehat{\gamma_\theta})$ for all $\hat \gamma \in \mc S$ and $\theta \in \calN$.
\end{itemize}
So the full time is:
\begin{align*}
k^2 * O(n) + k * C^k * O(n^{5/4}) + C^{2k-1} * {k^k \choose k} * O(n^{1/2}).
\end{align*}
This is dominated by the middle term.
\end{forme}

\subsection{Proof of \prettyref{thm:mixing_est}} \label{ssec:proof_main}

%We first provide the key supporting lemmas. Proofs of these results can be found in~\prettyref{app:proofs_supporting}. 
%We will use the notation from~\prettyref{algo:projection} and~\prettyref{algo:gmm_est} throughout the proofs. 

The proof is outlined as follows. 
Recall that the estimate $\hat\Gamma$ in \prettyref{eq:estimator} is supported on the subspace spanned by the columns of $\hat V$, whose projection is $\hat\gamma$ in \prettyref{eq:hat-gamma}. Similarly, the projection of the ground truth $\Gamma$ on the space $\hat V$ is denoted by $\gamma = \Gamma_{\hat V}$ in \prettyref{eq:gamma}.
Note that both $\gamma$ and $\hat\gamma$ are $k$-atomic distributions in $k$ dimensions.
Let $\hat H=\hat V\hat V^{\top}$ be the projection matrix onto the space spanned by the columns of $\hat V$.
By the triangle inequality, 
\begin{align}
W_1(\Gamma,\hat \Gamma)
\leq & ~ W_1(\Gamma,\Gamma_{\hat H})+W_1(\Gamma_{\hat H},\hat \Gamma)  \nonumber \\
\leq & ~ W_1(\Gamma,\Gamma_{\hat H})+W_1(\gamma,\hat\gamma).	\label{eq:w1-triangle}
\end{align}
We will upper bound the first term by $(d/n)^{1/4}$ (using Lemmas~\ref{lem:perturbation_psd} and \ref{lem:cov_bound} below) and the second term by $n^{-1/(4k-2)}$ (using the previous \prettyref{lmm:DMM-kdim}).
% Lemmas~\ref{lem:w1_sup_bound} and \ref{lem:existence_solution}).

%
%
%Thus the second term on the right-hand side of \prettyref{eq:w1-triangle} is at most
%\begin{equation}
%\label{eq:w1-proj}
%W_1(\Gamma_{\hat H},\hat \Gamma)\le W_1(\gamma,\hat\gamma),	
%\end{equation}
%where 
%\begin{equation}
%\gamma \triangleq \Gamma_{\hat V}
%\label{eq:gamma}
%\end{equation}
%is the projection of the original $d$-dimensional mixing distribution $\Gamma$ onto $\hat V$.
%Note that both $\gamma$ and $\hat\gamma$ are $k$-atomic distributions in $k$ dimensions.
%%Thus we need a guarantee for \prettyref{algo:gmm_est}

We first control the difference between $\Gamma$ and its projection onto the estimated subspace $\hat V$.
Since we do not impose any lower bound on $\norm{\mu_j}_2$, we cannot directly show the accuracy of $\hat V$ by means of perturbation bounds such as 
the Davis-Kahan theorem~\cite{Davis_Kahan_1970}. 
Instead, the following general lemma bounds the error by the difference of the covariance matrices.
For a related result, see \cite[Corollary 3]{Vempala_Wang_2004}.

%Instead,~\prettyref{lem:perturbation_psd} relates the atoms of $\Gamma$ to these same atoms projected by a noisy projection operator.
\begin{lem} \label{lem:perturbation_psd}
Let $\Gamma=\sum_{j=1}^k w_j\delta_{\mu_j}$ be a $k$-atomic distribution. 
Let $\Sigma = \Expect_{U\sim \Gamma}[UU^\top] = \sum_{j=1}^k w_j\mu_j\mu_j^\top$ with eigenvalues $\lambda_1\ge \dots \ge \lambda_d$. 
Let $\Sigma'$ be a symmetric matrix and \added{$\Pi_r'$} be the projection matrix onto the subspace spanned by the top $r$ eigenvectors of $\Sigma'$.
Then,
% for any $j \in [k]$, 
% \[
% w_j \Norm{\mu_j - H_r' \mu_j}_2^2
% \le \lambda_{r+1}+2\Norm{\Sigma-\Sigma'}_2.
% \]
% Consequently, 
% \nbwu{It seems the factor of $k$ is redundant? The proof bounds the sum over $j$ of the above. If so, please get rid of this and update other places accordingly.} \nb{I think $k$ is necessary because we have for each $j$, $w_j \norm{\mu_j - H \mu_j}_2^2 \le \lambda_{k + 1} + 2 \norm{\Sigma - \hat \Sigma}_2$, and we use $W_2^2 \le \sum_{j \in [k]} w_j \norm{\mu_j - H \mu_j}_2^2$.} 
\[
W_2^2(\Gamma, \Gamma_{\added{\Pi_r'}})\le k\(\lambda_{r+1}+2\Norm{\Sigma-\Sigma'}_2\).
\]
\end{lem}

We will apply \prettyref{lem:perturbation_psd} with $\Sigma'$ being the sample covariance matrix $\hat \Sigma$. 
The following lemma provides the concentration of $\hat \Sigma$ we need to prove the upper bound on the high-dimensional component of the error in~\prettyref{thm:mixing_est}.
\begin{lem} \label{lem:cov_bound}
Let $\Gamma \in \mc G_{k, d}$ and $\Sigma = \Expect_{U\sim \Gamma}[UU^\top]$.
Let $\hat \Sigma = \frac{1}{n}\sum_{i=1}^n X_iX_i^\top - I_d$, where $X_1, \hdots, X_n \iiddistr P_\Gamma$.
Then there exists a positive constant $C$ such that, with probability at least $1-\delta$,
\[
\Norm{\hat \Sigma - \Sigma}_2 \le C\(\sqrt{\frac{d}{n}} + k\sqrt{\frac{\log(k/\delta)}{n}}+\frac{\log(1/\delta)}{n}\).
\]
\end{lem}

\begin{proof} [\bf{Proof of~\prettyref{thm:mixing_est}}]
We first show that the estimator~\prettyref{eq:estimator} achieves the tail bound stated in \prettyref{eq:mixing_est_tail_bound}, which, after integration, implies the average risk bound in~\prettyref{eq:mixing_est_avg_risk}.
% , and show the average risk of \prettyref{eq:mixing_est_lb} in \prettyref{cor:tail_to_average}.
%Recall the triangle inequality in \prettyref{eq:w1-triangle}.
To bound the first term in \prettyref{eq:w1-triangle}, 
note that the rank of $\Sigma=\Expect_{U\sim\Gamma}[UU^\top]$ is at most $k$. 
Furthermore, 
the top $k$ left singular vectors of $[X_1,\ldots,X_n]$ coincide with 
the top $k$ eigenvectors of $\hat \Sigma = \frac{1}{n}\sum_{i=1}^n X_iX_i^\top - I_d$.
Applying Lemmas \ref{lem:perturbation_psd} and \ref{lem:cov_bound} yields that, with probability $1-\delta$,
\begin{equation}
\label{eq:w1-ub-high-d}
W_1(\Gamma,\Gamma_{\hat H})
\le \sqrt{2Ck} \(  \(\frac{d}{n}\)^{1/4}+\pth{\frac{k^2\log(k/\delta)}{n}}^{1/4}+\sqrt{\frac{\log(1/\delta)}{n}}\),
\end{equation}
where we used the fact that $W_1(\Gamma, \Gamma') \le W_2(\Gamma, \Gamma')$ by the Cauchy-Schwarz inequality. 
To upper bound the second term in \prettyref{eq:w1-triangle}, 
recall that $\hat V$ was obtained from $\{X_1,\ldots,X_{n}\}$ and hence is independent of $\{X_{n+1},\ldots,X_{2n}\}$.
Thus conditioned on $\hat V$, 
\[
x_i=\hat V^\top X_{i+n} \iiddistr \gamma * N(0,I_k), \quad i=1,\ldots,n.
\]
Let $\hat{\gamma}$ be obtained from \prettyref{algo:gmm_est} with input $x_1,\ldots,x_n$. 
By \prettyref{lmm:DMM-kdim}, with probability $1-\delta$, 
\begin{equation}
W_1(\gamma,\hat\gamma) \leq \added{Ck^5} n^{-1/(4k-2)} \sqrt{\log \frac{1}{\delta}}.
\label{eq:w1-ub-low-d}
\end{equation}
% \nbwu{Pengkun: double check the following sentence given the changes.}
Note that $(k^2\log(k/\delta)/n)^{1/4}+(\log(1/\delta)/n)^{1/2}$ in \prettyref{eq:w1-ub-high-d} is dominated by the right-hand side of \prettyref{eq:w1-ub-low-d}.
%$\epsilon= \added{C k^5} n^{-1/(4k-2)} \sqrt{\log \frac{1}{\delta}}$. 
The desired \prettyref{eq:mixing_est_tail_bound} follows from combining \prettyref{eq:w1-triangle}, 
\prettyref{eq:w1-ub-high-d}, and \prettyref{eq:w1-ub-low-d}.

%The average risk bound in~\prettyref{eq:mixing_est_avg_risk} immediately follows from~\prettyref{eq:mixing_est_tail_bound} by the fact that $\mb E Y = \int_0^\infty \mb P \{ Y > t \} \, dt$ for a non-negative random variable $Y$, and therefore
%\[
%\mb E W_1 (\hat \Gamma, \Gamma)
%\le C_k\( \(\frac{d}{n} \)^{1/4} + 
%%\int_0^\infty e^{-(nt^{4k-2})^{1/(2k-1)}} \, dt \)
%\int_0^\infty e^{-n^{1/(2k-1)} t^2} \, dt \)
%= C_k\( \(\frac{d}{n} \)^{1/4} + \frac{\sqrt{\pi}}{2}n^{-\frac{1}{4k-2}}\).
%\]
Finally, {we prove the lower bound in~\prettyref{eq:mixing_est_avg_risk}. For any subset $\mc G\subseteq\mc G_{k, d}$, we have}
\begin{align}
\label{eq:lb_split}
\inf_{\hat\Gamma}\sup_{\Gamma\in\mc G_{k,d}}\Expect W_1(\hat\Gamma,\Gamma)
\ge \inf_{\hat\Gamma}\sup_{\Gamma\in\mc G}\Expect W_1(\hat\Gamma,\Gamma)
\ge \frac{1}{2}\inf_{\hat\Gamma\in\mc G}\sup_{\Gamma\in\mc G}\Expect W_1(\hat\Gamma,\Gamma),
\end{align}
{where the second inequality follows from the triangle inequality for $W_1$.
%For any estimator $\hat\Gamma$ (not necessarily in $\mc G$), let $\tilde\Gamma = \arg\min_{\tilde\Gamma\in \mc G}W_1(\tilde\Gamma,\hat\Gamma)$.
%Since $\Gamma\in\mc G$, by definition $W_1(\tilde\Gamma,\hat\Gamma)\le W_1(\Gamma,\hat\Gamma)$.
%Then, by the triangle inequality, $W_1(\tilde\Gamma,\Gamma)\le W_1(\tilde\Gamma,\hat\Gamma)+W_1(\Gamma,\hat\Gamma)\le 2W_1(\hat\Gamma,\Gamma)$.
%Taking an expectation and then a supremum over $\Gamma\in\mc G$ on both sides, we have $\sup_{\Gamma\in\mc G}\Expect W_1(\tilde\Gamma,\Gamma)\le 2\sup_{\Gamma\in\mc G}\Expect W_1(\hat\Gamma,\Gamma)$.
%The second inequality in~\prettyref{eq:lb_split} follows.}
To obtain the lower bound in~\prettyref{eq:mixing_est_avg_risk}, we apply the $\Omega((d/n)^{1/4}\wedge 1)$ lower bound in~\cite[Theorem 10]{Wu_Zhou_2019}
for $d$-dimensional symmetric $2$-GM (by taking $\calG$ to the mixtures of the form \prettyref{eq:sym2GM}) and the $\Omega(n^{-1/(4k-2)})$ lower bound in~\cite[Proposition 7]{WY18} for $1$-dimensional $k$-GM (by taking $\calG$ to be the set of mixing distributions all but the first coordinates of which are zero).}
\end{proof}

\section{Density estimation} \label{sec:density_estimation}

In this section we prove the density estimation guarantee of \prettyref{thm:density_est} for finite Gaussian mixtures. 
The lower bound simply follows from the minimax quadratic risk of the Gaussian location model (corresponding to $k=1$), since
$H^2(N(\theta,I_d),N(\theta',I_d)) = 2 -2e^{\|\theta-\theta'\|_2^2/8} \asymp \|\theta-\theta'\|^2$ when $\theta,\theta'\in B(0,R)$.
Thus, we focus on the attainability of the parametric rate of density estimation. 
Departing from the prevailing approach of maximum likelihood, we aim to apply the estimator of Le Cam and Birg\'e which requires bounding the local entropy of Hellinger balls for $k$-GMs. This is given by the following lemma.

\begin{lem}[Local entropy of $k$-GM] \label{lem:hellinger_local_entropy}
For any $\Gamma_0 \in \mathcal G_{k, d}$, let $\mathcal P_{\epsilon}(\Gamma_0) = \{P_{\Gamma} : \Gamma \in \mathcal G_{k, d}, H(P_{\Gamma},P_{\Gamma_0})\le \epsilon\}$. Then, for any $\delta \le \epsilon/2$,
\begin{equation}
N(\delta, \mathcal P_{\epsilon}(\Gamma_0), H) \le \added{\(\frac{\epsilon}{\delta}\)^{c(dk^4 + (2k)^{2k+2})},}
%(\epsilon /\delta )^{c\cdot d},
\end{equation}
where the constant $c$ only depends only on $R$.
\end{lem}

\prettyref{lem:hellinger_local_entropy} shows that any $\epsilon$-Hellinger ball $\mathcal P_{\epsilon}(\Gamma_0)$ in the space of $k$-GMs
can be covered by at most $(\frac{C\epsilon}{\delta})^{Cd}$ $\delta$-Hellinger balls for some $C=C(k,R)$. This result is uniform in $\Gamma_0$ and depends optimally on $d$ but the dependency on the number of components $k$ is highly suboptimal.
\prettyref{lem:hellinger_local_entropy}  should be compared with the local entropy bound in \cite{gassiat2014local} obtained using a different approach than ours based on moment tensor.
Specifically, \cite[Example 3.4]{gassiat2014local} shows that for Gaussian location mixtures, the local bracketing entropy centered at $P_{\Gamma_0}$ is bounded by 
$N_{[]}(\delta, \mathcal P_{\epsilon}(\Gamma_0), H) \leq (\frac{C'\epsilon}{\delta})^{20kd}$, for some constant $C'$ depending on $P_{\Gamma_0}$ and $R$.
This result yields optimal dependency on both $d$ and $k$ but lacks uniformity in the center of the Hellinger neighborhood (which is needed for applying the theory of Le Cam and Birg\'e).

%Given the local entropy bound in \prettyref{lem:hellinger_local_entropy}, the parametric rate of $O_k(\frac{d}{n})$ in the squared Hellinger loss follows from the construction of Le Cam and Birg\'e (see~\cite{LeCam73,Birge83,birge1986estimating}; see also \cite[Lec.~18]{it-stats} for a self-contained exposition).
 %\prettyref{lem:hellinger_local_entropy} is proved in Sections \ref{sec:moment-tensor} and \ref{sec:localentropy}.

Given the local entropy bound in \prettyref{lem:hellinger_local_entropy}, the upper bound $O_k(\frac{d}{n})$  in the squared Hellinger loss in \prettyref{thm:density_est}
 immediately follows by invoking the Le Cam-Birg\'e construction \cite[Theorem 3.1]{birge1986estimating}; see also \cite[Lec.~18]{it-stats} for a self-contained exposition). For a high-probability bound that leads to \prettyref{eq:density_est-highprob}, see, e.g.,~\cite[Theorem 18.3]{it-stats}.
%follows from the construction of Le Cam and Birg\'e (see~\cite{LeCam73,Birge83,birge1986estimating}; see also \cite[Lec.~18]{it-stats} for a self-contained exposition).
%From this the upper bound in \prettyref{thm:density_est} immediately follows by invoking the Le Cam-Birg\'e construction (\cite[Theorem 3.1]{birge1986estimating}; for the high-probability bound \prettyref{eq:density_est-highprob} see e.g.,~\cite[Theorem 18.3]{it-stats}).

Before proceeding to the proof of \prettyref{lem:hellinger_local_entropy}, we note that the Le Cam-Birg\'e construction, based on (exponentially many) pairwise tests, does not lead to a computationally efficient scheme for density estimation. 
This problem is much more challenging than estimating the mixing distribution, for which we have already obtained a polynomial-time optimal estimator in \prettyref{sec:mixing_estimation}. (In fact, we show in~\prettyref{sec:de-connection}, estimation of the mixing distribution can be reduced to proper density estimation both statistically and computationally.)
Finding a computationally efficient proper density estimate that attains the parametric rate in \prettyref{thm:density_est} for arbitrary $k$, or even within logarithmic factors thereof, is open. \prettyref{ssec:density_est_efficient} presents some partial progress on this front: We show that the estimator in \prettyref{sec:mixing_estimation} with slight modifications achieves the optimal rate of $O(\sqrt{d/n})$ for $2$-GMs and the rate of $O((d/n)^{1/4})$ for general $k$-GMs; the latter result slightly improves (by logarithmic factors only) the state of the art in~\cite{Acharya_etal_2014}, but is still suboptimal. 
%We hypothesize that our proposed algorithm does indeed achieve the optimal rate for all $k$-GM's, and proving this is the subject of current work.

%In~\prettyref{sec:de-connection}, we moreover show that a proper density estimator yields an estimator of the mixing distribution that achieves the sharp minimax rate stated in~\prettyref{thm:mixing_est}. This estimator, however, is computationally suboptimal compared to the one that achieves~\prettyref{eq:mixing_est_tail_bound}, since its time complexity is the same as that of the density estimator.

Both the construction of the Hellinger covering for \prettyref{lem:hellinger_local_entropy} and the analysis of density estimation in \prettyref{ssec:density_est_efficient} rely
on the notion of moment tensors, which we now introduce.

\subsection{Moment tensors and information geometry of Gaussian mixtures}
\label{sec:moment-tensor}

We recall some basics of tensors; for a comprehensive review, see~\cite{kolda2009tensor}. The \emph{rank} of an order-$\ell$ tensor $T\in(\reals^d)^{\otimes \ell}$ is defined as the minimum $r$ such that $T$ can be written the sum of $r$ rank-one tensors, namely \cite{Kruskal1977}:
\begin{equation}
\rank(T)\triangleq\min \sth{r:T=\sum_{i=1}^r  \alpha_i u^{(1)}_i\otimes  \dots \otimes u^{(\ell)}_i, \quad u^{(j)}_i \in \reals^d, \alpha_i \in \reals },
%, i\in[r],j\in[\ell]
\end{equation}
We will also use the \emph{symmetric rank} \cite{CGLM2008}:
\begin{equation}
\rank_s(T)\triangleq\min \sth{ r:T=\sum_{i=1}^r \alpha_i u_i^{\otimes \ell}, \quad u_i \in \reals^d, \alpha_i \in \reals }.
\end{equation}
An order-$\ell$ tensor $T$ is \emph{symmetric} if $T_{j_1,\ldots, j_{\ell}} = T_{j_{\pi(1)},\ldots,j_{\pi(\ell)}}$ for all $j_1,\ldots,j_\ell \in [d]$ and all permutations $\pi$ on $[\ell]$.
% The Frobenius norm is defined as $\Norm{T}_F=\sqrt{\inner{T,T}}$.
% For two rank-1 tenors $x=u\otimes v \otimes \dots \otimes w$ 
The Frobenius norm of a tensor $T$ is defined as $\norm{T}_F \triangleq \sqrt{\inprod{T}{T}}$, where the tensor inner product is defined as $\Iprod{S}{T}=\sum_{j_1,\ldots,j_\ell\in[d]} S_{j_1,\ldots, j_{\ell}}T_{j_1,\ldots, j_{\ell}}$.
The spectral norm (operator norm) of a tensor $T$ is defined as 
\begin{equation}
\norm{T} \triangleq \max\{ \inprod{T}{u_1\otimes u_2 \otimes \dots \otimes u_{\ell}} : \norm{u_i}= 1, i =1,\dots,\ell \}.
\label{eq:Top}
\end{equation}
Denote the set of $d$-dimensional order-$\ell$ symmetric tensors by $\symmetric_{\ell}(\reals^d)$. 
For a symmetric tensor, the following result attributed to Banach (\cite{Banach1938,friedland2018nuclear}) is crucial for the present paper:
\begin{equation}
	\label{eq:op-symmetric}
	\norm{T} = \max\{ |\inprod{T}{u^{\otimes \ell}}|: \norm{u} = 1 \}.
\end{equation}
% \nb{I think you missed abs value. Also, I think $\Norm{u}= 1$ is wlog, so I suggest we put equality for both here and above.}
For $T\in \symmetric_{\ell}(\reals^d)$, if $\rank_s(T)\le r$, then the spectral norm can be bounded by the Frobenius norm as follows \cite{Qi2011}:\footnote{The weaker bound $\norm{T} \geq r^{-\ell/2} \norm{T}_F$, which suffices for the purpose of this paper, takes less effort to show. Indeed, in view of the Tucker decomposition \prettyref{eq:tucker}, combining~\prettyref{eq:Top} with \prettyref{eq:tucker-norm} yields that $\norm{T} \geq \max_{\bfj \in [r]^\ell} |\alpha_\bfj| \geq r^{-\ell/2} \norm{\alpha}_F = r^{-\ell/2} \norm{T}_F$. %\nbwu{Pengkun please double check.}
}
\begin{equation}
\label{eq:T-norms}	
\frac{1}{\sqrt{r^{\ell-1}}}\norm{T}_F \le \norm{T} \le \norm{T}_F.
\end{equation}

For any $d$-dimensional random vector $U$, its order-$\ell$ \emph{moment tensor} is 
\begin{equation}
M_{\ell}(U) \triangleq \mathbb E[\underbrace{U\otimes \cdots \otimes U}_{\text{$\ell$ times}}],
\label{eq:mtensor}
\end{equation}
which, by definition, is a symmetric tensor; in particular, $M_1(U)=\mathbb E[U]$ and $M_2(U-\mathbb E [U])$ are the mean and the covariance matrix of $U$, respectively. Given a multi-index $\bfj = (j_1,\ldots,j_d)\in \mathbb Z^d_+$, the $\bfj$th (multivariate) moment of $U$
% \nb{I recommend we use bold $\bfj$ for multiindex to be clear.}
\begin{equation}
m_{\bfj}(U) = \mathbb E[U_{j_1} \cdots U_{j_d}]
\end{equation}
is the $\bfj$th entry of the moment tensor $M_{|\bfj|}(U)$, with $|\bfj| \triangleq j_1+\ldots j_d$.
Since moments are functionals of the underlying distribution, we also use the notation $M_\ell(\Gamma)=M_\ell(U)$ where $U \sim \Gamma$.
An important observation is that the moment of the projection of a random vector can be expressed in terms of the moment tensor as follows:
for any $u\in\reals^d$,
\[
m_\ell(\Iprod{X}{u}) = \Expect[\Iprod{X}{u}^\ell]=\Expect[\Iprod{X^{\otimes \ell}}{u^{\otimes \ell}}]=\Iprod{M_\ell(X)}{u^{\otimes \ell}}.
\]
Consequently, the difference between two moment tensors measured in the spectral norm 
is equal to the maximal moment difference of their projections. Indeed, thanks to \prettyref{eq:op-symmetric},
\begin{equation}
\label{eq:mtensor-opnorm}
\|M_\ell(X)-M_\ell(Y)\| = \sup_{\|u\|=1} |m_\ell(\Iprod{X}{u})-m_\ell(\Iprod{Y}{u})|.
\end{equation}

Furthermore, if $U$ is a discrete random variable with a few atoms, then its moment tensor has low rank.
%; this observation was the basis of previous work on applying a tensor decomposition to learning Gaussian mixtures (see~\cite{AGHKT2014}). 
Specifically, if $U$ is distributed according to some $k$-atomic distribution $\Gamma = \sum_{i=1}^k w_i\delta_{\mu_i}$, then  
\begin{equation}
M_\ell(\Gamma) = \sum_{i=1}^k w_i\mu_i^{\otimes \ell},
\label{eq:mtensor-k}
\end{equation}
whose symmetric rank is at most $k$.

%Let $\mathcal G_k$ denote the set of $k$-atomic distributions supported on $[-1,1]^d$.
%% Let $\calG_k=\{U\in\reals^d: \Norm{U}\le 1, U\text{ is $k$-atomic}\}$ denote the set of $k$-atomic bounded random variables \nb{I highly recommend we switch to the language of distributions} in $\reals^d$.
%For any $\nu \sim \sum_{i=1}^k w_i\delta_{\mu_i}\in \mathcal G_k$, let $P_{\nu} = \nu*N(0,I_d)=\sum_i w_i N(\mu_i,I_d)$, which is an order-$k$ Gaussian mixture. 
%% $U+Z_d$, where $Z_d\sim N(0,I_d)$.

The following result gives a characterization of statistical distances (squared Hellinger, KL, or $\chi^2$-divergence) between $k$-GMs in terms of the moment tensors up to \emph{dimension-independent} constant factors. 
Note that the upper bound in one dimension has been established in \cite{WY18} (by combining Lemma 9 and 10 therein).
\begin{theorem}[Moment characterization of statistical distances]
\label{thm:H-M}
%For any $k\in\naturals$ and $M>0$, there exist constants $c$ and $C$ depending on $k$ and $M$, such that 
For any pair of $k$-atomic distributions $\Gamma, \Gamma'$ supported on the ball $B(0,R)$ in $\reals^d$, 
for any $D \in \{H^2,\KL,\chi^2\}$,
\begin{equation}
\label{eq:H-M}
\added{(Ck)^{-4k}} \max_{\ell \le 2k-1} \Fnorm{M_{\ell}(\Gamma)-M_{\ell}(\Gamma')}^2
\le D(P_\Gamma, P_{\Gamma'}) \le \added{C e^{36k^2}} \max_{\ell \le 2k-1} \Fnorm{M_{\ell}(\Gamma)-M_{\ell}(\Gamma')}^2.
\end{equation}
where the constant $C$ may depend on $R$ but not $k$ or $d$.
\end{theorem}

%Interpolation of moment tensors
To prove \prettyref{thm:H-M} we need a few auxiliary lemmas. The following lemma bounds the difference of higher-order moment tensors of $k$-atomic distributions using those of the first $2k-1$ moment tensors. The one-dimensional version was shown in \cite[Lemma 10]{WY18} using polynomial interpolation techniques; however, it is hard to extend this proof to multiple dimensions as multivariate polynomial interpolation (on arbitrary points) is much less well-understood. Fortunately, this difficulty can be sidestepped by exploiting the relationship between moment tensor norms and projections in \prettyref{eq:mtensor-opnorm}.
%It is known that any $k$-atomic distribution can be identified using its moments up to degree $2k-1$; this result holds in arbitrary dimension. 
%Furthermore, using polynomial interpolation, it is shown in \cite[Lemma 10]{WY18} that in one dimension, the difference of higher order moments can be bounded using those of the first $2k-1$ moments; however, it is difficult to extend this proof to multiple dimensions as multivariate polynomial interpolation (on arbitrary points) is much less well-understood. 
%Nevertheless, it is possible to sidestep this difficulty by exploiting the relationship between moment tensor norms and projections. The following lemma bounds the difference of higher-order moment tensors using lower-order ones:
\begin{lemma}
\label{lmm:moment-interpolation}	
	Let $U,U'$ be $k$-atomic random variables in $\reals^d$. Then for any $j \geq 2k$,
	\[
	\|M_j(U)-M_j(U')\| \leq 3^j \max_{\ell\in[2k-1]} \|M_\ell(U)-M_\ell(U')\|.
	\]
\end{lemma}
\begin{proof}
	\begin{align*}
	\|M_j(U)-M_j(U')\| 
	\stepa{=} & ~ \sup_{\|v\|=1} |m_j(\Iprod{U}{v})-m_j(\Iprod{U'}{v})|  \\
	\stepb{\leq} & ~ 3^j \sup_{\|v\|=1} \max_{\ell\in[2k-1]}|m_\ell(\Iprod{U}{v})-m_\ell(\Iprod{U'}{v})|  \\
	\stepc{=} & ~ 3^j \max_{\ell\in[2k-1]}\|M_\ell(U)-M_\ell(U')\|,
	\end{align*}
	where (a) and (c) follow from \prettyref{eq:mtensor-opnorm}, and (b) follows from \cite[Lemma 10]{WY18}.
\end{proof}

The lower bound part of \prettyref{thm:H-M} can be reduced to the one-dimensional case, which is covered by the following lemma. The proof relies on Newton interpolating polynomials and is deferred till \prettyref{sec:l2_moments_lb}.
\begin{lem} \label{lem:l2_moments_lb}
Let $\gamma, \gamma'$ be $k$-atomic distributions supported on $[-R, R]$. 
%Let $p, p$ be the densities associated with the distributions $P_{\Gamma'} = \gamma * N(0, 1)$ and $P_{\gamma_0} = \gamma_0 * N(0, 1)$, respectively. Then there are constants $C_{k,1}, C_{k, 2}$ such that for any $r \in [2k-1]$, 
Then for any $(2k-1)$-times differentiable test function $h$,
\begin{align}
H(\gamma*N(0,1), \gamma' *N(0,1)) \ge c \left| \int h d\gamma - \int h d\gamma'  \right|,
%|m_r(\gamma) - m_r(\gamma_0)| 
\label{eq:hellinger_moments}
\end{align}
where $c$ is some constant depending only on $k$, $R$, and $\max_{0\leq i\leq 2k-1} \|h^{(i)}\|_{L_\infty([-R,R])}$. \added{In the particular case where $h(x) = x^i$ for $i \in [2k-1], c \ge (Ck)^{-k}$ for a constant $C$ depending only on $R$.}
\end{lem}

\begin{proof}[Proof of \prettyref{thm:H-M}]
Since
\begin{align}
H^2(P, Q) \le \KL(P \| Q) \le \chi^2(P \| Q ),
\label{eq:divcompare}
\end{align}
(see, e.g., \cite[Section 2.4.1]{Tsybakov09}), %\prettyref{eq:divcompare}, 
it suffices to prove the lower bound for $H^2$ and the upper bound for $\chi^2$.

 Let $U\sim \Gamma$ and $U'\sim \Gamma'$,  $X\sim P_\Gamma=\Gamma * N(0,I_d)$ and $X' \sim P_{\Gamma'}=\Gamma'*N(0,I_d)$. 
% For any $\theta\in S^{d-1}$, 
% let $\Gamma_\theta = \calL(\Iprod{\theta}{U})$ and $\Gamma'_\theta = \calL(\Iprod{\theta}{U'})$, both of which are $k$-atomic distributions in one dimension.
Then $\Iprod{\theta}{X} \sim P_{\Gamma_\theta}$ and $\Iprod{\theta}{X'} \sim P_{\Gamma_\theta'}$. By the data processing inequality, 
\begin{equation}
H(P_{\Gamma},P_{\Gamma'})\ge\sup_{\theta\in S^{d-1}}H(P_{\Gamma_\theta},P_{\Gamma'_\theta}).
\end{equation}
Applying~\prettyref{lem:l2_moments_lb} to all monomials of degree at most $2k-1$, we obtain
 \begin{equation}
 H(P_{\Gamma},P_{\Gamma'})\ge (Ck)^{-k}  \sup_{\theta \in S^{d-1}}\max_{\ell \le 2k-1}|m_{\ell}(\Iprod{\theta}{U})-m_{\ell}(\Iprod{\theta}{U'})|
 = (Ck)^{-k} \max_{ \ell \le 2k-1} \norm{M_{\ell}(U)-M_{\ell}(U')},
 \end{equation}
for some constant $C$, where the last equality is due to \prettyref{eq:mtensor-opnorm}.
Thus the desired lower bound for Hellinger follows from the tensor norm comparison in \prettyref{eq:T-norms}.

%Since for $X\sim P_{\Gamma}$ we have $V^\top X \sim P_{V^\top U}$, and for $\tilde X\sim P_{V^\top U}$ we have $V\tilde X+V_\perp Z_{d-2k}\sim P_{\Gamma}$. Then, by the data processing inequality, we have $\chi^2(P_{\Gamma} \| P_{\Gamma'} ) =\chi^2(P_{V^\top U} \| P_{V^\top U'} )$.  Next we bound the moment differences:
 %\[
 %|m_\bfj(V^\top U)-m_\bfj(V^\top U')|
 %\le \Norm{M_{|\bfj|}(V^\top U)-M_{|\bfj|}(V^\top U')}_F
 %\le (2k)^{{|\bfj|}}\Norm{M_{|\bfj|}(V^\top U)-M_{|\bfj|}(V^\top U')}.
 %\]
 %By the one-dimensional result \cite{},
 %\[
 %\Norm{M_{|\bfj|}(V^\top U)-M_{|\bfj|}(V^\top U')}
 %\le \Norm{M_{|\bfj|}(U)-M_{|\bfj|}(U')}
 %\le 3^\ell \max_{\ell \le 2k-1} \Norm{M_\ell(U)-M_\ell(U')}.
 %\]
 %Therefore,
 %\[
 %\chi^2(P_{V^\top U} \| P_{V^\top U'} ) \le C \max_{\ell \le 2k-1} \Norm{M_\ell(\nu)-M_\ell(\nu')}^2\sum_{\bfj}\frac{(18k)^{|\bfj|}}{\bfj!}=C e^{36k^2}\max_{\ell \le 2k-1} \Norm{M_\ell(\nu)-M_\ell(\nu')}^2.
 %\]

To show the upper bound for $\chi^2$, we first reduce the dimension from $d$ to $2k$. 
Without loss of generality, assume that $d \geq 2k$ (for otherwise we can skip this step).
Since both $U$ and $U'$ are $k$-atomic, the collection of atoms of $U$ and $U'$ lie in some subspace spanned by the orthonormal basis $\{v_1,\dots,v_{2k}\}$. Let $V=[v_1,\dots,v_{2k}]$ and let $V_\perp = [v_{2k+1},\dots, v_d]$ consist of orthonormal basis of the complement, so that $[V,V_\perp]$ is an orthogonal matrix. 
Write $X=U+Z$, where $Z\sim N(0,I_d)$ is independent of $U$. 
Then $V^\top X = V^\top U + V^\top Z \sim \nu * N(0, I_{2k}) = P_\nu$, where $\nu=\calL(V^\top U)$ is a $k$-atomic distribution on $\reals^{2k}$. Furthermore, $V_\perp^\top X = V_\perp^\top Z \sim N(0,I_{d-2k})$ and is independent of $V^\top X$.
Similarly, 
$(V^\top X',V_\perp^\top X') \sim P_{\nu'} \otimes N(0,I_{d-2k})$, where $\nu'=\calL(V^\top U')$.
Therefore, 
\begin{align*}
\chi^2(P_\Gamma \| P_{\Gamma'}) 
= & ~  \chi^2(\calL(V^\top X,V_\perp^\top X) \| \calL(V^\top X',V_\perp^\top X')) = 
\chi^2(P_{\nu} \otimes N(0,I_{d-2k}) \| P_{\nu'} \otimes N(0,I_{d-2k})) \\
= & ~ 	\chi^2(P_{\nu} \| P_{\nu'} ).
\end{align*}
For notational convenience, let $B=V^\top U\sim\nu$ and $B'=V^\top U'\sim\nu'$.

To bound $\chi^2(P_{\nu} \| P_{\nu'} )$, we first assume that $\Expect[B']=0$.
%By following the argument in \cite[Lemma 9]{WY18}, 
%we show that
%\begin{equation}
%%\chi^2(P_{V^\top U} \| P_{V^\top U'} ) \lesssim \sum_{\bfj\in\integers_+^{2k}}\frac{(m_\bfj(V^\top U)-m_\bfj(V^\top U'))^2}{\bfj!}.
%\chi^2(P_{\Gamma} \| P_{\Gamma'} ) \lesssim \sum_{\bfj\in\integers_+^{d}}\frac{(m_\bfj(U)-m_\bfj(U'))^2}{\bfj!}.
%\label{eq:chi-ub}
%\end{equation}
For each multi-index $\bfj=(j_1,\ldots,j_d)\in\integers_+^{2k}$, define the $\bfj$th Hermite polynomial as 
\begin{equation}
H_\bfj(x) = \prod_{i=1}^{2k} H_{j_i}(x_i), \quad x \in \reals^{2k}
\label{eq:Hbfj}
\end{equation}
which is a degree-$|\bfj|$ polynomial in $x$. Furthermore, the following orthogonality property is inherited from that of univariate Hermite polynomials:
for $Z\sim N(0,I_{2k})$,
\begin{equation}
\Expect[H_{\bfj}(Z) H_{\bfj'}(Z)]=\bfj! \indc{\bfj=\bfj'}.
\label{eq:hermite-orthod}
\end{equation}
		Recall the exponential generating function of Hermite polynomials (see~\cite[22.9.17]{AS64}): 
		for $x,b\in\reals$, $\phi(x-b)=\phi(x)\sum_{j\ge 0}H_j(x)\frac{b^j}{j!}$. It is straightforward to obtain the multivariate extension of this result:
		\[
		\phi_{2k}(x-b)=\phi_{2k}(x) \sum_{\bfj\in\integers_+^{2k}} \frac{H_\bfj(x)}{\bfj!} \prod_{i=1}^{2k} b_i^{j_i}, \quad x,b\in\reals^{2k}.
		\]
		Integrating $b$ over $B\sim\nu$, we obtain the following expansion of the density of $P_\nu$:
		\[
		p_\nu(x) = \Expect[\phi_{2k}(x-B)] = \phi_{2k}(x) \sum_{\bfj\in\integers_+^{2k}}  \frac{H_\bfj(x)}{\bfj!} \underbrace{\Expect\qth{\prod_{i=1}^{2k} B_i^{j_i}}}_{m_\bfj(B)} .
		\]
		Similarly, 
		$p_{\nu'}(x) =\phi_{2k}(x) \sum_{\bfj\in\integers_+^{2k}}  \frac{1}{\bfj!} m_\bfj(B') H_\bfj(x)$.
		Furthermore,
		 by the assumption that $\Expect[B']=0$ and $\|B'\|\leq \|U'\| \leq R$ almost surely, Jensen's inequality yields 
    \[
        p_{\nu'}(x) =\phi_{2k}(x)\Expect[\exp(\Iprod{B'}{x}-\|B'\|^2/2)] \ge \phi_{2k}(x)\exp(-R^2/2).
    \]
    Consequently, 
		\begin{align*}
		\chi^2(P_{\nu}\|P_{\nu'})
		\leq & ~ e^{R^2/2} \int_{\reals^{2k}} dx \frac{(p_{\nu}(x)-p_{\nu'}(x))^2}{\phi_{2k}(x)} \\
		\stepa{=} & ~ e^{\frac{R^2}{2}} \sum_{\bfj\in\integers_+^{2k}}\frac{(m_\bfj(B)-m_\bfj(B'))^2}{\bfj!} \\
		\stepb{\leq} & ~ e^{\frac{R^2}{2}} \sum_{\ell\geq 1} \frac{\Fnorm{M_\ell(B)-M_\ell(B')}^2}{\ell!} (2k)^{\ell} \\
		\stepc{\leq} & ~ e^{\frac{R^2}{2}} e^{2k}  \max_{\ell\in[2k-1]} \Fnorm{M_\ell(B)-M_\ell(B')}^2 +
		e^{\frac{R^2}{2}} \sum_{\ell\geq 2k} \frac{(4k^2)^{\ell}}{\ell!} \norm{M_\ell(B)-M_\ell(B')}^2 \\
		\stepd{\leq} & ~ e^{\frac{R^2}{2}} \Bigg(e^{2k} + \underbrace{\sum_{\ell\geq 2k} \frac{(36k^2)^{\ell}}{\ell!}}_{\leq e^{36k^2}}  \Bigg)
				\max_{\ell\in[2k-1]} \Fnorm{M_\ell(B)-M_\ell(B')}^2,
		\end{align*}
		where (a) follows from the orthogonality relation \prettyref{eq:hermite-orthod}; (b) is by the fact that $(|\bfj|)! \leq \bfj! (2k)^{|\bfj|}$ for any $\bfj\in\integers_+^{2k}$;
		(c) follows from the tensor norm comparison inequality \prettyref{eq:T-norms}, since the symmetric rank of $M_\ell(B)-M_\ell(B')$ is at most $2k$ for all $\ell$;
		(d) follows from \prettyref{lmm:moment-interpolation}.

Finally, if $\Expect[B'] \neq 0$, by the shift-invariance of $\chi^2$-divergence, applying the following simple lemma to $\mu=\Expect[B']$ (which satisfies $\|\mu\|\leq R$) yields the desired upper bound.
 \end{proof}

\begin{lemma}
\label{lmm:moment-centering}		
	For any random vectors $X$ and $Y$ and any deterministic $\mu\in\reals^d$,
	\[
	\|M_\ell(X-\mu)-M_\ell(Y-\mu)\| \leq \sum_{k=0}^{\ell} \binom{\ell}{k} \|M_k(X)-M_k(Y)\| \|\mu\|^{\ell-k}
	\]	
\end{lemma}
\begin{proof}
Using \prettyref{eq:mtensor-opnorm} and binomial expansion, we have: 
	\begin{align*}
	\|M_\ell(X-\mu)-M_\ell(Y-\mu)\| 
	= & ~ \sup_{\|u\|=1} |m_\ell(\Iprod{X}{u}-\Iprod{\mu}{u})-m_\ell(\Iprod{Y}{u}-\Iprod{\mu}{u})|. \\
	\leq & ~ 	\sup_{\|u\|=1} \sum_{k=0}^{\ell} \binom{\ell}{k} |m_k(\Iprod{X}{u})-m_k(\Iprod{Y}{u})| |\Iprod{\mu}{u}|^{\ell-k}  \\
	\leq & ~ 	\sum_{k=0}^{\ell} \binom{\ell}{k} \|M_k(X)-M_k(Y)\| \|\mu\|^{\ell-k}
	\end{align*}
	where in the step we used the Cauchy-Schwarz inequality.
\end{proof}

%If~\prettyref{lem:chi_square_moments_ub} works, we have the first line of the following. Then the next three lines follow from definition and~\prettyref{eq:T-norms}.
%\begin{align*}
%\chi^2 \(P_{\Gamma} ||  P_{\Gamma'} \) &\le \sup_{\theta \in S^{d-1}, \ell \in [2k-1]} | m_\ell(\theta^{\top} U) - m_\ell(\theta^{\top} U')|^2 \\
%&= \sup_{\theta \in S^{d-1}, \ell \in [2k-1]} \inprod{M_\ell(U) - M_\ell(U')}{\theta^{\otimes \ell}} \\
%&= \max_{\ell \in [2k-1]} \norm{M_\ell(U) - M_\ell(U')}^2 \\
%&\le \max_{\ell \in [2k-1]} \norm{M_\ell(U) - M_\ell(U')}_F^2.
%\end{align*}

\subsection{Local entropy of Hellinger balls}
\label{sec:localentropy}

Before presenting the proof of \prettyref{lem:hellinger_local_entropy}, we discuss the connection and distinction between our approach and the existing literature on the metric entropy of mixture densities \cite{Ghosal_VdV_2001,Ho_Nguyen_2016_annals,Ho_Nguyen_2016_ejs,Zhang_2009,maugis2011non,Bontemps_Gadat_2014,Saha_2017} and clarify the role of the moment tensors. 
Both the previous work and the current paper bound the statistical difference between mixtures in terms of moment differences (through either Taylor expansion or orthogonal expansion). For example, the seminal work \cite{Ghosal_VdV_2001} bounds the global entropy of nonparametric Gaussian mixtures in one dimension by first constructing a finite mixture approximation via moment matching, then discretizing the weights and atoms.
%, as carried out in \cite[Lemma 3.1 and Theorem 3.1]{Ghosal_VdV_2001}, respectively.
The crucial difference is that in the present paper we directly work with moment parametrization as opposed to the natural parametrization (atoms and weights). As mentioned in \prettyref{sec:main}, to eliminate the unnecessary logarithmic factors and obtain the exact parametric rate in high dimensions, it is crucial to obtain a tight control of the \emph{local} entropy as opposed to the global entropy, which relies on a good parametrization that bounds the Hellinger distance from both above and below -- see \prettyref{eq:good-param}. This \emph{two-sided bound} is satisfied by the moment tensor reparametrization, thanks to \prettyref{thm:H-M}, but not the natural parametrization. Therefore, to construct a good local covering, we do so in the moment space, by leveraging the low-rank structure of moment tensors.

\begin{proof}[Proof of~\prettyref{lem:hellinger_local_entropy}]
Recall from \prettyref{sec:notation} that $N(\epsilon,A, \rho)$ the $\epsilon$-covering number of the set $A$ with respect to the metric $\rho$, i.e., the minimum cardinality of an $\epsilon$-covering set $A_{\epsilon}$ such that, for any $v \in A$, there exists $\tilde v \in A_{\epsilon}$ with $\rho(v, \tilde v)<\epsilon$. 

Let $\mathcal M_{\epsilon} = \{M(\Gamma)\colon P_{\Gamma} \in \mathcal P_{\epsilon}\}$, where 
$\calP_\epsilon$ is the Hellinger neighborhood of $P_{\Gamma_0}$, 
$M(\Gamma)=(M_1(\Gamma),\dots,M_{2k-1}(\Gamma))$ consists of the moment tensors of $\Gamma$ up to degree $2k-1$. \added{Let $c_k'=\sqrt{c_k}$ and $C_k'=\sqrt{C_k}$ where $c_k \triangleq (Ck)^{-4k}$ and $C_k \triangleq C'e^{36k^2}$ are the constants from~\prettyref{thm:H-M}.}
To obtain a $\delta$-covering of $\calP_\epsilon$, we first show that it suffices to construct a $\frac{\delta}{2C_k'}$-covering of the moment space $\calM_\epsilon$ with respect to the distance $\rho(M,M')\triangleq \max_{\ell\le 2k-1} \norm{M_\ell-M_\ell'}_F$ and thus
% It follows from~\prettyref{thm:H-M} that
\begin{equation}
N(\delta, \mathcal P_{\epsilon}, H) \le N(\delta/(2C_k'), \mathcal M_{\epsilon}, \rho). \label{eq:hellinger_moments_cover}
\end{equation}
To this end, let $\calN$ be the optimal $\frac{\delta}{2C_k'}$-covering of $\calM_\epsilon$ with respect to $\rho$, and we show that $\calN'=\{P_\Gamma: \Gamma= \argmin_{\Gamma':P_{\Gamma'}\in\calP_\epsilon} \rho(M(\Gamma'), M), M\in\calN\}$ is a $\delta$-covering of $\calP_\epsilon$. For any $P_\Gamma\in\calP_\epsilon$, by the covering property of $\calN$, there exists a tensor $M\in\calN$ such that $\rho(M,M(\Gamma))<\frac{\delta}{2C_k'}$. 
By the definition of $\calN'$, there exists $P_{\tilde\Gamma}\in \calN'$ such that $\rho(M(\tilde\Gamma),M)<\frac{\delta}{2C_k'}$. 
Therefore, $\rho(M(\tilde\Gamma),M(\Gamma))<\frac{\delta}{C_k'}$ and thus $H(P_{\tilde\Gamma},P_\Gamma)<\delta$ by~\prettyref{thm:H-M}.
% \footnote{
% 	Alternatively, the change of variables step can be proved by requiring the $\delta$-covering to be a subset of $\calP_\epsilon$. 
% }
% where $\rho(M,M')\triangleq \max_{\ell\le 2k-1} \norm{M_\ell-M_\ell'}_F$. 

Next we bound the right side of \prettyref{eq:hellinger_moments_cover}. 
Since $\Gamma_0,\Gamma$ are both $k$-atomic, it follows from \prettyref{thm:H-M} that
\begin{equation}
\mathcal M_{\epsilon} \subseteq M(\Gamma_0) + \{\Delta: \norm{\Delta_\ell}_F \le \epsilon/c_k', \rank_s(\Delta_\ell)\le 2k, \forall \ell \le 2k-1\},
\end{equation}
where $\Delta=(\Delta_1,\dots,\Delta_{2k-1})$ and $\Delta_{\ell} \in \symmetric_{\ell}(\reals^d)$.
Let $\mathcal D_{\ell} = \{\Delta_{\ell} \in \symmetric_{\ell}(\reals^d): \norm{\Delta_{\ell}}_F\le \epsilon/c_k', \rank_s(\Delta_\ell)\le 2k\}$, and $\mathcal D = \mathcal D_1 \times \dots \times \mathcal D_{2k-1}$ be their Cartesian product.
By monotonicity,
\begin{equation}
N(\delta/(2C_k'), \mathcal M_{\epsilon}, \rho) \le N(\delta/(2C_k'), \mathcal D, \rho)
\le \prod_{\ell=1}^{2k-1}N(\delta/(2C_k'), \mathcal D_{\ell}, \norm{\cdot}_F).
\end{equation}
\added{By \prettyref{lem:N-tensor-low-rank} next,}
\begin{equation}
N(\delta/(2C_k'), \mathcal M_{\epsilon}, \rho) \le \prod_{\ell=1}^{2k-1} \( \frac{\added{c} C_k' \ell \epsilon}{2 c_k'\delta} \)^{2dk} \( \frac{\added{c} C_k' \epsilon}{2 c_k' \delta} \)^{(2k)^\ell},
\end{equation}
for some absolute constant $c$. So we obtain, for constants $\tilde C,\tilde C'$ that does not depend on $d$ or $k$, that
\begin{align*}
N(\delta/(2C_k'), \mathcal M_{\epsilon}, \rho) &\le \( \frac{{\tilde C} k^{4k} e^{36k^2} k \epsilon}{\delta} \)^{4dk^2} \( \frac{{\tilde C} k^{4k} e^{36k^2} \epsilon}{\delta} \)^{(2k)^{2k}} \\
&\le \( \frac{\tilde C' k^{4k+1} e^{36k^2} \epsilon}{\delta} \)^{4dk^2 + (2k)^{2k}}.
\end{align*}
%since we can swallow the $k$ into the $C'$ in front of $e^{36k^2}$.
\end{proof}

\begin{lem} \label{lem:N-tensor-low-rank}
Let $\mathcal T = \{ T \in \symmetric_{\ell}(\reals^d): \norm{T}_F \le \epsilon, \rank_s(T) \le r\}$.
Then, for any $\delta\le \epsilon/2$,
\begin{equation}
N(\delta, \mathcal T, \norm{\cdot}_F) \le \({\frac{\added{c} \ell \epsilon}{\delta}}\)^{dr}\({\frac{\added{c}\epsilon}{\delta}}\)^{r^{\ell}},
\end{equation}
for some absolute constant \added{$c$}.
\end{lem}
\begin{proof}
For any $T\in \mathcal T$, $\rank_s(T)\le r$. Thus
$T= \sum_{i=1}^r a_i v_i^{\otimes \ell}$ for some $a_i \in \reals$ and $v_i \in S^{d-1}$. 
Furthermore, $\norm{T}_F \leq \epsilon$. Ideally, if the coefficients satisfied $|a_i| \leq \epsilon$ for all $i$, then we could cover the $r$-dimensional $\epsilon$-hypercube with an $\frac{\epsilon}{2}$-covering, which, combined with a $\frac{1}{2}$-covering of the unit sphere that covers the unit vectors $v_i$'s, constitutes a desired covering for the tensor.
Unfortunately the coefficients $a_i$'s need not be small due to the possible cancellation between the rank-one components (consider the counterexample of $0=v^{\otimes \ell}-v^{\otimes \ell}$). Next, to construct the desired covering we turn to the Tucker decomposition of the tensor $T$. 

Let $u=(u_1,\dots,u_r)$ be an orthonormal basis for the subspace spanned by $(v_1,\dots,v_r)$. In particular, let $v_i=\sum_{j=1}^r b_{ij} u_j$. Then 
\begin{equation}
T = \sum_{\bfj=(j_1,\dots,j_\ell)\in [r]^\ell} \alpha_{\bfj} \underbrace{u_{j_1}\otimes \dots \otimes u_{j_\ell}}_{\triangleq u_{\bfj}},
\label{eq:tucker}
\end{equation}
where $\alpha_{\bfj} = \sum_{i=1}^r a_i b_{ij_1} \cdots b_{ij_\ell}$. In tensor notation, 
$T$ admits the following \emph{Tucker decomposition}
\begin{equation}
T = \alpha \times_1 U \cdots \times_\ell U 
\end{equation}
where the symmetric tensor $\alpha=(\alpha_\bfj)\in \symmetric_{\ell}(\reals^r)$ is called the core tensor and $U$ is a $r \times d$ matrix whose rows are given by $u_1,\ldots,u_r$.

Due to the orthonormality of $(u_1,\dots,u_r)$, we have for any $\bfj,\bfj' \in[r]^\ell$,
\begin{equation}
\inprod{u_{\bfj}}{u_{\bfj'}} = \prod_{i=1}^\ell \inprod{u_{j_i}}{u_{j'_i}} = \indc{\bfj=\bfj'}.
\end{equation}
Hence we conclude from \prettyref{eq:tucker} that 
\begin{equation}
\norm{\alpha}_F = \norm{T}_F.
\label{eq:tucker-norm}
\end{equation}
In particular $\norm{\alpha}_F \le \epsilon$. Therefore,
\begin{equation}
\mathcal T \subseteq \mathcal T' \triangleq \sth{ T= \sum_{\bfj\in [r]^{\ell}} \alpha_\bfj u_{j_1}\otimes \dots \otimes u_{j_{\ell}}: \norm{\alpha}_F\le \epsilon,\inprod{u_i}{u_j} = \indc{i=j} }.
\end{equation}

Let $\tilde A$ be a $\frac{\delta}{2}$-covering of $\{\alpha\in \symmetric_\ell(\reals^r): \norm{\alpha}_F\le \epsilon\}$ under $\norm{\cdot}_F$ of size $(\frac{\added{c}\epsilon}{\delta})^{r^{\ell}}$ for some absolute constant $c$;
let $\tilde B$ be a $\frac{\delta}{2\ell \epsilon}$-covering of $\{(u_1,\dots,u_r): \inprod{u_i}{u_j}=\indc{i=j}\}$ under the maximum of column norms of size $(\frac{\added{c}\ell \epsilon}{\delta})^{dr}$. Let $\tilde{\mathcal T'} = \{\sum_{\bfj\in [r]^{\ell}} \tilde \alpha_\bfj \tilde u_{j_1} \otimes \dots \otimes \tilde u_{j_{\ell}}: \tilde \alpha \in \tilde A, \tilde u \in \tilde B\}$.
Next we verify the covering property.

For any $T \in \mathcal T'$, there exists $\tilde T \in \tilde{\mathcal T'}$ such that $\norm{\alpha-\tilde\alpha}_F\le \frac{\delta}{2}$ and $\max_{i\le r} \norm{u_i-\tilde u_i}\le \frac{\delta}{2\ell \epsilon}$.
Then, by the triangle inequality,
\begin{equation}
\norm{T-\tilde T}_F 
\le \sum_\bfj |\alpha_\bfj| \norm{u_{j_1}\otimes \dots \otimes u_{j_\ell}-\tilde u_{j_1}\otimes \dots \otimes \tilde u_{j_\ell} }_F
+ \norm{\sum_\bfj (\alpha_\bfj-\tilde \alpha_\bfj) \tilde u_{j_1}\otimes \dots \otimes \tilde u_{j_\ell} }_F.
%{\sum_\bfj |\alpha_\bfj-\tilde \alpha_\bfj| \norm{\tilde u_{j_1}\otimes \dots \otimes \tilde u_{j_\ell} }_F}
\end{equation}
The second term is at most $\norm{\alpha-\tilde \alpha}_F \le \delta/2$. For the first term, it follows from the triangle inequality that
\begin{equation}
\norm{u_{j_1}\otimes \dots \otimes u_{j_\ell}-\tilde u_{j_1}\otimes \dots \otimes \tilde u_{j_\ell}}_F
\le \sum_{i=1}^{\ell} \norm{u_{j_1}\otimes \dots \otimes (u_{j_i}-\tilde u_{j_i})\otimes \dots \otimes \tilde u_{j_\ell}}_F
\le \frac{\delta}{2 \epsilon}.
\end{equation}
Therefore, the first term is at most $\frac{\delta}{2 \epsilon} \norm{\alpha}_F\le \delta/2$.
\end{proof}

\subsection{Efficient proper density estimation} \label{ssec:density_est_efficient}
%The Le Cam-Birg\'e estimator is computationally intractable as it uses an exponential number of pairwise comparisons. 
%In this subsection, we analyze the theoretical guarantee of density estimation using the procedure for mixing distribution estimation in \prettyref{sec:mixing_estimation}.

To remedy the computational intractability of the Le Cam-Birg\'e estimator, in this subsection we adapt the procedure for mixing distribution estimation in \prettyref{sec:mixing_estimation} for density estimation.
Let $\hat \Gamma$ be the estimated mixing distribution as defined in \eqref{eq:estimator}, with the following modifications:
\begin{itemize}
	\item The grid size in \prettyref{eq:gridsize-mixing} is adjusted to
 to $\epsilon_n=n^{-1/2}$.
As such, in the set of $k$-atomic candidate distributions in \prettyref{eq:candidates}, $\calW$ denotes an $(\epsilon_n,\norm{\cdot}_1)$-covering of the probability simplex $\Delta^{k-1}$, and $\calA$ denotes an $(\epsilon_n,\norm{\cdot}_2)$-covering the $k$-dimensional ball $\{x\in\reals^k:\norm{x}_2\le R\}$.

\item In the determination of the best mixing distribution in \prettyref{eq:hatgamma-kdim}, instead of comparing the Wasserstein distance, we directly compare the projected moments on the directions over an $(\epsilon_n,\norm{\cdot}_2)$-covering $\calN$ of $S^{k-1}$:
\[
\hat\gamma=\arg\min_{\gamma'\in \calS}\max_{\theta\in\calN}
\max_{r\in[2k-1]}|m_r(\gamma_\theta')-m_r(\widehat{\gamma_\theta})|,
\]
where $\gamma'_\theta$ denotes the projection of $\gamma'$ onto the direction $\theta$ (recall \prettyref{eq:gamma_theta}). 
\end{itemize}
We then report $P_{\hat \Gamma}=\hat \Gamma * N(0,I_d)$ as a proper density estimate. 
By a similar analysis to \prettyref{rem:algorithm_time}, using those finer grids, the run time of the procedure increases to $n^{O(k)}$.
The next theorem provides a theoretical guarantee for this estimator in general $k$-GM model.
%the corresponding density $P_{\hat\Gamma}$.

\begin{theorem} \label{thm:density_est_efficient}
There exists an absolute constant $C$ such that, with probability $1-\delta$,
\[
H(P_{\Gamma},P_{\hat\Gamma})
\le C\sqrt{k}\pth{\frac{d+k^2\log(k/\delta)}{n}}^{1/4}+ \frac{(e^{Ck}\log(1/\delta))^{\frac{2k-1}{2}}}{\sqrt{n}} .
\]
\end{theorem}
\begin{proof}
Recall the notation $\Sigma, \hat \Sigma, \hat V, \hat H, \gamma$ defined in \prettyref{ssec:proof_main}.
% Recall that $\hat H = \hat V\hat V^\top$ defined in \prettyref{ssec:proof_main} is the projection matrix onto the estimated subspace.
By the triangle inequality, 
\begin{equation}
\label{eq:H-tri-ub}
H(P_{\Gamma},P_{\hat\Gamma})
\le H(P_{\Gamma},P_{\Gamma_{\hat H}})
+H(P_{\Gamma_{\hat H}},P_{\hat\Gamma})
\le H(P_{\Gamma},P_{\Gamma_{\hat H}})
+H(\gamma,\hat\gamma).
\end{equation}
For the first term, we have
\begin{equation}
\label{eq:H_sq_cov}
H^2(P_{\Gamma},P_{\Gamma_{\hat H}})
\le \KL(P_{\Gamma},P_{\Gamma_{\hat H}})
\stepa{\le} \frac{1}{2}W_2^2(\Gamma, \Gamma_{\hat H})
\stepb{\le} k\|\Sigma-\hat\Sigma \|_2,
\end{equation}
where 
(a) follows from the convexity of the \KL\ divergence (see \cite[Remark 5]{polyanskiy2015dissipation});
(b) applies \prettyref{lem:perturbation_psd}.

Next we analyze the second term in the right-hand side of \prettyref{eq:H-tri-ub} conditioning on $\hat V$.
By \prettyref{thm:H-M}, \prettyref{eq:T-norms}, and \prettyref{eq:mtensor-opnorm}, the Hellinger distance is upper bounded by the difference between the projected moments:
\begin{equation}
\label{eq:H-moments}
H(\gamma,\hat\gamma)
\le e^{Ck^2}\sup_{\theta\in S^{k-1}}\max_{r\in[2k-1]}|m_r(\gamma_\theta)-m_r(\hat\gamma_\theta)|.
\end{equation}
It follows from \prettyref{lem:covering-S} that
\begin{equation}
\label{eq:max-covering}
\min_{\gamma'\in \calS}\max_{\theta\in S^{k-1}}\max_{r\in[2k-1]}|m_r(\gamma_\theta')-m_r(\gamma_\theta)|
\le \frac{2kR^{2k-1}}{\sqrt{n}}.
\end{equation}
By \prettyref{eq:w1_sup_bound2} and \prettyref{eq:w1-hat-bound}, with probability $1-\delta/2$, 
\begin{equation}
\label{eq:max-widehat}
\max_{\theta\in\calN} \max_{r\in[2k-1]} | m_r(\widehat{\gamma_\theta})-m_r(\gamma_\theta)|
\le \frac{(Ck^2\log(k/\delta))^{\frac{2k-1}{2}}}{\sqrt{n}},
\end{equation}
for an absolute constant $C$.
Therefore, by \eqref{eq:max-covering} and \eqref{eq:max-widehat},
\begin{equation}
\label{eq:max-hat-gamma}    
\min_{\gamma'\in \calS}\max_{\theta\in\calN}\max_{r\in[2k-1]}|m_r(\gamma_\theta')-m_r(\widehat{\gamma_\theta})|
\le \frac{(C'k^2\log(k/\delta))^{\frac{2k-1}{2}} {+(C'k)^{4k}}}{\sqrt{n}},
\end{equation}
for an absolute constant $C'\ge C$. 
Note that the minimizer of \eqref{eq:max-hat-gamma} is our estimator $\hat\gamma$.
Consequently, combining \eqref{eq:max-widehat} and \eqref{eq:max-hat-gamma}, we obtain
\[
\max_{\theta\in\calN}\max_{r\in[2k-1]}|m_r(\hat\gamma_\theta)-m_r(\gamma_\theta)|
\le \frac{2\((C'k^2\log(k/\delta))^{\frac{2k-1}{2}}{+(C'k)^{4k}}\)}{\sqrt{n}}.
\]
Then it follows from \prettyref{lem:covering-N} that
\[
\sup_{\theta\in S^{k-1}}\max_{r\in[2k-1]}|m_r(\gamma_\theta)-m_r(\hat\gamma_\theta)|
\le \frac{(C''k^2\log(k/\delta))^{\frac{2k-1}{2}}{+(C''k)^{4k}}}{\sqrt{n}},
\]
for an absolute constant $C''$.
Applying \eqref{eq:H-moments} yields that, with probability $1-\delta/2$,
\begin{equation}
\label{eq:H-gamma}
H(\gamma,\hat\gamma)
\le \frac{(e^{C'''k}\log(1/\delta))^{\frac{2k-1}{2}}}{\sqrt{n}},
\end{equation}
for an absolute constant $C'''$.
We conclude the theorem by applying \prettyref{lem:cov_bound}, \eqref{eq:H_sq_cov}, and \eqref{eq:H-gamma} to \eqref{eq:H-tri-ub}.
\end{proof}

Compared with the optimal parametric rate $O_k(\sqrt{d/n})$ in \prettyref{thm:density_est}, the rate $O_k((d/n)^{1/4})$ in \prettyref{thm:density_est_efficient} is suboptimal. 
It turns out that, for the special case of 2-GMs, we can achieve the optimal rate using the same procedure with an extra centering step.
% Determining whether the procedure is optimal for general $k$ remains an open problem.
Specifically, using the first half of observations $\{X_1,\dots,X_n\}$, we compute the sample mean and covariance matrix by
\[
\hat\mu = \frac{1}{n}\sum_{i=1}^n X_i, \qquad S=\frac{1}{n}\sum_{i=1}^n(X_i-\hat\mu)(X_i-\hat \mu)^\top - I_d.
\]
Let $\hat u\in\reals^d$ be the top eigenvector of the sample covariance matrix $S$.
Then we center and project the second half of the observations by $x_i=\Iprod{\hat u}{X_{i+n}-\hat\mu}$ for $i\in[n]$, which reduces the problem to one dimension. Then we apply the one-dimensional DMM algorithm with $x_1,\dots,x_n$ and obtain $\hat\gamma=\sum_{i=1}^2\hat w_i\delta_{\hat\theta_i}$.
Finally, we report $P_{\hat \Gamma}$ with the mixing distribution 
\[
\hat\Gamma=\sum_{i=1}^2 \hat w_i \delta_{\hat\theta_i\hat u+\hat \mu }.
\]
The next result shows the optimality of $P_{\hat \Gamma}$.
%However, it remains open whether the same procedure is optimal for general $k$.
% Note that, conditioning on $\hat\mu$ and $S$, we have $x_i\iiddistr \tilde\gamma=\sum_{i=1}^2 w_i\delta_{\hat u^\top(\mu_i-\hat\mu)}$.

\begin{theorem} \label{thm:rate_density_2gm}
With probability at least $1-\delta$,
\[
H(P_\Gamma,P_{\hat\Gamma})\lesssim \sqrt{\frac{d+\log(1/\delta)}{n}}.
\]
\end{theorem}
\begin{proof}
Let $\Gamma=w_1\delta_{\mu_1}+w_2\delta_{\mu_2}$, where $\|\mu_i\|\leq R, i=1,2$.
Denote the population mean and covariance matrix by $\mu=\Expect_\Gamma[U]$ %=w_1\mu_1+w_2\mu_2$ 
and $\Xi=\Expect_\Gamma(U-\mu)(U-\mu)^\top$, respectively.
We have the following standard results for the sample mean and covariance matrix (see, \eg, \cite[Eq.~(1.1)]{LM2019sub} and \prettyref{lem:cov_bound}):
\begin{equation}
\label{eq:mu-cov}
\Norm{\hat\mu-\mu}_2, ~ \Norm{S-\Xi}_2 \leq  C(R) \sqrt{\frac{d+\log(1/\delta)}{n}},
\end{equation}
with probability at least $1-\delta/2$.
% Sample covariance
% \[
% \Norm{S-\Xi}
% \le \Norm{\hat \Sigma - \Sigma} + \Norm{\hat\mu\hat\mu^\top - \mu\mu^\top}
% \]
% \[
% \Norm{\hat\mu\hat\mu^\top - \mu\mu^\top}
% \le (\Norm{\hat\mu}_2+\Norm{\mu}_2)\Norm{\hat\mu-\mu}_2...
% \]
Let $\Gamma'=\sum_{i=1}^2w_i\delta_{\Iprod{\hat u}{\mu_i-\mu}\hat u+\mu}$, 
and $\tilde\Gamma = \sum_{i=1}^2w_i\delta_{\Iprod{\hat u}{\mu_i-\hat\mu}\hat u+\hat\mu}$.
By the triangle inequality, 
\begin{equation}
\label{eq:H-tri-3terms}
H(P_\Gamma,P_{\hat\Gamma})
\le H(P_\Gamma,P_{\Gamma'}) + H(P_{\Gamma'},P_{\tilde\Gamma}) + H(P_{\tilde\Gamma},P_{\hat\Gamma}).
\end{equation}
We upper bound three terms separately conditioning on $\hat\mu$ and $\hat u$.
For the second term of \eqref{eq:H-tri-3terms}, applying the convexity of the squared Hellinger distance yields that
\begin{align}
H^2(P_{\tilde\Gamma},P_{\Gamma'})
&\le \sum_{i=1}^2 w_i H^2\pth{N(\hat u\hat u^\top (\mu_i-\mu)+\mu,I_d),N(\hat u\hat u^\top (\mu_i-\hat\mu)+\hat\mu,I_d)}\nonumber\\
& = \sum_{i=1}^2 w_i \left(2-2e^{-\frac{\Norm{(I-\hat u\hat u^\top)(\mu-\hat\mu)}_2^2}{8} }\right)
% \le \sum_{i=1}^2 w_i \frac{\Norm{(I-\hat u\hat u^\top)(\mu-\hat\mu)}_2}{2}
\le \Norm{\mu-\hat\mu}_2^2,\label{eq:H-tri-3terms-2}
\end{align}
where we used $e^x\ge 1+x$.
For the third term \eqref{eq:H-tri-3terms}, note that 
conditioned on the $(\hat u,\hat \mu)$,
 $x_i\iiddistr P_{\tilde\gamma}$ where $\tilde\gamma=\sum_{i=1}^2 w_i\delta_{\Iprod{\hat u}{\mu_i-\hat\mu}}$. Note that $\hat\Gamma$ and $\tilde\Gamma$ are supported on the same affine subspace 
$\{\theta\hat u+\hat\mu: \theta\in\reals\}$. Thus
\begin{equation}
\label{eq:H-tri-3terms-3}
H(P_{\hat\Gamma},P_{\tilde\Gamma}) 
= H(P_{\hat\gamma}, P_{\tilde\gamma})
\lesssim \sqrt{\frac{\log(1/\delta)}{n}},
\end{equation}
with probability at least $1-\delta/2$, where the inequality follows from the statistical guarantee of the DMM algorithm in \cite[Theorem 3]{WY18}.
% \prettyref{lem:w1_sup_bound}.
% \nbwu{\prettyref{lem:w1_sup_bound} is $W_1$ guarantee. We need Hellinger. Maybe add it to \prettyref{lem:w1_sup_bound}?}.
% \cite[Theorem 3]{WY18}. \nb{Or from~\prettyref{lem:w1_sup_bound}.}

It remains to upper bound the first term of \eqref{eq:H-tri-3terms}.
Denote the centered version of $\Gamma,\Gamma'$ by $\pi,\pi'$. 
Using $\mu=w_1\mu_1+w_2\mu_2$ and $w_1+w_2=1$, we may write
$\pi=w_1\delta_{\lambda w_2 u} + w_2\delta_{-\lambda w_1 u}$, 
$\pi'=w_1\delta_{\lambda w_2 \hat u\hat u^\top u} + w_2\delta_{-\lambda w_1 \hat u\hat u^\top u}$,
where
$\lambda \triangleq \Norm{\mu_1-\mu_2}_2$, and $u\triangleq \frac{\mu_1-\mu_2}{\Norm{\mu_1-\mu_2}_2}$.
Then
\begin{equation}
\label{eq:H-tri-3terms-1}
H(P_\Gamma,P_{\Gamma'})
= H(P_\pi, P_{\pi'}).
\end{equation}
By \prettyref{thm:H-M}, it is equivalent to upper bound the difference between the first three moment tensors. 
Both $\pi$ and $\pi'$ have zero mean. 
Using $w_2^2 - w_1^2 = w_2 - w_1$, their covariance matrices and the third-order moment tensors are found to be
\begin{align*}
&\Xi  = \Expect_\pi[UU^\top]=\lambda^2 w_1w_2uu^\top,\quad
&&T = \Expect_\pi[U^{\otimes 3}]= \lambda^3 w_1w_2(w_1-w_2)u^{\otimes 3}, \\ 
&\Xi' = \Expect_{\pi'}[UU^\top]=\lambda^2 w_1w_2\Iprod{u}{\hat u}^2\hat u\hat u^\top, \quad
&&T' = \Expect_{\pi'}[U^{\otimes 3}]=\lambda^3 w_1w_2(w_1-w_2) \Iprod{u}{\hat u}^3 \hat u^{\otimes 3}. 
% \nb{\Iprod{u}{\hat u}^3 \hat u^{\otimes 3} ?}
\end{align*}
Applying \prettyref{thm:H-M} and \prettyref{lem:sig-T-cov} below yields that $H(P_\pi, P_{\pi'})\lesssim \Norm{S-\Xi}$.
The proof is completed by combining \eqref{eq:mu-cov} -- \eqref{eq:H-tri-3terms-1}.
\end{proof}

\begin{lemma}
\label{lem:sig-T-cov}
\[
\Norm{\Xi-\tilde\Xi}_F + \Norm{T-T'}_F\lesssim \Norm{S-\Xi}.
\]
\end{lemma}
\begin{proof}
Let $\sigma = \lambda^2 w_1 w_2$
%, which is the top eigenvalue of $\Xi$,
% \nb{What? $\sigma^2$ is a scalar...do you just mean $\sigma = \lambda^2 w_1 w_2$? And can we mention here that $\sigma$ is the eigenvalue of $\Xi$.} 
and $\cos\theta = \Iprod{u}{\hat u}$.
Since $\lambda \lesssim 1$, we have
\begin{align*}
&\Norm{\Xi-\tilde\Xi}_F^2=\sigma^2(1-\cos^4\theta)\lesssim \sigma^2 \sin^2\theta,\\
&\Norm{T-T'}_F^2 = \lambda^2\sigma^2(w_1-w_2)^2(1-\cos^6\theta) \lesssim  \sigma^2 \sin^2\theta.
\end{align*}
% \nb{I see the same conclusion if the changes above are made. Although I didn't see exact equality in the first equality. But the same upper bound holds.}
Since $\Xi=\sigma uu^\top$, by the Davis-Kahan theorem \cite{Davis_Kahan_1970}, $\sin \theta \lesssim \frac{\Norm{S-\Xi}}{\sigma}$, completing the proof.
\end{proof}

\subsection{Connection to mixing distribution estimation}
\label{sec:de-connection}

The next result shows that optimal estimation of the mixing distribution can be reduced to that of the mixture density, both statistically and computationally, provided that the density estimate is proper (a valid $k$-GM). Note that this does not mean an optimal density estimate $P_{\hat \Gamma}$ automatically yields an optimal estimator of the mixing distribution $\hat \Gamma$ for \prettyref{thm:mixing_est}. Instead, we rely on an intermediate step that allows us to estimate the appropriate subspace and then perform density estimation in this low-dimensional space. 
%Note that the estimator in~\prettyref{thm:mixing_est_density} is not obtained directly from the proper density estimator of~\prettyref{thm:density_est}, i.e., given $P_{\hat \Gamma}$ satisfying~\prettyref{eq:result_density}, we do not simply use $\hat \Gamma$. We instead rely on an intermediate step that allows us to estimate the appropriate subspace, and we then perform density estimation in that low-dimensional space. Thus this estimator, like that in~\prettyref{thm:mixing_est}, also relies on decomposing the problem into two sub-problems: subspace estimation and low-dimensional mixing distribution estimation.

\begin{thm} \label{thm:mixing_est_density}
Suppose that for each $d\in\naturals$, there exists a proper density estimator $\hat P=\hat P(X_1,\ldots,X_n)$, such that for every  $\Gamma \in \mc G_{k, d}$ and 
$X_1,\ldots,X_n\iiddistr P_\Gamma$,
\begin{align}
\mb E H(\added{\hat P}, P_\Gamma) \le c_k (d/n)^{1/2},
\label{eq:hellinger_bound_dim}
\end{align}
for some constant $c_k$.
Then there is an estimator $\hat \Gamma$ of the mixing distribution $\Gamma$ and \added{a positive constant $C$} such that
\begin{align}
\added{
\mb E W_1(\hat \Gamma, \Gamma) \le  \(  (Ck)^{k/2} \sqrt{c_k} \(\frac{d}{n} \)^{1/4} +  C k^{5} c_k^{\frac{1}{2k-1}} \(\frac{1}{n} \)^{\frac{1}{4k-2}} \). \label{eq:lk_algo_general_bound}} 
\end{align}
% \begin{align*}
% \replaced{}
% {
% \mb E W_1(\hat \Gamma, \Gamma) \le  \(  (2kC_k c_k)^{1/2} \(\frac{d}{n} \)^{1/4} +  C' k^{7/2} (C_k c_k)^{1/2k-1} \(\frac{k}{n} \)^{\frac{1}{4k-2}} \)
% }
% \end{align*}
% where \replaced{$(Ck)^k$}{$C_k$} is the constant $c$ in~\prettyref{lem:l2_moments_lb}.
\end{thm}

\begin{proof} [Proof of~\prettyref{thm:mixing_est_density}]
We first construct the estimator $\hat\Gamma$ using $X_1,\dots,X_{2n} \iiddistr P_\Gamma$.

\added{
Let $\hat P \in \mc P_{k, d}$ be the proper mixture density estimator from $\{X_i\}_{i\le n}$ satisfying}
\begin{equation}
\label{eq:guarantee-d}
\added{\Expect H(\hat P, P_{\Gamma}) \le c_k \sqrt{d/n},}
\end{equation}
\added{for a positive constant $c_k$, as guaranteed by \prettyref{eq:hellinger_bound_dim}. Since $\hat P$ is a proper estimator, it can be written $\hat P = \hat \Gamma' * N(0, I_d)$ for some $\hat \Gamma' \in \mc G_{k, d}$. }

%Let $\hat \Gamma' \in \mc G_{k, d}$ be the estimator from $\{X_i\}_{i\le n}$ and let $\hat P \triangleq \hat \Gamma' * N(0, I_d)$ satisfying 
%\begin{equation} \label{eq:guarantee-d}
%\Expect H(P_{\hat \Gamma'}, P_{\Gamma}) \le c_k \sqrt{d/n},
%\end{equation}
%for a positive constant $c_k$, as guaranteed by \prettyref{eq:hellinger_bound_dim}.

Let $\hat V \in \reals^{d \times k}$ be a matrix whose columns form an orthonormal basis for the space spanned by the atoms of $\hat \Gamma'$, $\hat H=\hat V\hat V^\top$, and $\gamma=\Gamma_{\hat V}$.
Note that conditioned on $\hat V$, $\{ \hat V^{\top} X_i \}_{i = n + 1, \hdots, 2n}$ is an i.i.d.~sample drawn from the $k$-GM $P_{\gamma}$. Invoking \prettyref{eq:hellinger_bound_dim} again, there exists an estimator $\hat \gamma = \sum_{j = 1}^k \hat w_j \delta_{\hat \psi_j} \in \calG_{k,k}$ such that
\begin{equation}
\label{eq:guarantee-k}
\Expect H(P_{\hat\gamma}, P_{\gamma}) \le c_k \sqrt{k/n}.
\end{equation}
We will show that $\hat \Gamma\triangleq\hat \gamma_{\hat V^\top} = \sum_{j = 1}^k \hat w_j \delta_{\hat V \hat \psi_j}$ achieves the desired rate \prettyref{eq:lk_algo_general_bound}. 
Recall from \prettyref{eq:w1-triangle} the risk decomposition:
\begin{equation}
\label{eq:w1-triangle-density}
W_1(\Gamma,\hat \Gamma)
\le W_1(\Gamma,\Gamma_{\hat H})+W_1(\gamma,\hat \gamma).
\end{equation}
Let $\Sigma=\Expect_{U\sim \Gamma}[UU^\top]$ and $\hat\Sigma = \Expect_{U\sim\hat\Gamma'}[UU^\top]$ whose ranks are at most $k$.
Then $\hat H$ is the projection matrix onto the space spanned by the top $k$ eigenvectors of $\hat \Sigma$.
It follows from \prettyref{lem:perturbation_psd} and the Cauchy-Schwarz inequality that $W_1(\Gamma,\Gamma_{\hat H})\le \sqrt{2k\Norm{\Sigma-\hat\Sigma}_2}$. 
% \added{Let $C_k \triangleq (Ck)^k$ be the constant in~\prettyref{lem:l2_moments_lb}.}
By \prettyref{lem:l2_moments_lb} and the data processing inequality of the Hellinger distance,
\[
\Norm{\Sigma-\hat\Sigma}_2
=\sup_{\theta \in S^{d-1}} |m_2(\Gamma_\theta)-m_2(\hat \Gamma_\theta')|
\le C_k\sup_{\theta \in S^{d-1}}H(P_{\Gamma_{\theta}},P_{\hat \Gamma_{\theta}'})
\le C_k H(P_{\Gamma},P_{\hat \Gamma'}),
\]
\added{where $C_k = (Ck)^k$ for a constant $C$.}
Therefore, by \prettyref{eq:guarantee-d}, we obtain that
\begin{equation}
\label{eq:w1-high-d-density}
\Expect W_1(\Gamma,\Gamma_{\hat H})\le \sqrt{2k C_k c_k \(\frac{d}{n} \)^{1/2}}.	
\end{equation}
We condition on $\hat V$ to analyze the second term on the right-hand side of \prettyref{eq:w1-triangle-density}.
By Lemmas \ref{lem:w1_dim_reduction} and \ref{lem:w1_moments_ub}, there is a constant $C'$ such that
\[
W_1(\gamma,\hat \gamma)
\le k^{5/2}\sup_{\theta\in S^{k-1}}W_1(\gamma_{\theta},\hat \gamma_{\theta})
\le \added{C' k^{7/2}} \sup_{\theta\in S^{k-1},r\in[2k-1]}|m_r(\gamma_\theta) - m_r(\hat \gamma_\theta)|^{\frac{1}{2k-1}}.
\]
Again, by \prettyref{lem:l2_moments_lb} and the data processing inequality, for any $\theta\in S^{k-1}$ and $r\in [2k-1]$,
\[
|m_r(\gamma_\theta) - m_r(\hat \gamma_\theta)|
\le C_k H\(P_{\hat \gamma_\theta}, P_{\gamma_\theta}\) 
\le C_k H\(P_{\hat \gamma}, P_{\gamma}\).
\]
Therefore, by \prettyref{eq:guarantee-k}, we obtain that
\begin{equation}
\label{eq:w1-low-d-density}
\Expect W_1(\gamma,\hat \gamma)
\le \added{C' k^{7/2}} \( C_k c_k \(\frac{k}{n} \)^{1/2} \)^{\frac{1}{2k-1}}.
\end{equation}
The conclusion follows by applying \prettyref{eq:w1-high-d-density} and \prettyref{eq:w1-low-d-density} in \prettyref{eq:w1-triangle-density}.
\end{proof}

At the crux of the above proof is the following key inequality for $k$-GMs in $d$ dimensions:
\begin{equation}
W_1(\gamma, \hat \gamma) \lesssim_{k,d} H(P_\gamma, P_{\hat \gamma})^{1/(2k-1)}, \label{eq:hellinger_to_w1}
\end{equation}
which we apply after a dimension reduction step. The proof of~\prettyref{eq:hellinger_to_w1} relies on two crucial facts:
for one-dimensional $k$-atomic distributions $\gamma, \hat \gamma$,
\begin{equation}
W_1(\gamma, \hat \gamma) \lesssim_k \max_{\ell \in [2k-1]} |m_\ell(\gamma) - m_\ell(\hat \gamma)|^{1/(2k-1)}, \label{eq:w1_to_moments}
\end{equation}
and
\begin{equation}
\max_{\ell \in [2k-1]} |m_\ell(\gamma) - m_\ell(\hat \gamma)| \lesssim_k H(P_\gamma, P_{\hat \gamma}). \label{eq:hellinger_to_moments}
\end{equation}
Then \prettyref{eq:hellinger_to_w1} immediately follows from \prettyref{lem:w1_dim_reduction} and \prettyref{eq:mtensor-opnorm}.

Relationships similar to~\prettyref{eq:hellinger_to_w1} are found elsewhere in the literature on mixture models, e.g.,~\cite{Chen_1995,Ho_Nguyen_2016_annals,Ho_Nguyen_2016_ejs,HK2015}, where they are commonly used to translate a density estimation guarantee into one for mixing distributions.
For example, \cite[Proposition 2.2(b)]{Ho_Nguyen_2016_annals} showed the non-uniform bound $W_r^r(\gamma, \hat \gamma) \leq C(\gamma) H(P_\gamma, P_{\hat \gamma})$, where $r$ is a parameter that depends on the level of overfitting in the model; see \cite{Ho_Nguyen_2016_annals,Ho_Nguyen_2016_ejs} for more results for other models such as location-scale Gaussian mixtures. For uniform bound similar to \prettyref{eq:hellinger_to_w1}, 
 \cite[Theorem 6.3]{HK2015} showed $W_{2k-1}^{2k-1}(\gamma, \hat \gamma) \lesssim \norm{p_\gamma - p_{\hat \gamma}}_\infty$ in one dimension.

Conversely, distances between mixtures can also be bounded by transportation distances between mixing distributions, e.g., the middle inequality in 
\prettyref{eq:H_sq_cov} for the KL divergence. 
Total variation inequality of the form $\TV(F * \gamma, F * \hat \gamma) \lesssim W_1(\gamma, \hat \gamma)$ for arbitrary $\gamma,\gamma'$ are shown in
\cite[Proposition 7]{polyanskiy2015dissipation} or \cite[Proposition $5.3$]{Bontemps_Gadat_2014}, provided that $F$ has bounded density.
See also \cite[Theorem 3.2(c)]{Ho_Nguyen_2016_ejs} and \cite[Proposition 2.2(a)]{Ho_Nguyen_2016_annals} for results along this line.

\section{Numerical studies} \label{sec:numerical}

We now present numerical results. We compare the estimator~\prettyref{eq:estimator} to the classical EM algorithm. The EM algorithm is guaranteed only to converge to a local optimum (and very slowly without separation conditions), and its performance depends heavily on the initialization chosen. It moreover accesses all data points on each iteration.
%, and it may require many iterations to converge. 
%These weaknesses mean that even a grid-search based estimator like~\prettyref{eq:estimator} can surpass it in some cases.
 As such, although the estimator~\prettyref{eq:estimator} (henceforth referred to simply as DMM) relies on grid search, it can be competitive with or even surpass the EM algorithm both statistically and computationally in some cases, as our experiments show.

All simulations are run in Python. The DMM algorithm relies on the CVXPY~\cite{Diamond_Boyd_2016} and CVXOPT~\cite{Andersen_etal_2013} packages; see Section $6$ of~\cite{WY18} for more details on the implementation of DMM. We also use the Python Optimal Transport package~\cite{Flamary_2017} to compute the $d$-dimensional $1$-Wasserstein distance. 

In all experiments, we let $\sigma = 1$ and $d = 100$, and we let $n$ range from $10,000$ to $200,000$ in increments of $10,000$. For each model and each value of $n$, we run $10$ repeated experiments; we plot the mean error and standard deviation of the error in the figures. We initialize EM randomly, and our stopping criterion for the EM algorithm is either after $1000$ iterations or once the relative change in log likelihood is below $10^{-6}$. For the dimension reduction step in the computation of~\prettyref{eq:estimator}, we first center our data, then do the projection using the top $k - 1$ singular vectors.
 %of $\hat \Sigma $ in \prettyref{eq:hatSigma}. 
Thus when $k = 2$, we project onto a one-dimensional subspace and only run DMM once, so the grid search of~\prettyref{algo:gmm_est} is never invoked. Sample splitting is used for the estimator~\prettyref{eq:estimator} for purposes of analysis only; in the actual experiments, we do not sample split.

When $k = 3$, we project the data to a $2$-dimensional subspace after centering. In this case, we need to choose $\mc W, \mc N$, the $\epsilon_{n, k}$-nets on the simplex $\Delta^{k-1}$ and on the unit sphere $S^{k-2}$, respectively. Here $\mc W$ is chosen by discretizing the probabilities and 
%Note that the exponent here is $k - 2 = 1$ because of the initial centering and projection to $(k-1)$-dimensional space. 
%We cannot form an exactly minimal $\epsilon_{n, k}$-net on $S^1$, but there are many ways to form a grid on the unit sphere that are good enough for practical purposes and that have size $(C_k/\epsilon_{n,k})^{k-2}$. 
$\mc N$ is formed by by gridding the angles $\alpha \in [-\pi, \pi]$ and using the points $(\cos \alpha, \sin \alpha)$. 
Note that here, $|\mc W| \le (C_1/\epsilon_{n, k})^{k-1}, |\mc N| \le (C_2/\epsilon_{n, k})^{k-2}$. For example, when $n = 10000$, $1/\epsilon_{n, k} \approx  3$. In the experiments in~\prettyref{fig:k3}, we choose $C_1 = 1, C_2 = 2$.

%In our experience, with a coarser grid $\mc N$ our algorithm can still achieve fairly high accuracy in the well-separated models while gaining some speed.
%$\mc W$ needs to be large enough to contain reasonable weights, and a grid of size $9$ in this case is large enough. The size of $\mc N$ has a major impact on the time of~\prettyref{algo:gmm_est}; with $C_2 = 4$, the size is $12$. 

In~\prettyref{fig:k2}, we compare the performance on the symmetric $2$-GM, where the sample is drawn from the distribution $\frac{1}{2} N(\mu, I_d) + \frac{1}{2} N(-\mu, I_d)$. For~\prettyref{fig:k2a}, $\mu = 0$, i.e., the components completely overlap. And for~\prettyref{fig:k2b} and~\prettyref{fig:k2c}, $\mu$ is uniformly drawn from the sphere of radius $1$ and $2$, respectively. In~\prettyref{fig:k2d}, the model is $P_\Gamma = \frac{1}{4} N(\mu, I_d) + \frac{3}{4} N(-\mu, I_d)$ where $\mu$ is drawn from the sphere of radius $2$. 
%That is, the model is the same as the one used in~\prettyref{fig:k2c} except that the weights are uneven. We still run PCA and DMM without any grid search.
Our algorithm and EM perform similarly for the model with overlapping components; our algorithm is more accurate than EM in the model where $\norm{\mu}_2 = 1$, but EM improves as the model components become more separated. There is little difference in the performance of either algorithm in the uneven weights scenario. 
%In terms of running time, computing~\prettyref{eq:estimator} takes about the same time as EM for smaller values of $n$, but EM slows much more as $n$ increases since it accesses all the observations on each iteration. For the largest sample size in the experiments, computing~\prettyref{eq:estimator} is about 6 times faster than EM.

\begin{figure}[ht]
	\centering
	\subfigure[$\mu=0$]%
	{\label{fig:k2a} \includegraphics[width=0.4\columnwidth]{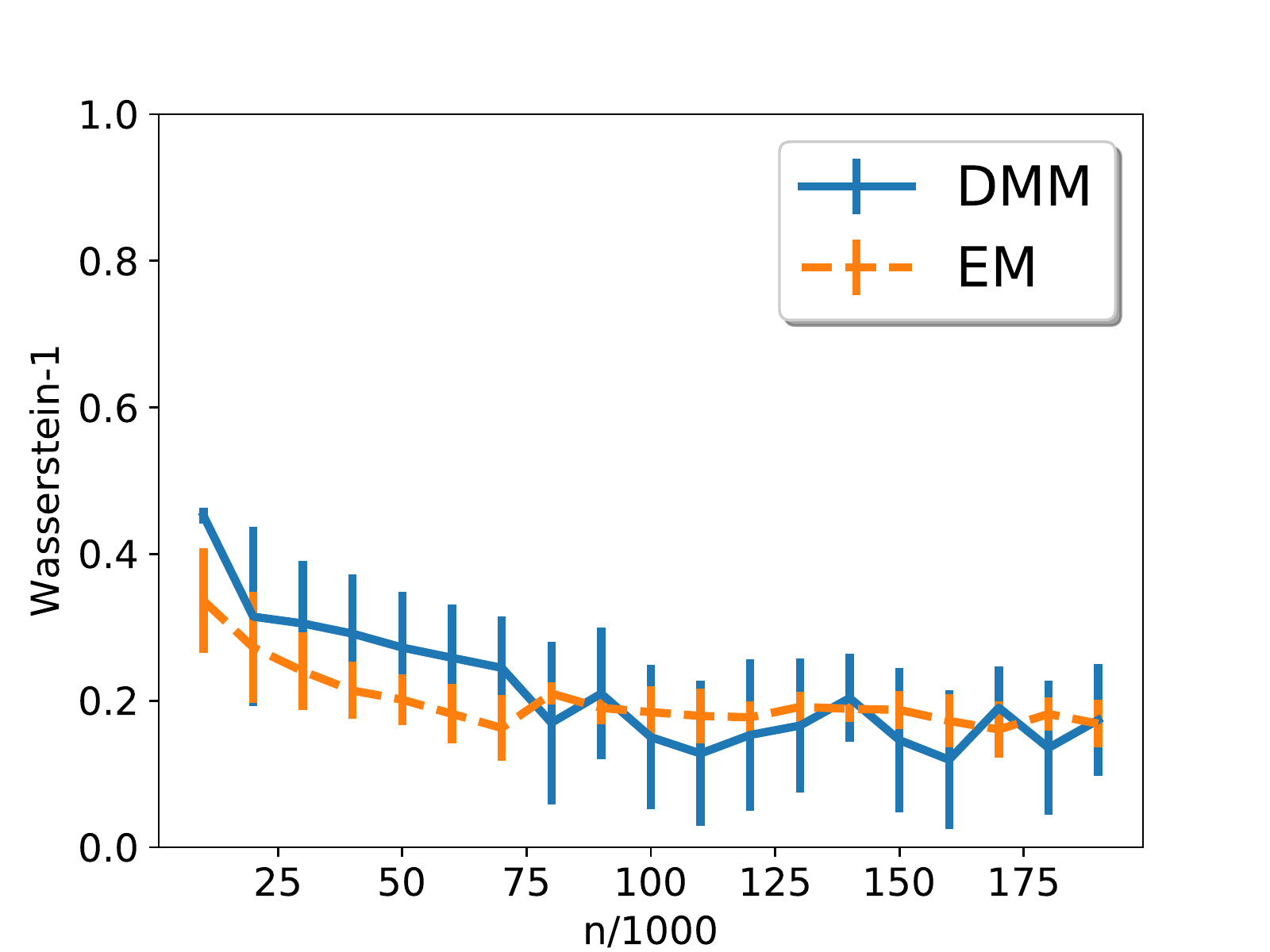}}
	\subfigure[$\|\mu\|=1$]%
	{\label{fig:k2b} \includegraphics[width=0.4\columnwidth]{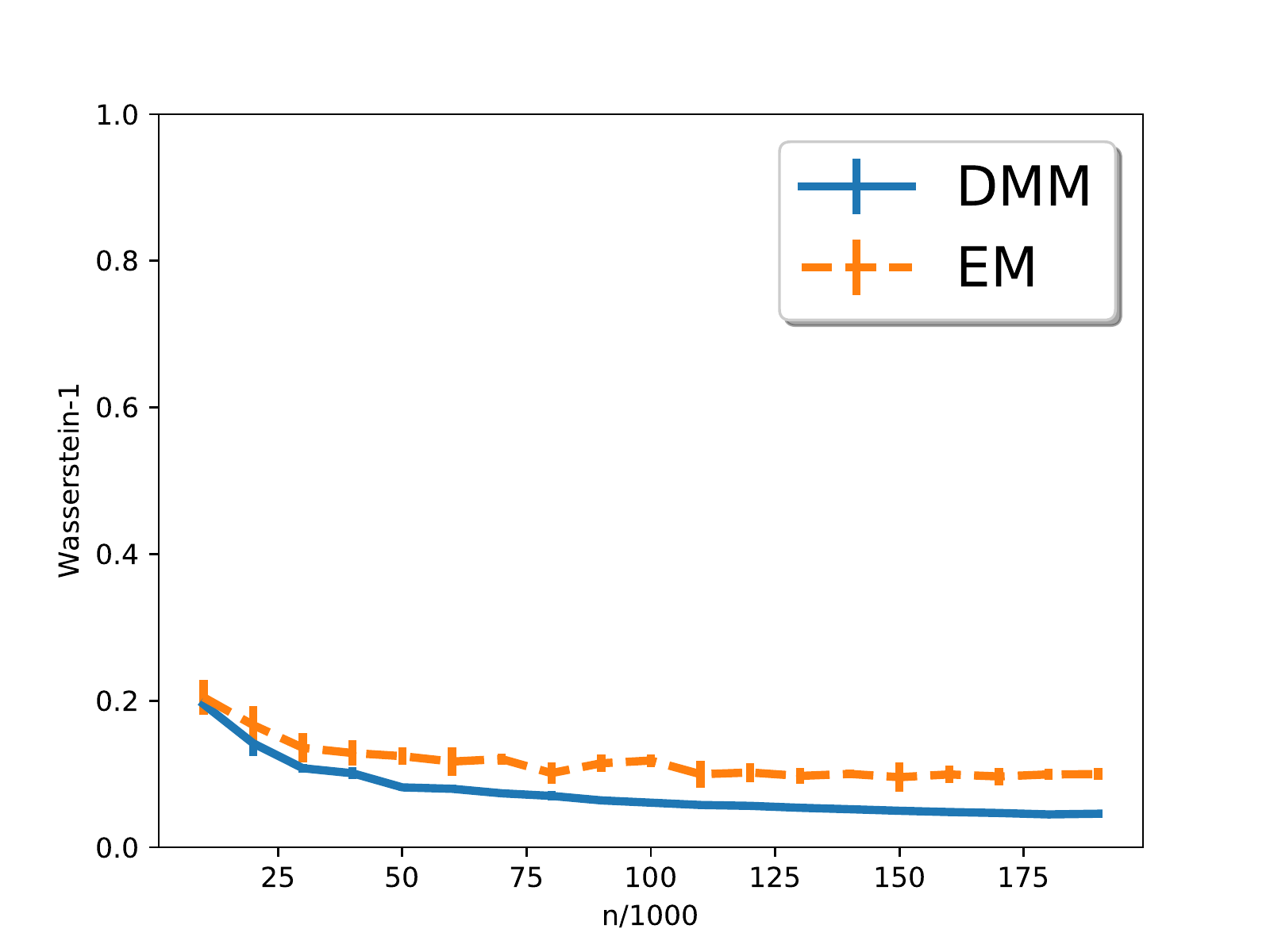}}
	\subfigure[$\|\mu\|=2$]%
	{\label{fig:k2c} \includegraphics[width=0.4\columnwidth]{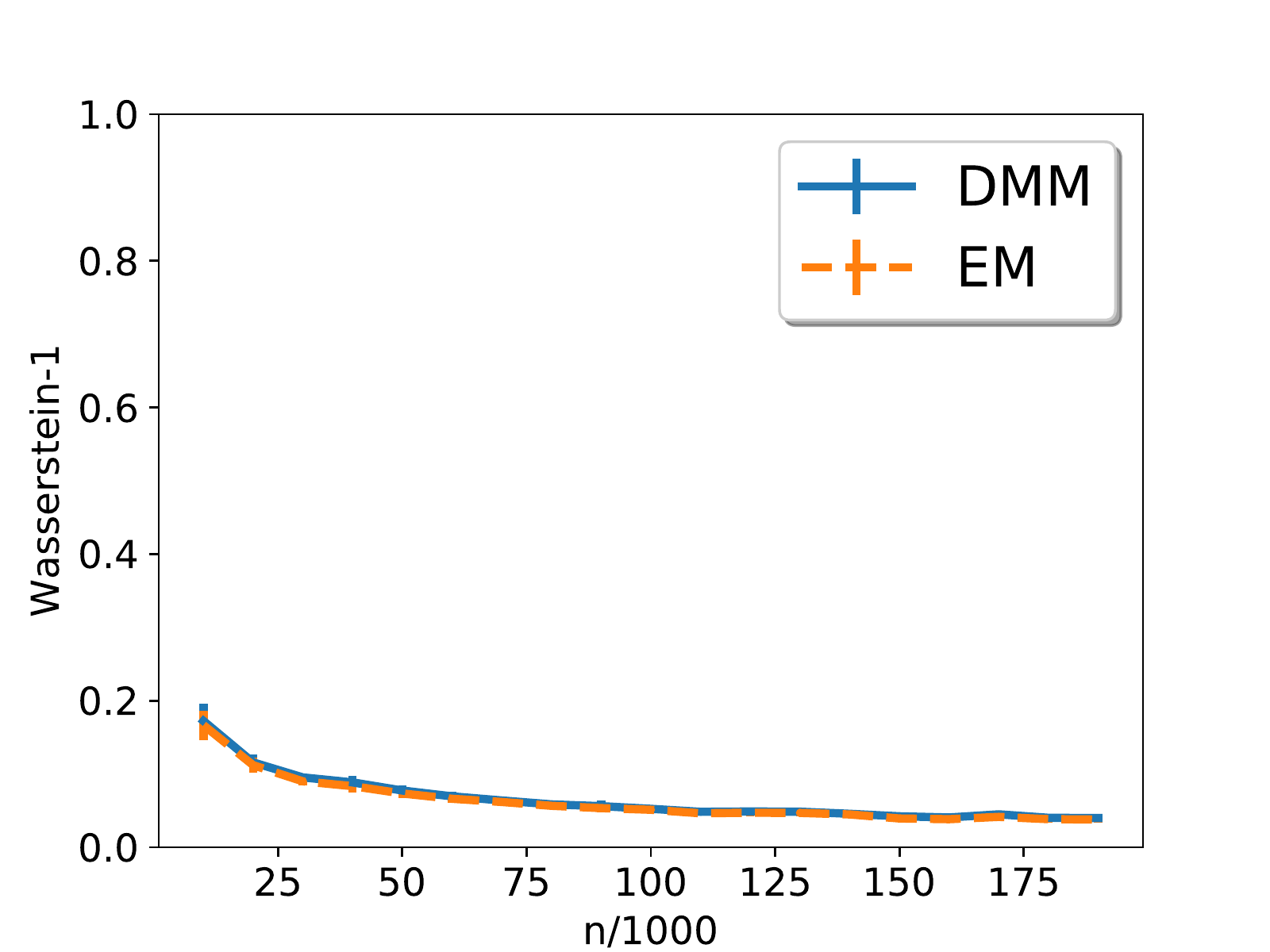}}
	\subfigure[$\|\mu\|=2$]%
	{\label{fig:k2d} \includegraphics[width=0.4\columnwidth]{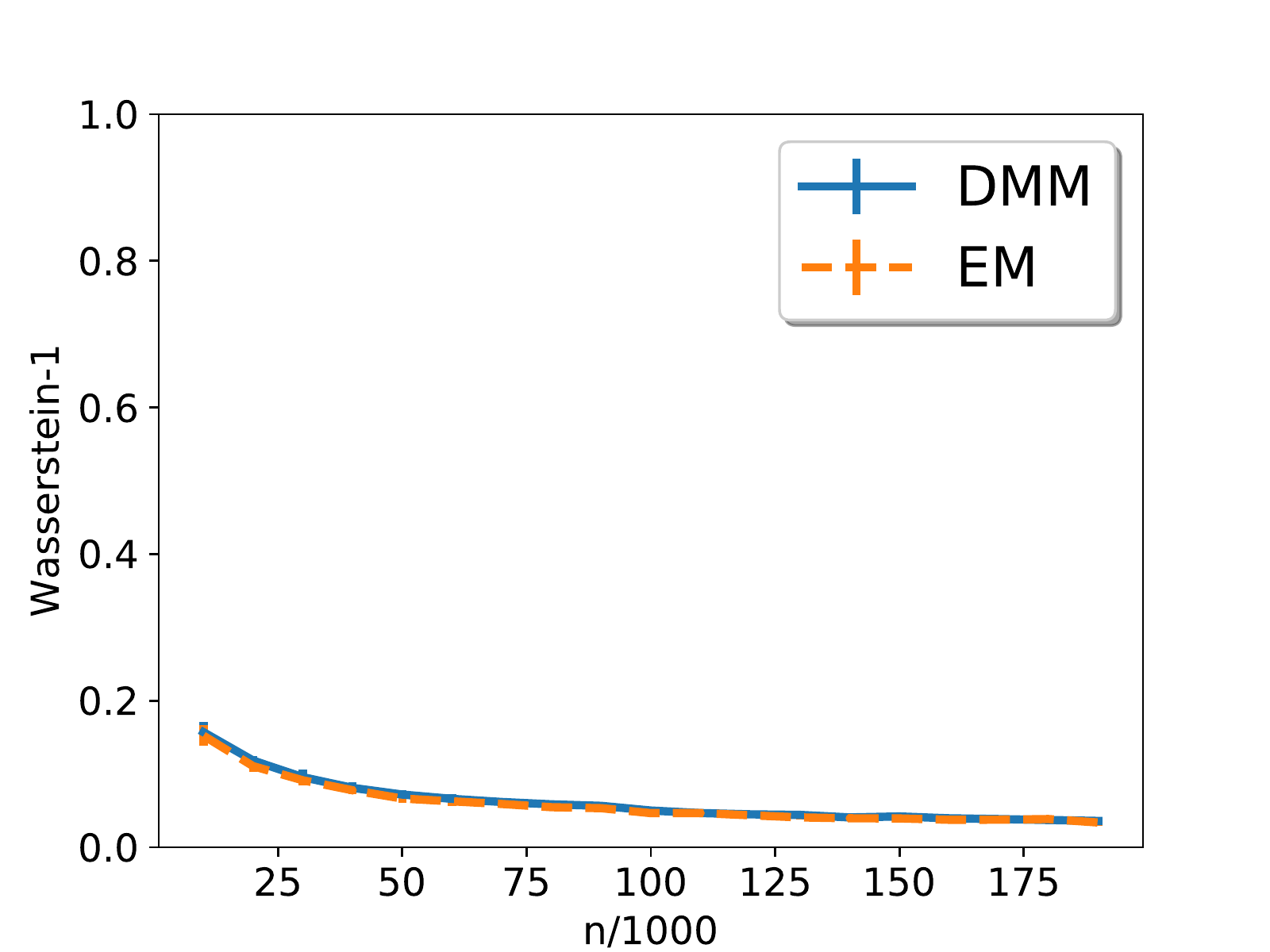}}
	\caption{In the first three figures, $P_\Gamma = \frac{1}{2} N(\mu, I_d) + \frac{1}{2} N(-\mu, I_d)$, for increasing values of $\norm{\mu}_2$. In the final figure, $P_\Gamma = \frac{1}{4} N(\mu, I_d) + \frac{3}{4} N(-\mu, I_d)$ where $\norm{\mu}_2 = 2$.}
	\label{fig:k2}
\end{figure}

In~\prettyref{fig:k3}, we compare the performance on the $3$-GM model $\frac{1}{3} N(\mu, I_d) + \frac{1}{3} N(0, I_d) + \frac{1}{3} N(-\mu, I_d)$ for different values of 
separation $\|\mu\|$.
%. We follow the same pattern of increasing the separation of the components in each experiment. %For~\prettyref{fig:k3a}, $\mu = 0$, i.e., the components completely overlap. And for~\prettyref{fig:k3b} and~\prettyref{fig:k3c}, $\mu$ is uniformly  drawn from the sphere of radius $1$ and $2$, respectively.
In these experiments, we see the opposite phenomenon in terms of the relative performance of our algorithm and EM: the former improves more as the centers become more separated. This seems to be because in, for instance, the case where $\mu = 0$, the error in each coordinate for DMM is fairly high, and this is compounded when we select the two-coordinate final distribution. The performance of our algorithm improves rapidly here because as the model becomes more separated, the errors in each coordinate become very small. Note that since we have made the model more difficult to learn by adding a center at $0$, the errors are higher than for the $k = 2$ example in every experiment for both algorithms. 
%In terms of running time, for $k = 3$, EM is faster than our algorithm for smaller sample size because of the grid search being invoked for~\prettyref{algo:gmm_est}. But EM slows much more rapidly as the sample size increases and is actually slower than our algorithm for large values of $n$.

%The DMM algorithm of this paper should not be adaptive to increasing separation of the model atoms because of the coarseness of the grids used. But we posit that if we were to use finer nets, we might see faster rates in models with well-separated centers. In~\prettyref{fig:k3d}, we let the model be exactly as in~\prettyref{fig:k3c}, but we run our algorithm with a finer weights net, specifically with $C_1 = 2$ intead of $C_1 = 1$. We don't see meaningful improvement of our algorithm's accuracy, which suggests it may not be adaptive. 

\begin{figure}[ht]
	\centering
	\subfigure[$\mu=0$]%
	{\label{fig:k3a} \includegraphics[width=0.4\columnwidth]{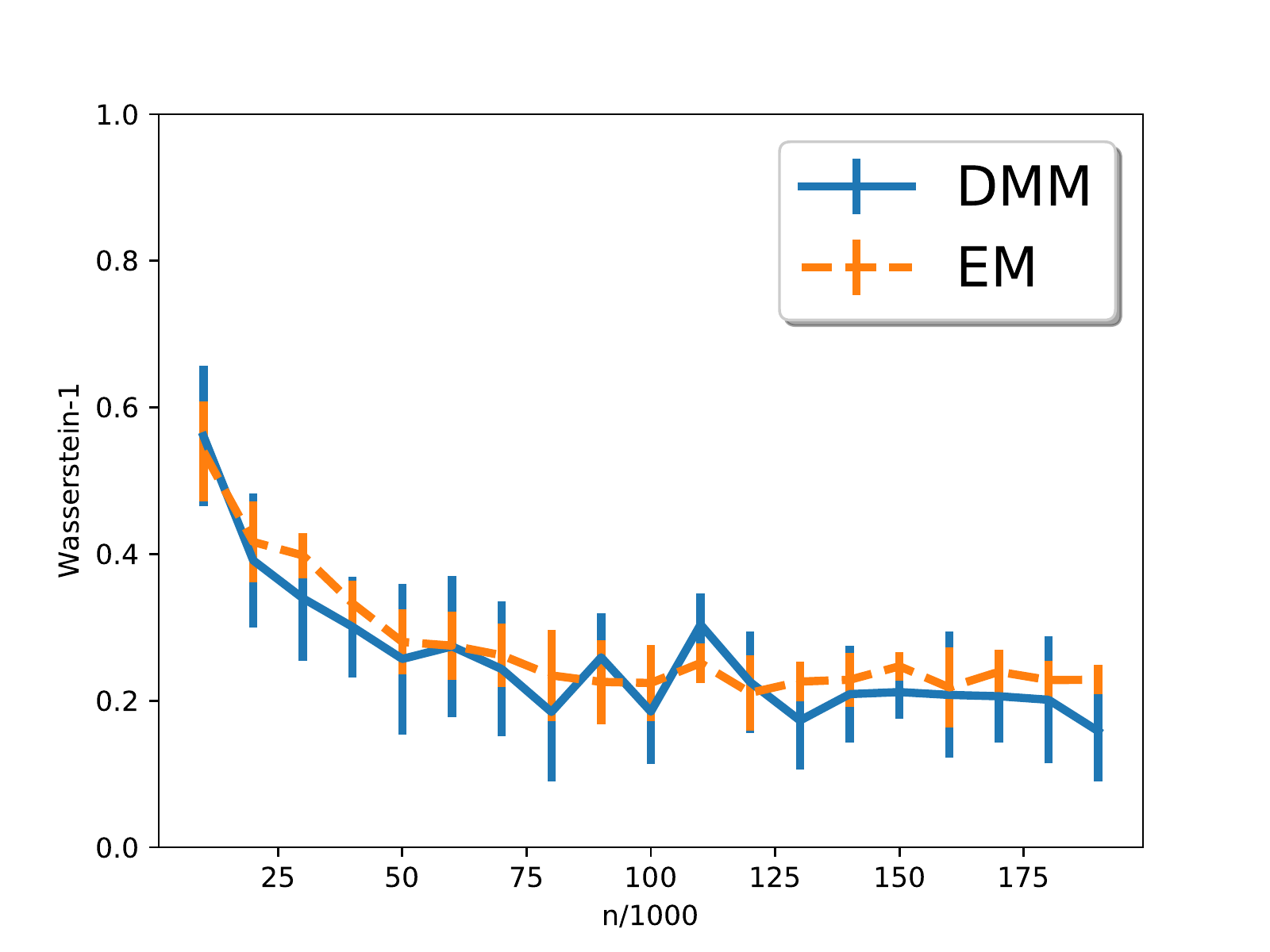}}
	\subfigure[$\|\mu\|=1$]%
	{\label{fig:k3b} \includegraphics[width=0.4\columnwidth]{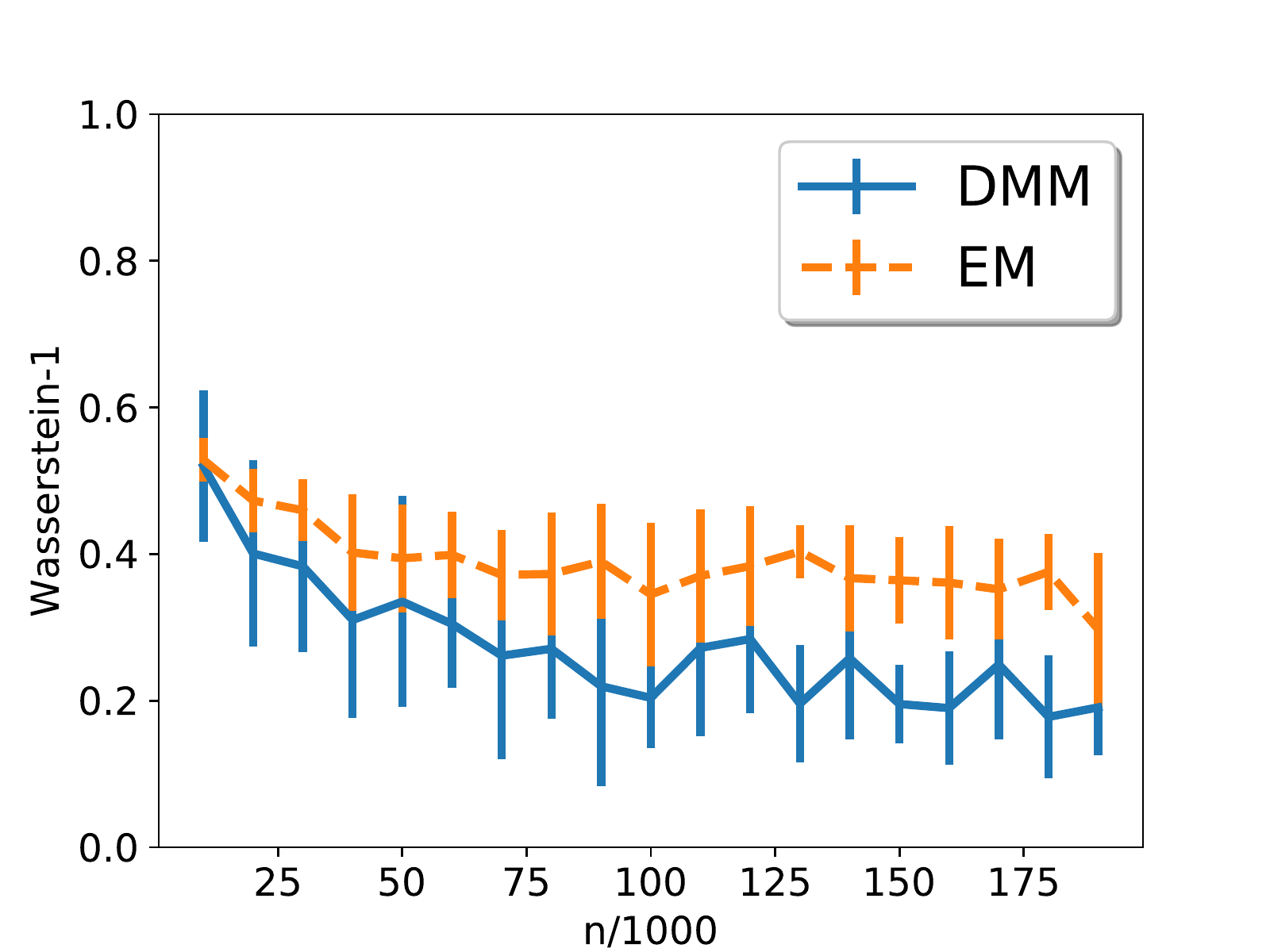}}
	\subfigure[$\|\mu\|=2$]%
	{\label{fig:k3c} \includegraphics[width=0.4\columnwidth]{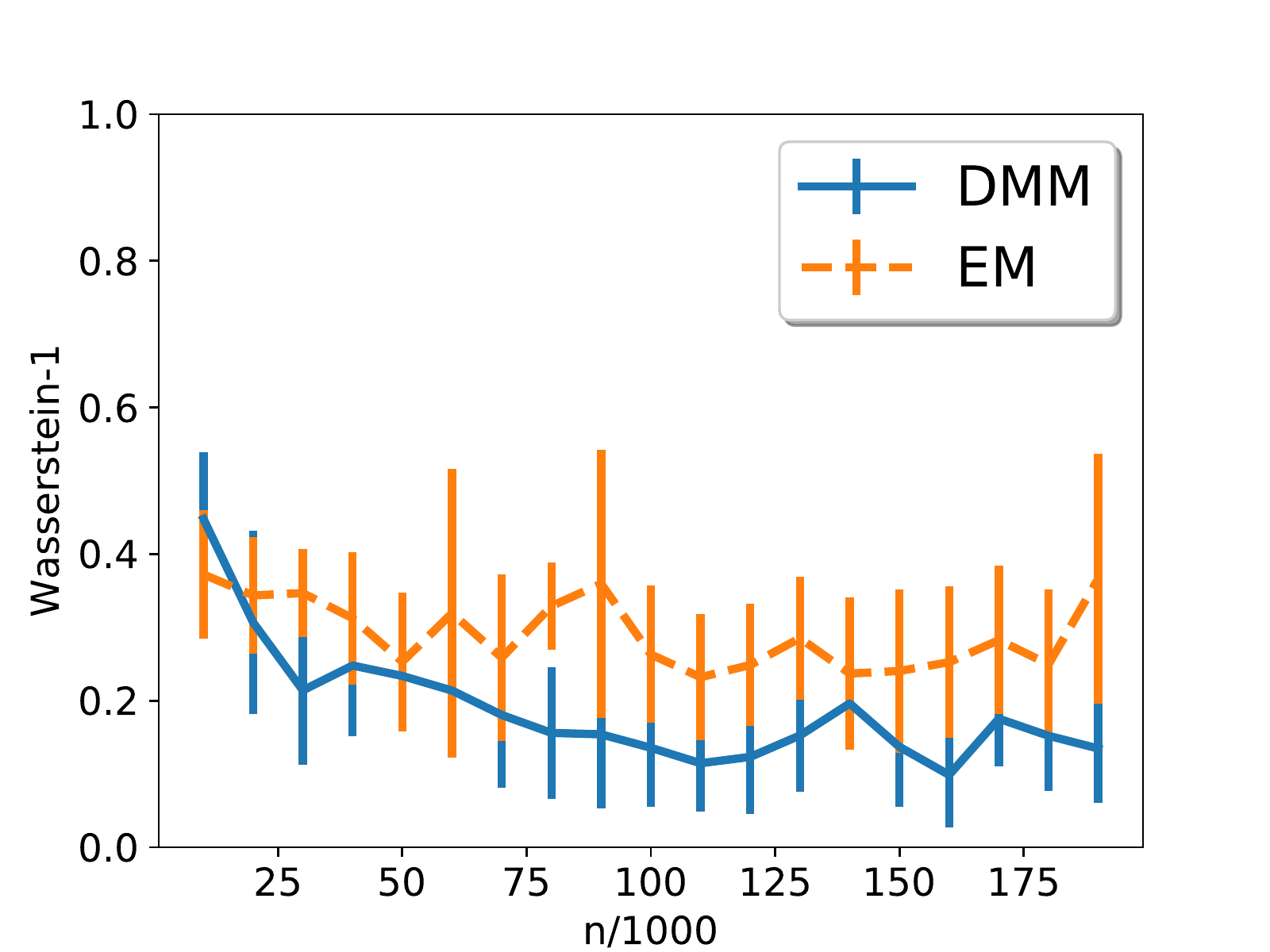}}
	%\subfigure[$\|\mu\|=2, C_1 = 2, C_2 = 4$]%
	%{\label{fig:k3d} \includegraphics[width=0.4\columnwidth]
%{sim_overn_error_k03_k3_d100_factormodel2_weightseven_sigma1_factorweights2_factorthetas4.pdf}}
	\caption{$P_\Gamma = \frac{1}{3} N(\mu, I_d) + \frac{1}{3} N(0, I_d) + \frac{1}{3} N(-\mu, I_d)$ for increasing values of $\norm{\mu}_2$.}
	\label{fig:k3}
\end{figure}

In~\prettyref{fig:misc}, we provide further experiments to explore the adaptivity of the estimator produced by the algorithm in~\prettyref{sec:mixing_estimation}. The settings are the same as in the previous experiments except we choose a finer grid with parameter $C_1 = 2$ instead of $C_1 = 1$, for otherwise the quantization error of the weights is too large. 
%In these experiments, a coarse weights net did not contain the weights close to the truth, so we saw very poor performance of DMM until we allowed the weights net to be somewhat finer.}

\added{In~\prettyref{fig:k02_k3}, we let the true model be exactly as in~\prettyref{fig:k2c}, but we run the algorithm from~\prettyref{sec:mixing_estimation} using $k=3$. As in~\prettyref{fig:k3}, DMM seems to improve more rapidly than EM as $n$ increases. But here, DMM has higher error than EM for small $n$. In~\prettyref{fig:dirichlet}, we let $k=3$ and create a model without the symmetry structures of the models in previous experiments by drawing the atoms uniformly from the unit sphere and the weights from a $\text{Dirichlet}(1, 1, 1)$ distribution. This model is more difficult to learn for both DMM and EM, but DMM still outperforms EM in terms of accuracy.}

\begin{figure}[ht]
	\centering
	\subfigure[$P_\Gamma = \frac{1}{2} N(\mu, I_d) + \frac{1}{2} N(-\mu, I_d)$ with $\norm{\mu}_2 = 2$.] % $k_0=2, k=3$.
	{\label{fig:k02_k3} \includegraphics[width=0.32\columnwidth]{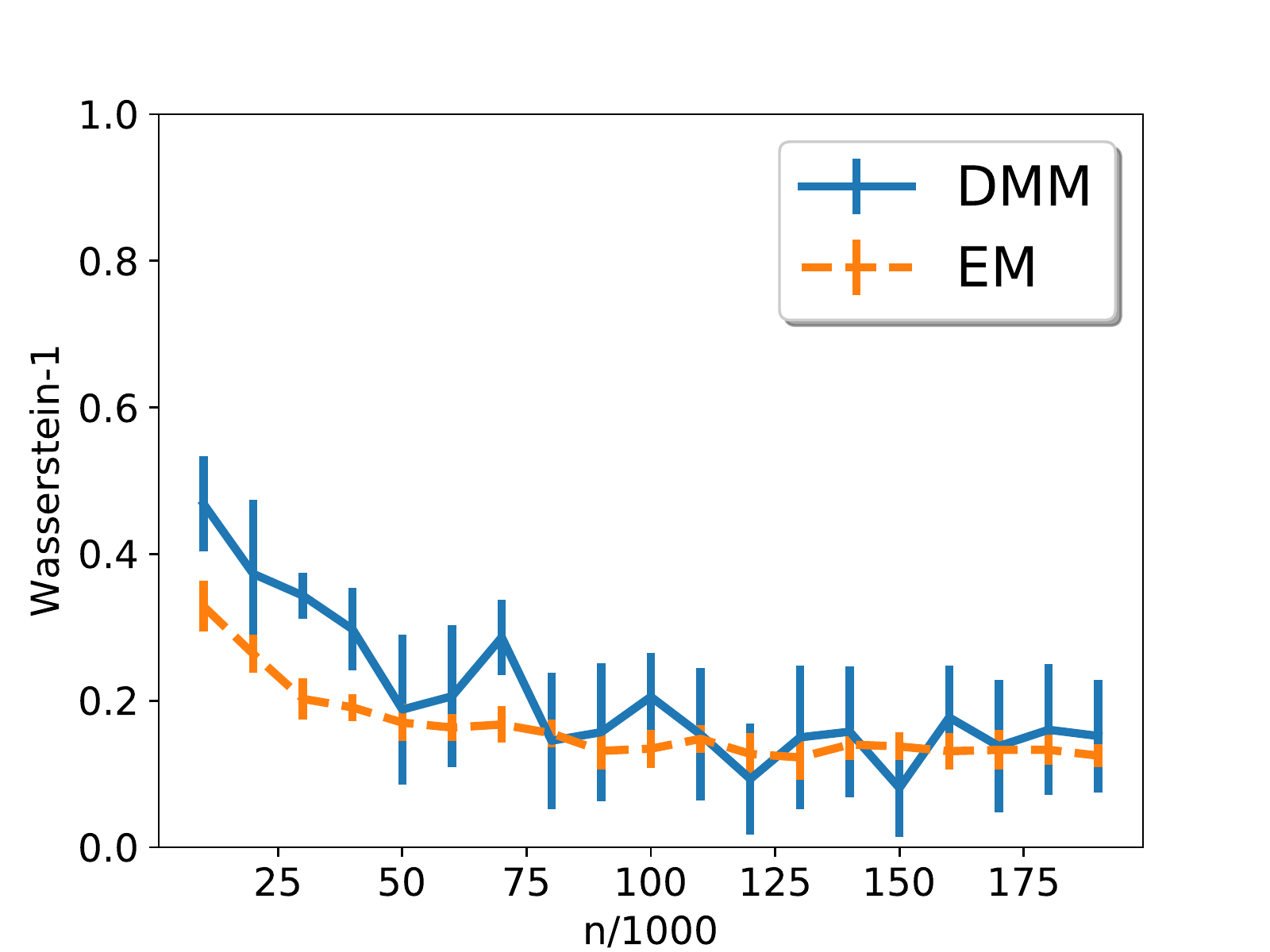}}
	\subfigure[$P_\Gamma = \sum_{j = 1}^3 w_j N(\mu_j, I_d)$ with $\mu_j$'s drawn uniformly from the unit sphere and $(w_1, w_2, w_3) \sim \text{Dirichlet}(1, 1, 1)$.]
	{\label{fig:dirichlet} \includegraphics[width=0.32\columnwidth]{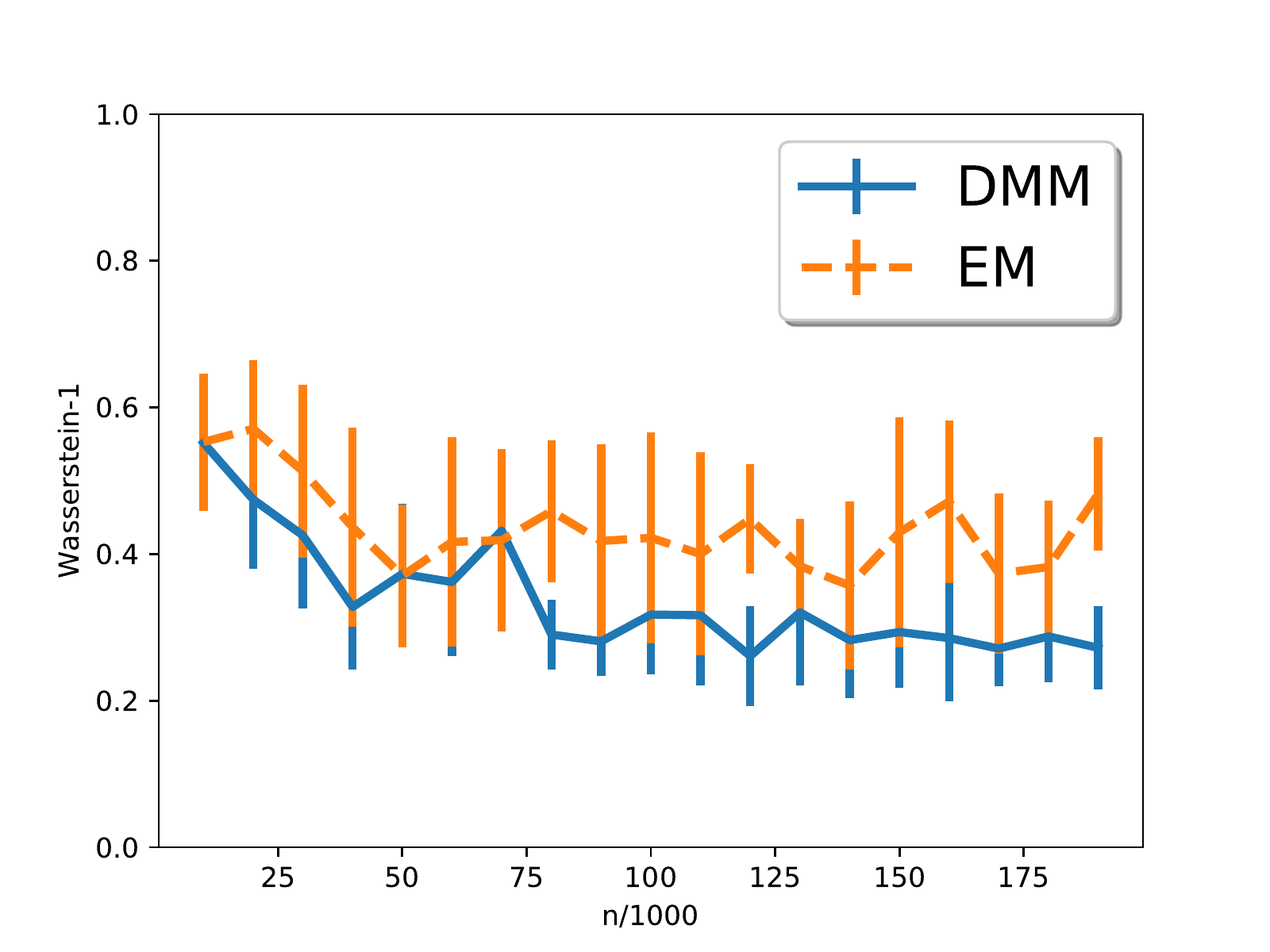}}
	\caption{Adaptivity and asymmetry.}
	\label{fig:misc}
\end{figure}

\added{Finally, we provide a table of the average running time (in seconds) for each experiment. As expected, in the experiments in~\prettyref{fig:k2}, the DMM algorithm is faster than EM. In the experiments in~\prettyref{fig:k3}, DMM manages to run faster on average than EM in the two more separated models,~\prettyref{fig:k3b} and~\prettyref{fig:k3c}. Also unsurprisingly, DMM is slower in the~\prettyref{fig:k3} experiments than those in~\prettyref{fig:k2}, because grid search is invoked in the former. In the over-fitted case in~\prettyref{fig:k02_k3}, DMM is much slower on average than EM, and in fact is slower on average than DMM in any other experiment setup.
%in addition to being less accurate than EM for small samples. 
In \prettyref{fig:dirichlet}, where the model does not have special structure, DMM nonetheless runs on average in time faster than for some of the symmetric models in~\prettyref{fig:k3}, and moreover again improves on the average run time of EM.}
\begin{table}[ht] \label{tab:run_time}
\centering 
\begin{tabular}{rrr} 
 Experiment &       DMM &        EM \\
          \prettyref{fig:k2a} &  0.114407 &  0.678521 \\
          \prettyref{fig:k2b} &  0.121561 &  1.163886 \\
          \prettyref{fig:k2c} &  0.206713 &  0.573640 \\
          \prettyref{fig:k2d} &  0.221704 &  0.818138 \\
          \prettyref{fig:k3a} &  1.118308 &  0.985668 \\
          \prettyref{fig:k3b} &  2.257503 &  2.582501 \\
          \prettyref{fig:k3c} &  2.179928 &  3.576998 \\
          \prettyref{fig:k02_k3} &  3.840112 &  1.350464 \\
          \prettyref{fig:dirichlet} &  1.907299 &  2.546508 \\
\end{tabular}
\end{table}

\section{Discussion} \label{sec:discussion}

%We now discuss connections between this work and classical estimators for Gaussian mixtures. We also touch on limitations of this paper and areas for further research.
%We end the paper with a discussion on the assumptions and limitations of this paper and open problems for further research.

%\paragraph{On assumptions of this paper}
In this paper we focused on the Gaussian location mixture model \prettyref{eq:gmm_model} in high dimensions, where the variance parameter $\sigma^2$ and the number of components $k$ are known, and the centers lie in a ball of bounded radius. Below we discuss weakening these assumptions and other open problems.

\paragraph{Unbounded centers}
While the assumption of bounded support is necessary for estimating the mixing distribution (otherwise the worst-case $W_1$-loss is infinity), it is not needed for density estimation \cite{Acharya_etal_2014,Li_Schmidt_2017,Ashtiani_etal_2018}. In fact, \cite{Acharya_etal_2014} first uses a crude clustering procedure to partition the sample into clusters whose means are close to each other, then zooms into each cluster to perform density estimation. For the lower bound, the worst case occurs when each cluster is equally weighted and highly separated, so that the effective sample size for each component is $n/k$, leading to the lower bound of $\Omega(\frac{kd}{n})$. On the other hand, the density estimation guarantee for NPMLE in~\cite{Ghosal_VdV_2001,Zhang_2009,Saha_2017} relies on assumptions of either compact support or tail bound on the mixing distribution.

\paragraph{Location-scale mixtures}
We have assumed that the covariance of our mixture is known and common across components. There is a large body of work studying general location-scale Gaussian mixtures, see, e.g.,~\cite{Moitra_Valiant_2010,Ho_Nguyen_2016_annals,Ashtiani_etal_2018}. The introduction of the scale parameters makes the problem significantly more difficult. For parameter estimation, the optimal rate remains unknown even in one dimension except for $k=2$ \cite{Hardt_Price_2015}. 
%This is beyond the scope of the present paper but such considerations are important for future research. 
In the special case where all components share the same unknown variance $\sigma^2$, the optimal rate in one dimension is shown in~\cite{WY18} to be $n^{-1/(4k)}$ for estimating the mixing distribution and $n^{-1/(2k)}$ for $\sigma^2$, achieved by Lindsay's estimator \cite{Lindsay_1989}. Modifying the procedure in \prettyref{sec:mixing_estimation} by replacing the DMM subroutine with Lindsay's estimator, this result can be extended to high dimensions as follows (see \prettyref{app:mixing_est_unknownsigma} for details), provided that the unknown covariance matrix is isotropic; otherwise the optimal rate is open.

\begin{thm}[Unknown common variance] \label{thm:mixing_est_unknownsigma}
Assume the setting of \prettyref{thm:mixing_est}, where $P_\Gamma = \Gamma * N(0,\sigma^2 I_d)$ for some unknown $\sigma$ bounded by some absolute constant $C$ and $\Gamma \in \calG_{k,d}$.
 Given $n$ i.i.d.~observations from $P_\Gamma$, there exists an estimator $(\hat\Gamma,\hat\sigma)$ such that
%\footnote{The notation $a \asymp_k b$ means that $c_k a \leq b \leq C_k a$ for some constants $c_k, C_k$ depending only on $k$. }
\begin{equation}
\mb E W_1(\hat \Gamma, \Gamma) 
\lesssim_k \(\frac{d}{n} \)^{1/4} \wedge 1 + n^{-1/(4k)}, \quad
\mb E |\hat\sigma^2-\sigma^2|
\lesssim_k n^{-1/(2k)}.
%\label{eq:mixing_est_avg_risk}
\end{equation}
Furthermore, both rates are minimax optimal.
\end{thm}

\paragraph{Number of components}
This work assumes that the parameter $k$ is known and fixed. Since the centers are allowed to overlap arbitrarily, $k$ is effectively an upper bound on the number of components.
If $k$ is allowed to depend on $n$, the optimal $W_1$-rate is 
%one-dimensional version of \prettyref{thm:mixing_est} 
shown in \cite[Theorem 5]{WY18} to be $\Theta(\frac{\log\log n}{\log n})$ provided $k=\Omega(\frac{\log n}{\log\log n})$, including nonparametric mixtures. Extending this result to the high-dimensional setting of \prettyref{thm:mixing_est} is an interesting future direction.

%If $k$ is allowed to depend on $n$, the optimal $W_1$-rate is 
%%one-dimensional version of \prettyref{thm:mixing_est} 
%shown in \cite{WY18} to be $\Theta(n^{-1/(4k-2)})$ and $\Theta(\frac{\log\log n}{\log n})$ 
%for $k=O(\frac{\log n}{\log\log n})$ and $k=\Omega(\frac{\log n}{\log\log n})$, respectively. Extending this result to the high-dimensional setting of \prettyref{thm:mixing_est} is an interesting future direction.

The problem of selecting the mixture order $k$ has been extensively studied.
For instance, many authors have considered likelihood-ratio based tests; however, standard asymptotics for such tests may not hold \cite{Hartigan_1985}. Various workarounds have been considered, including method of moments \cite{Lindsay_1989,dacunha1997estimation}, tests inspired by the EM algorithm \cite{Li_Chen_2010}, quadratic approximation of the log-likelihood ratio \cite{Liu_Shao_2003}, and penalized likelihood \cite{gassiat2012consistent}.
A common practical method is to infer $k$ from an eigengap in the sample covariance matrix. In our setting, this technique is not viable even if the model centers are separated, since the atoms may all lie on a low-dimensional subspace. However, under separation assumptions we may infer a good value of $k$ from the estimated mixing distribution $\hat\Gamma$ of our algorithm.

\paragraph{Efficient algorithms for density estimation}
%As mentioned in \prettyref{sec:related}, for the high-dimensional $k$-GM model, achieving the optimal rate $O(\sqrt{\frac{d}{n}})$ (or the less ambitious goal of outperforming the result $\tilde O((\frac{d}{n})^{1/4})$ of~\cite{Acharya_etal_2014}) with a proper density estimate in polynomial time is unresolved. A notable exception is the special case of symmetric 2-GM \prettyref{eq:sym2GM}, for which spectral algorithms achieve the sharp rate for both parameter estimation and density estimation. Indeed, by \prettyref{thm:H-M}, since both the first and third moment tensors are zero by symmetry, we have
%\begin{equation}
%H(P_\mu,P_{\mu'})^2 \asymp \Fnorm{\mu \mu^\top-\mu' \mu'^\top}^2.
%\label{eq:H-M-sym2GM}
%\end{equation}
%Let $\hat\mu\hat\mu^\top$ be the best rank-one approximation of $\hat\Sigma$ in \prettyref{eq:hatSigma}. Then for $d \leq n$, it is easy to show that (see, e.g.,~\cite[Appendix B]{Wu_Zhou_2019})   $\Expect\Fnorm{\mu\mu^\top-\hat\mu\hat\mu^\top}^2 = O(\frac{d}{n})$ and $\Expect[\min\{\|\hat\mu-\mu\|,\|\hat\mu+\mu\|\}] = O((\frac{d}{n})^{1/4})$. In contrast, if the 2-GM is asymmetric, then it is necessary to use the third-order moment tensor as the model is not identified by the first two moments; however, low-rank tensor approximation is difficult without extra separation assumptions, creating algorithmic challenges for density estimation.

As mentioned in \prettyref{sec:related}, for the high-dimensional $k$-GM model, achieving the optimal rate $O_k(\sqrt{d/n})$ with a proper density estimate in polynomial time is unresolved except for the special case of $k=2$. 
Such a procedure, as described in~\prettyref{ssec:density_est_efficient}, is of method-of-moments type (involving the first three moments); nevertheless, 
thanks to the observation that the one-dimensional subspace of spanned by the centers of a zero-mean $2$-GM can be extracted from the covariance matrix, 
we can reduce the problem to one dimension by projection, thereby sidestepping third-order tensor decomposition which poses computational difficulty.
Unfortunately, this observation breaks down for $k$-GM with $k\geq 3$, as covariance alone does not provide enough information for learning the subspace accurately.
For this reason it is unclear whether the algorithm in \prettyref{ssec:density_est_efficient} is capable to achieve the optimal rate of $\sqrt{d/n}$ and so far we can only prove a rate of $(d/n)^{1/4}$ in \prettyref{thm:density_est_efficient}. Closing this computational gap (or proving its impossibility) is a challenging open question.

\paragraph{Analysis of the MLE}

A natural approach to any estimation problem is the maximum likelihood estimator, which, for the $k$-GM model \prettyref{eq:pGamma}, is defined as $\hat \Gamma_{\MLE} = \argmax_{\Gamma \in \calG_{k,d}} \sum_{i=1}^n \log p_\Gamma(X_i)$. 
%Computational issues aside, 
Although this non-convex optimization is difficult to solve in high dimensions, it is of interest to understand the statistical performance of the MLE and whether it can achieve the optimal rate of density estimation in \prettyref{thm:density_est}.

A rate of convergence for the MLE is typically found by bounding the \emph{bracketing entropy} of the class of square-root densities; see, e.g.,~\cite{VdV_Wellner_1996,VdG_2000}. Given a function class $\mc F$ of real-valued functions on $\reals^d$, its $\epsilon$-bracketing number $N_{[]}(\epsilon)$ is defined as the minimum number of brackets (pairs of functions which differ by $\epsilon$ in $L_2$-norm), such that each $f\in\calF$ is sandwiched between one of such brackets. 
Suppose that the class $\mc F$ is parametrized by $\theta$ in some $D$-dimensional space $\Theta$. For such parametric problems, it is reasonable to expect that the bracketing number of $\mc F$ behaves similarly to the covering number of $\Theta$ as $\(\frac{1}{\epsilon}\)^{O(D)}$ (see, for instance, the discussion on \cite[p.~122]{VdG_2000}). 
%\added{In low-dimensional Gaussian mixtures, these two quantities are often related using the smoothness of the parameterization plus an additional truncation argument (see, e.g.~\cite[Theorem 3.1]{Ghosal_VdV_2001} and \cite[Lemma 2.1]{Ho_Nguyen_2016_annals}).} -- both results are for infinite mixtures!
Such bounds for Gaussian mixtures were obtained in \cite{maugis2011non}. For example, for $d$-dimensional $k$-GMs,~\cite[Proposition B.4]{maugis2011non} yields the following bound for the global bracketing entropy:
\begin{equation}
\log N_{[]}(\epsilon) \lesssim k d \log\frac{C d}{\epsilon}. \label{eq:maugis}
\end{equation}
Using standard result based on bracketing entropy integral (c.f.~e.g.~\cite[Theorem $7.4$]{VdG_2000}), this result leads to the following high-probability bound for the MLE $\hat \Gamma_{\rm ML}$:
\begin{equation}
H(P_{\hat \Gamma_{\rm ML}}, P_\Gamma) \le C \sqrt{\frac{d k \log (d n)}{n}}, \label{eq:mle_nonsharp}
\end{equation}
which has the correct dependency on $k$, but is suboptimal by a logarithmic factor compared to \prettyref{thm:density_est}.
It is for this reason that we turn to the Le Cam-Birg\'e estimator, which relies on bounding the local Hellinger entropy without brackets, in proving \prettyref{thm:density_est}.
Obtaining a local version of the bracketing entropy bound in \prettyref{eq:maugis} and determining the optimality (without the undesirable log factors) of the MLE for high-dimensional GM model remains open.

\paragraph{Adaptivity}
%\nbwu{Mention the worst-case nature of \prettyref{thm:mixing_est}, and say that like in \cite{HK2015,WY18}, the way to go beyond this pessimistic result is to consider ... (I don't think you need to define any notations like you did in Definition 1). I think (please double check):  if we simply change the grid size  from $n^{-\frac{1}{4k-2}}$ to $n^{-1/2}$ in \prettyref{algo:gmm_est}, we automatically get adaptivity, with the price of increasing the time to $n^{...}$. The analysis is analogous by using the improved version of the moment-to-Wasserstein inequality in \prettyref{lem:w1_moments_ub} (see \cite[Prop.~4]{WY18}). Then say something like for ease of exposition, we do not pursue this in details.}
%\nb{Here is a crucial point though, and I've tried to discuss this below. I originally thought our density estimator, OR for instance, as you propose, changing the grid size in~\prettyref{algo:gmm_est} to $n^{-1/2}$ would give an adaptive estimator. I'm guessing that it does, but I do not know how to prove it. This is because of the reliance on sliced Wasserstein distance in the proof. Even if the atoms of $\Gamma$ are well-separated, does it mean the atoms of $\Gamma_\theta$ are where $\theta = \argmax W_1(\Gamma_\theta, \hat \Gamma_\theta)$? I discussed this some below. I think it would be interesting to work on this in the future. Anyway here is the current attempt at this section:}

The rate in~\prettyref{thm:mixing_est} is optimal in the worst-case scenario where the centers of the Gaussian mixture can overlap. 
%It is reasonable to expect that under stronger assumptions, a sharper rate is achievable. Indeed, 
To go beyond this pessimistic result, in one dimension, \cite{HK2015} showed that when the atoms of $\Gamma$ form $k_0$ well-separated (by a constant) clusters (see~\cite[Definition 1]{WY18} for a precise definition), the optimal rate is $n^{-1/(4(k-k_0) +2)}$, interpolating the rate $n^{-1/(4k-2)}$ in the worst case ($k_0=1$) and the parametric rate $n^{-1/2}$ in the best case ($k_0=k$). Furthermore, this can be achieved adaptively by either the minimum distance estimator \cite[Theorem 3.3]{HK2015} or the DMM algorithm \cite[Theorem 2]{WY18}. 

In high dimensions, it is unclear how to extend the adaptive framework in \cite{HK2015}.
For the procedure considered in \prettyref{sec:mixing_estimation}, by~\prettyref{lem:perturbation_psd}, the projection $\hat V$ obtained from PCA preserves the separation of the atoms of $\Gamma$. 
Therefore, in the special case of $k=2$, if we first center the data so that the projection $\gamma$ in \prettyref{eq:gamma} is one-dimensional, then the adaptive guarantee of the DMM algorithm allows us to adapt to the clustering structure of the original high-dimensional mixture; however, if $k>2$, \prettyref{algo:gmm_est} must be invoked to learn the multivariate $\gamma$, and it does not seem possible to obtain an adaptive version of \prettyref{lem:w1_sup_bound}, since some of the projections may have poor separation, e.g. when all the atoms are aligned with the first coordinate vector.

\section*{Acknowledgments}

The authors would like to thank David Pollard for a helpful conversation about the chaining technique in \prettyref{lem:chain_tail_probs} and to Elisabeth Gassiat
for bringing \cite{gassiat2014local,gassiat2012consistent} to our attention.

\appendix
%!TEX root = gmm.tex

\section{Auxiliary lemmas} \label{sec:appendix_auxiliary}

The following moment comparison inequality bound the Wasserstein distance between two univariate $k$-atomic distributions using their moment differences:

\begin{lem}[{\cite[Proposition 1]{WY18}}] \label{lem:w1_moments_ub} 
For any $\gamma, \gamma' \in \mc G_{k, 1}$ supported on $[-R,R]$,
\begin{align*}
W_1(\gamma, \gamma') &\le \added{C \cdot k} \max_{r \in [2k-1]} |m_r(\gamma) - m_r(\gamma')|^{1/(2k - 1)},
\end{align*}
where $C$ depends only on $R$.
%\added{where $C$ is a constant not depending on $k, \gamma,$ or $\gamma'$.}
\end{lem}

\begin{lem}[{Hypercontractivity inequality \cite[Theorem 1.9]{SS2012}}] \label{lem:hypercontractivity} 
Let $Z\sim N(0,I_d)$. Let $g: \mb R^d \rightarrow \mb R$ be a polynomial of degree at most $q$. Then for any $t > 0$, 
\begin{align*}
\mb P \{ |g(Z) - \mb E g(Z)| \ge t \} \le e^2\exp \( - \(\frac{t^2}{C \Var \, g(Z) }\)^{1/q} \),
\end{align*}
where $C$ is a universal constant. 
\end{lem}

\begin{forme}
In~\prettyref{lem:moment_increments_sg}, the bound is:
\begin{align*}
\mb P \{ |f_r(\theta_1) - f_r(\theta_2)| \ge \norm{\theta_1 - \theta_2}_2 \lambda \} \le 2 \exp \(- \frac{\lambda^2}{8r^r R^{2r}} \) + \\
e^2 \exp \(- \( \frac{ \lambda^2}{C\( \sum_{j=0}^{\lfloor{r/2}\rfloor} c_{j, r} \sqrt{(r-2j)^2 2^{2(r-2j)-1} (R^{2(r-2j)} + (ck(r-2j))^{r-2j}}\)^2} \)^{1/r} \).
\end{align*}
\end{forme}

%-------------------------------------------------------------------------------
\begin{lem} \label{lem:moment_increments_sg}
Fix $r \in [2k-1]$. Let $f_r(\theta)$ be the process defined in~\prettyref{eq:empirical_process}. Let $\lambda > 0$. There are positive constants $C, c$ such that, for any $\theta_1,\theta_2\in S^{k-1}$,
\begin{align}
& \mb P \{ |f_r(\theta_1) - f_r(\theta_2)| \ge \norm{\theta_1 - \theta_2}_2 \lambda \} \le \added{C \exp \( - \frac{c \lambda^{2/r}}{kr} \).} \label{eq:increment-ub}\\
& \mb P \{ |f_r(\theta_1)| \ge  \lambda \} \le \added{C \exp \( - \frac{c \lambda^{2/r}}{kr} \).} \label{eq:single-ub}
\end{align}
\end{lem}
\begin{proof}
Define $\Delta \triangleq \sqrt{n} \( \tilde m_r\(\theta_1\) - \tilde m_r\(\theta_2\) \)$.
Then, $f_r(\theta_1) - f_r(\theta_2) = \Delta - \mb E \Delta$.
Recall that $\tilde m_r(\theta)=\frac{1}{n}\sum_{i=1}^n H_r(\theta^\top X_i)$, where $X_i=U_i+Z_i$, $U_i \iiddistr \gamma$, and $Z_i \iiddistr N(0, I_k)$. 
Conditioning on $U=(U_1,\dots,U_n)$, we have 
\[
\Expect(\Delta| U)
= \frac{1}{\sqrt{n}} \sum_{i =1}^n \( (\theta_1^{\top}U_i)^r - (\theta_2^{\top}U_i)^r \) .
\]
Now $|(\theta_1^{\top}U_i)^r - (\theta_2^{\top}U_i)^r| \le r R^r \norm{\theta_1 - \theta_2}_2$ since $\norm{U_i}_2 \le R$.
By Hoeffding's inequality, 
\begin{equation}
\label{eq:ub-U}
\mb P \{ |\mb E (\Delta | U) - \mb E \Delta| \ge \norm{\theta_1 - \theta_2}_2 \lambda \} \le 2\exp \( - \frac{\lambda^2}{2 r^2 R^{2r}} \).
\end{equation}
We now condition on $U$ and analyze $|\Delta-\Expect(\Delta| U)|$. Since $\Delta$ is a polynomial of degree $r$ in $Z_1,\dots,Z_n$, by \prettyref{lem:hypercontractivity}, 
\begin{equation}
\label{eq:ub-Z}
\mb P \{ |\Delta-\Expect(\Delta| U)| \ge \norm{\theta_1 - \theta_2}_2 \lambda \added{|U}\} \le e^2 \exp \(-\(\frac{\norm{\theta_1 - \theta_2}_2^2}{C\Var(\Delta|U))}\)^{1/r}\lambda^{2/r} \).
\end{equation}
It remains to upper-bound $\Var \( \Delta | U \)$. We have
\[
\Var \( \Delta | U \)
=\frac{1}{n}\sum_{i=1}^n\Var \( H_r(\theta_1^\top X_i) - H_r(\theta_2^\top X_i) | U_i \).
\]
Since the standard deviation of a sum is no more than the sum of the standard deviations,
\[
\sqrt{\Var  \( H_r(\theta_1^{\top}X_i) - H_r(\theta_2^{\top}X_i) | U_i \)} 
\le \sum_{j = 0}^{\lfloor{r/2}\rfloor} c_{j, r} \sqrt{ \mb E \( \( (\theta_1^{\top}X_i)^{r - 2j} - (\theta_2^{\top}X_i)^{r - 2j} \)^2 | U_i \)}, 
\]
where \added{$c_{j, r} = \frac{r!}{2^jj!(r-2j)!}$.}
% Since $\theta^\top X_i\sim N(\theta^\top U_i,1)$ ...
For any $\ell \le r$, we have $|(\theta_1^{\top}X)^\ell -(\theta_2^{\top}X)^\ell|\le \ell \Norm{X}_2^\ell \Norm{\theta_1-\theta_2}_2 $ and thus
\[
\mb E \( \( (\theta_1^{\top}X_i)^{\ell} - (\theta_2^{\top}X_i)^{\ell} \)^2 | U_i \) 
\le \ell^2 \Norm{\theta_1-\theta_2}_2^2 \mb E (\Norm{X_i}_2^{2\ell}|U_i). 
\]
Since $\norm{X_i}_2^{2\ell} \le 2^{2\ell-1} \(\norm{Z_i}_2^{2\ell} + \norm{U_i}_2^{2\ell}\)$ and $\mb E \norm{Z_i}_2^{2\ell} \le (c k\ell)^\ell$ for a constant $c$,
\[
\mb E \( \( (\theta_1^{\top}X_i)^{\ell} - (\theta_2^{\top}X_i)^{\ell} \)^2 | U_i \) 
\le \norm{\theta_1 - \theta_2}_2^2 \cdot \ell^2 2^{2\ell - 1} \( R^{2\ell} + (ck\ell)^\ell \).
\]
\added{
Note that $c_{j, r} = {r \choose 2j} \frac{(2j)!}{2^j j!} \le {r \choose 2j} 2^j j! \le {r \choose 2j} (2j)^j$. Letting $a_{j, r} = (r-2j)2^{r-2j} \(R^{r-2j} + \sqrt{ck(r-2j)}^{r-2j}\)$, we have}
\begin{align*}
\sum_{j = 0}^{\lfloor{r/2}\rfloor} c_{j, r} \sqrt{ \mb E \( \( (\theta_1^{\top}X_i)^{r - 2j} - (\theta_2^{\top}X_i)^{r - 2j} \)^2 | U_i \)} &\le \norm{\theta_1 - \theta_2}_2 \max_{j \in \{0, \hdots, \lfloor{r/2}\rfloor\}}\( (2j)^j  a_{j, r} \) \sum_{j = 0}^{\lfloor{r/2}\rfloor} {r \choose 2j} \\
&\le \norm{\theta_1 - \theta_2}_2 \max_{j \in \{0, \hdots, \lfloor{r/2}\rfloor\}}\( (2j)^j  a_{j, r} \) (C')^r,
\end{align*}
for a constant $C'$. And there is a constant $C^{''}$ depending on $R$ such that
\begin{align*}
(2j)^j a_{j, r} &\le (C^{''})^r (2j)^j (k (r-2j))^{(r-2j)/2} \\
&\le (C^{''})^r \exp \(\frac{2j}{2} \log (2j) +  \frac{r-2j}{2} \log (k (r-2j)) \) \\
&\le (C^{''})^r \exp \( \frac{r}{2} \log r + \frac{r}{2} \( p \log p +  (1-p) \log k (1-p) \) \),
\end{align*}
where $p = 2j/r \in [0, 1]$. The term $p \log p +  (1-p) \log k (1-p) \le \log k$, so $(2j)^j a_{j, r} \le (C^{''} \sqrt{kr})^r$. Thus $\Var \( \Delta | U \)\le (c' k r)^r \norm{\theta_1 - \theta_2}_2^2$ for a constant $c'$.
%Therefore, \prettyref{eq:ub-Z} holds with probability $1- e^2\exp(-c_k \lambda^{2/r})$. \nb{This doesn't quite make sense;~\prettyref{eq:ub-Z} is already an upper bound on the probability. We could just say: the upper bound on the probability in~\prettyref{eq:ub-Z} can be written as $e^r \exp (-c_k \lambda^{2/r})$, where $c_k = c' kr$ for a constant $c'$.} 
Then \prettyref{eq:increment-ub} follows from~\prettyref{eq:ub-U} and~\prettyref{eq:ub-Z}.

The second inequality,~\prettyref{eq:single-ub}, can be proved by a similar application of Hoeffding's Inequality and~\prettyref{lem:hypercontractivity}. 
\end{proof}

\begin{forme}
I changed the citation of Pollard here. Originally I wanted to just cite his inequality. In Exercise $2$ on page 29, (Section $4.9$, ``Problems,''  he gives the final inequality. But this would require making some assumptions about $\Psi_\alpha$ norms and I felt it was easiest to do things directly. I actually discussed with him. There are expectation chaining bounds in Ledoux and Talagrand but I did not find them very citeable because they are for expectation and also they use the $\Psi_\alpha$ thing. I think we can proceed directly here.
\end{forme}
%-------------------------------------------------------------------------------

%adapted from the proof of Inequality $(31)$, in $4.7.1$, ``Tail chaining via packing'', page $20$, of~\cite{Pollard_mini}.]
The following lemma is adapted from \cite[Section 4.7.1]{Pollard_mini}.
\begin{lem}
\label{lem:chain_tail_probs}
Let $\Theta$ be a finite subset of a metric space with metric $\rho$. Let $f(\theta)$ be a random process indexed by $\theta \in \Theta$. Suppose that for $\alpha>0$, and for $\lambda > 0$, we have
\begin{align}
\label{eq:sg_assump}
\mb P \{ |f(\theta_1) - f(\theta_2)| \ge \rho(\theta_1, \theta_2) \lambda \} 
\le C_\alpha \exp \(- c_\alpha \lambda^{\alpha} \), \quad \forall~\theta_1, \theta_2 \in \Theta.
\end{align}
Let $\theta_0\in\Theta$ be a fixed point and $\epsilon_0=\max_{x,y\in\Theta}\rho(x,y)$.
Then \added{there is a constant $C$} such that with probability $1-C_\alpha \exp (- c_\alpha t^{\alpha})$,
\begin{align*}
\added{
\max_{\theta\in\Theta}|f(\theta)-f(\theta_0)|
\le C \cdot 2^{1/\alpha} \int_0^{\epsilon_0/2}\( t+ \( \frac{1}{c_\alpha} \log \frac{\epsilon_0|\calM(r,\Theta,\rho)|}{r} \)^{1/\alpha} \) dr.} 
% \\
% \deleted{\max_{\theta\in\Theta}|f(\theta)-f(\theta_0)|
% \le C \cdot 2^{1/\alpha} \int_0^{\epsilon_0/2}\( t+\log^{1/\alpha} \frac{\epsilon_0|\calM(r,\Theta,\rho)|}{r}  \) dr.}
\end{align*}
\end{lem}
\begin{proof}
We construct an increasing sequence of approximating subsets by maximal packing. 
Let $\Theta_0=\{\theta_0\}$.
For $i=0,1,2,\dots$, 
%we enlarge $\Theta_{i}$ to a maximal subset $\Theta_{i+1}\subseteq \Theta$ of points that are $\epsilon_{i+1}$-separated, where $\epsilon_{i+1}=\epsilon_i/2$.
let 
$\Theta_{i+1}$  be a maximal subset of $\Theta$ containing $\Theta_{i}$ that constitutes an $\epsilon_{i+1}$-packing, where $\epsilon_{i+1}=\epsilon_i/2$.
Since $\Theta$ is finite, the procedure stops after a finite number of iterations, resulting in $\Theta_0\subseteq \Theta_1\subseteq\dots\subseteq\Theta_m=\Theta$.
By definition, 
\[
N_i\triangleq |\Theta_{i}|\le M(\epsilon_i,\Theta,\rho).
\]
For $i=0,\dots,m-1$, 
define a sequence of mappings $\ell_i:\Theta_{i+1}\to \Theta_i$ by $\ell_i(t)\triangleq \arg\min_{s\in \Theta_i}\rho(t,s)$ (with ties broken arbitrarily). Then,
\[
\max_{\theta\in \Theta}|f(\theta)-f(\theta_0)|
\le \sum_{i=0}^{m-1}\max_{s\in \Theta_{i+1}}|f(s)-f(\ell_i(s))|.
\]
Since $\Theta_i$ is a maximal $\epsilon_i$-packing, we have $\rho(t,\ell_i(t))\le \epsilon_i$ for all $t\in \Theta_{i+1}$.
By the assumption \prettyref{eq:sg_assump} and a union bound, with probability $1-\sum_{i=0}^{m-1}N_{i+1}C_\alpha \exp (- c_\alpha \lambda_i^{\alpha} )$,
\begin{equation}
\label{eq:chaining-general-ub}
\max_{\theta\in \Theta}|f(\theta)-f(\theta_0)|
\le \sum_{i=0}^{m-1}\lambda_i\epsilon_i.
\end{equation}
Set $\lambda_i=( t^{\alpha} + \frac{1}{c_\alpha}\log ( 2^{i+1}N_{i+1} ) )^{1/\alpha}$. 
Note that $\epsilon_{i+1}=\epsilon_0 2^{-(i+1)}$.
Then $\lambda_i\le 2^{1/\alpha} F(\epsilon_{i+1})$, where \added{$F(r)\triangleq t+\(\frac{\log \( \epsilon_0 M(r,\Theta,\rho)/r \)}{c_\alpha}\)^{1/\alpha}$} is a decreasing function for $r\leq \epsilon_0$.
\begin{forme}
We can see this from the convexity of the function $g(x) = x^{1/\alpha}$ for $x > 0$ when $\alpha \le 1$. Then
\begin{align*}
(x + y)^{1/\alpha} \le 2^{1/\alpha - 1} \( x^{1/\alpha} + y^{1/\alpha}\).
\end{align*}
Our $\alpha$ is $2/r$ for $r \in [2k-1]$. If $r = 1$, then $\alpha = 2$ and $g(x) = x^{1/\alpha}$ is not convex. But the relationship $\sqrt{x+y} \le \sqrt{x} + \sqrt{y}$ holds anyway. We should just make a statement for general $\alpha$ because it should always work.
\end{forme}
By \prettyref{eq:chaining-general-ub}, \added{there are constants $C', C$} such that with probability $1-C_\alpha \exp (- c_\alpha t^{\alpha})$,
\[
\max_{\theta\in \Theta}|f(\theta)-f(\theta_0)|
\le \added{2^{1/\alpha}} \sum_{i=0}^{m-1} F(\epsilon_{i+1})\epsilon_i
\le \added{C' \cdot 2^{1/\alpha}} \int_{\epsilon_{m+1}}^{\epsilon_1}F(r) \, dr
\le \added{C \cdot 2^{1/\alpha}} \int_{0}^{\epsilon_1}F(r) \, dr.\qedhere
\]
\end{proof}

We now show some properties of the $\epsilon$-coverings used in the main results. 
In the following, let 
$\calN$ be an $(\epsilon,\Norm{\cdot}_2)$-covering of the unit sphere $S^{k-1}$.
Recall the set $\calS$ of distributions in \prettyref{eq:candidates}, where 
$\calW$ is an $(\epsilon,\Norm{\cdot}_1)$-covering of the probability simplex $\Delta^{k-1}$,
$\calA$ is an $(\epsilon,\Norm{\cdot}_2)$-covering of the ball $\{x\in\reals^k:\norm{x}_2\le R\}$.
\begin{lemma}
\label{lem:covering-S}
For any $\gamma\in \calG_{k,k}$,
\[
\min_{\gamma'\in \calS} \max_{\theta\in S^{k-1}}\max_{r\in[\ell]}|m_r(\gamma'_\theta)-m_r(\gamma_\theta) |
\le (R^\ell + \ell R^{\ell-1})\epsilon.
\]
\end{lemma}
\begin{proof}
Suppose $\gamma=\sum_{j=1}^k w_j \delta_{\mu_j}$.
By the definition of the $\epsilon$-covering, there exists $(w_1',\dots,w_k')\in \calW, \mu_j'\in \calA$ such that $\sum_{j=1}^k|w_j-w_j'|<\epsilon_n$, and $\Norm{\mu_j-\mu_j'}_2<\epsilon_n$ for all $j\in[k]$.
Then, for any $\theta\in S^{k-1}$ and $r\in [\ell]$,
\begin{align*}
\abs{\sum_{j=1}^{k} w_j'(\theta^\top \mu_j')^r - w_j(\theta^\top \mu_j)^r }
&\le \sum_{j=1}^{k}|w_j'-w_j||\theta^\top \mu_j'|^r +  \sum_{j=1}^{k} w_j|(\theta^\top \mu_j')^r- (\theta^\top \mu_j)^r|\\
&\le R^r \epsilon + rR^{r-1} \epsilon,
\end{align*}
where we used $\frac{|a^r-b^r|}{|a-b|}\le r|a\vee b|^{r-1}$.
\end{proof}

\begin{lemma}
\label{lem:covering-N}
For any $\gamma,\gamma'\in\calG_{k,k}$,
\[
\sup_{\theta\in S^{k-1}} \max_{r\in[\ell]}|m_r(\gamma_\theta)-m_r(\gamma_\theta')|
\le \max_{\theta\in \calN}  \max_{r\in[\ell]}|m_r(\gamma_\theta)-m_r(\gamma_\theta')| + 2\ell R^{\ell} \epsilon.
\]
\end{lemma}
\begin{proof}
By the definition of the $\epsilon$-covering, for any $\theta\in S^{k-1}$, there exists $\tilde \theta\in\calN$ such that $\Norm{\theta - \tilde \theta}_2<\epsilon$.
Then, by the triangle inequality, for any $r\in[\ell]$,
\begin{equation}
\label{eq:app-covering-S}
|m_r(\gamma_\theta)-m_r(\gamma_\theta')|
\le |m_r(\gamma_\theta)-m_r(\gamma_{\tilde\theta})|
+|m_r(\gamma_{\tilde\theta})-m_r(\gamma_{\tilde\theta}')|
+|m_r(\gamma_{\tilde\theta}')-m_r(\gamma_\theta')|.
\end{equation}
Suppose $\gamma=\sum_{j=1}^k w_j \delta_{\mu_j}$.
% and $\gamma'=\sum_{j=1}^k w_j' \delta_{\mu_j'}$.
Then the first term is upper bounded by
\[
\abs{\sum_{j=1}^{k} w_j(\theta^\top \mu_j)^r - w_j(\tilde\theta^\top \mu_j)^r }
\le \sum_{j=1}^{k} w_j \abs{(\theta^\top \mu_j)^r - (\tilde\theta^\top \mu_j)^r}
\le rR^{r}\epsilon,
\]
where we used $\frac{|a^r-b^r|}{|a-b|}\le r|a\vee b|^{r-1}$.
The same upper bound also holds for the third term of \prettyref{eq:app-covering-S}.
Therefore,
\[
\max_{r\in[\ell]}|m_r(\gamma_\theta)-m_r(\gamma_\theta')|
\le 2\ell R^{\ell}\epsilon + \max_{\theta\in \calN}  \max_{r\in[\ell]}|m_r(\gamma_\theta)-m_r(\gamma_\theta')|.
\]
The conclusion follows by taking the supremum over $\theta\in S^{k-1}$.
\end{proof}

% }

\section{Proof of Lemmas \lowercase{\ref{lem:w1_dim_reduction}--\ref{lem:cov_bound}}} \label{app:proofs_supporting}

%In this subsection we prove Lemmas \ref{lem:w1_dim_reduction}--\ref{lem:cov_bound}.
In this subsection we prove the supporting lemmas for \prettyref{sec:mixing_estimation}.

\begin{proof}[Proof of \prettyref{lem:w1_dim_reduction}]
For the lower bound, simply note that for any $\theta\in S^{d-1}$, $\norm{U - U'}_2 \ge |\theta^\top U - \theta^\top U'|$ by the Cauchy-Schwartz inequality. Taking expectations on both sides with respect to the optimal $W_1$-coupling $\calL(U,U')$ of $\Gamma$ and $\Gamma'$ yields the lower bound.

For the upper bound, we show that there exists $\theta\in S^{d-1}$ that satisfies the following properties:
\begin{enumerate}
	\item The projection $y \mapsto \theta^\top y$ is injective on $\supp(\Gamma) \cup \supp(\Gamma')$;
	
	\item 
For all $y \in \supp(\Gamma)$ and $y'\in\supp(\Gamma')$, we have
\begin{equation}
\label{eq:exist-theta}
%\norm{U - U'}_2 \le k^2\sqrt{d} |\theta^\top U - \theta^\top U'|
\norm{y-y'}_2 \le k^2\sqrt{d} |\theta^\top y - \theta^\top y'|.
\end{equation}
\end{enumerate}
%holds almost surely for any coupling $U\sim \Gamma$ and $U'\sim \Gamma'$.
This can be done by a simple probabilistic argument.
Let $\theta$ be drawn from the uniform distribution on $S^{d-1}$, which fulfills the first property with probability one.
Next, for any fixed $x\in\reals^d$, we have
\begin{align}
\mb P \{ |\theta^\top x | < t\Norm{x}_2 \} 
= \frac{2\pi^{\frac{d-2}{2}}/\added{G}(\frac{d-2}{2})}{2\pi^{\frac{d-1}{2}}/\added{G}(\frac{d-1}{2})} \int_{-t}^t \(1 - u^2 \)^{(d-3)/2} \ du <  t\sqrt{d},
\label{eq:randomsphere}
\end{align}
where $G(t) \triangleq \int_0^\infty x^{t-1} e^{-x} \, dx$ denotes the Gamma function for positive real $t$.
% \begin{forme}
% This uses the fact that $\Gamma(p/2 - 1/2 + 1/2) \approx \Gamma(p/2 - 1/2) (p/2 - 1/2)^{1/2} \approx \sqrt{p}$. 
% \end{forme}
Let $\calX=\{y-y':y\in \supp(\Gamma), y'\in\supp(\Gamma')\}$, whose cardinality is at most $k^2$.
By a union bound, 
\[
\mb P \{ \exists \, x \in \calX \text{ s.t. } |\theta^\top x| < t \norm{x}_2 \} < k^2 t \sqrt{d}, 
\]
and thus
\[
\mb P \{ |\theta^\top x|\ge t\Norm{x}_2,~\forall~x\in\calX \}>1-k^2t\sqrt{d}.
\]
% $|\theta^\top x|\ge t\Norm{x}_2$ for all $x\in \calX$ with probability strictly bigger than $1-k^2t\sqrt{p}$.
This probability is strictly positive for $t = 1 / (k^2 \sqrt{d})$. Thus, there exists $\theta \in S^{d-1}$ such that \prettyref{eq:exist-theta} holds.
Since $\iprod{\theta}{\cdot}$ is injective on the support of $\Gamma$ and $\Gamma'$, denote its inverse by $g: \reals\to \supp(\Gamma) \cup \supp(\Gamma')$. 
Then any coupling of the pushforward measures $\Gamma_\theta$ and $\Gamma'_{\theta}$ gives rise to a coupling of $\Gamma$ and $\Gamma'$ in the sense that if $\calL(V,V')$ is a coupling of $\Gamma_\theta$ and $\Gamma_{\theta'}$ then $\calL(g(V),g(V'))$ is a coupling of $\Gamma$ and $\Gamma'$. By \prettyref{eq:exist-theta}, we have
\[
\norm{g(V)-g(V')}_2 \le k^2\sqrt{d} |V-V'|.
\]
Taking expectations of both sides 
with respect to  $\calL(V,V')$ being the optimal $W_1$-coupling of $\Gamma_\theta$ and $\Gamma'_\theta$ yields the desired upper bound.
\end{proof}

Next we prove~\prettyref{lem:w1_sup_bound}. Note that a simple union bound here would lead to a rate of $(\log n/n)^{1/(4k-2)}$. To remove the unnecessary logarithmic factors, we use the chaining technique (see the general result in~\prettyref{lem:chain_tail_probs}), which entails proving the concentration of the increments of a certain empirical process. 

\begin{forme}
 Here are some notes on the first sentence in the proof of~\prettyref{lem:w1_sup_bound}. The continuity of $\theta \mapsto W_1(\widehat{\gamma_\theta}, \gamma_\theta)$ is necessary so we can say the sup on the compact set $S^{k-1}$ is achieved, i.e., is a max. Note that $W_1(\widehat{\gamma_\theta}, \gamma_\theta)$ is continuous in $\theta$ by the same argument as used to show a metric on a metric space $\mc X$ is continuous on the space $\mc X \times \mc X$ (equipped with the right closeness in paired points). It just uses the triangle inequality of the absolute value on the reals. That is, let $\epsilon > 0$. Pick $\delta = \epsilon/2R$. Let $\norm{\theta_1 - \theta_2}_2 \le \delta$. 
\begin{align*}
|W_1(\widehat{\gamma_{\theta_1}}, \gamma_{\theta_1}) - W_1(\widehat{\gamma_{\theta_2}}, \gamma_{\theta_2})| &\le | W_1(\widehat{\gamma_{\theta_1}}, \gamma_{\theta_1}) - W_1(\widehat{\gamma_{\theta_1}}, \gamma_{\theta_2})| + |W_1(\widehat{\gamma_{\theta_1}}, \gamma_{\theta_2}) - W_1(\widehat{\gamma_{\theta_2}}, \gamma_{\theta_2})| \\
&\le W_1(\gamma_{\theta_1}, \gamma_{\theta_2}) + W_1(\widehat{\gamma_{\theta_1}}, \widehat{\gamma_{\theta_2}}) \\
&\le 2R\norm{\theta_1 - \theta_2}_2,
\end{align*}
where the last line is just by the property of $W_1$ used below.

For the monotone convergence theorem below, here is the argument. The sequence is 
\begin{align*}
0 \le \bone \{ \max_{\theta \in \Theta_n} f(\theta) > t \} \le \bone \{ \max_{\theta \in \Theta_{n+1}} f(\theta) > t \} \le \hdots \le 1
\end{align*}
And we have pointwise convergence of the indicator functions: $\bone \{ \max_{\theta \in \Theta_n} f(\theta) > t \} \rightarrow_{n \rightarrow \infty} \bone \{\sup_{\theta \in S^{k-1}} f(\theta) > t \}$. Therefore,
\begin{align*}
\mb P \{ \max_{\theta \in \Theta_n} f(\theta) > t \} \rightarrow \mb P \{\sup_{\theta \in S^{k-1}} f(\theta) > t \}.
\end{align*}
Our sequence of $\Theta_n$'s here should be a sequence of finite subsets of $S^{k-1}$. See page $2$ of David Pollard's chaining chapter for details.
\end{forme}

\begin{proof} [Proof of \prettyref{lem:w1_sup_bound}]
By the continuity of $\theta \mapsto W_1(\widehat{\gamma_\theta}, \gamma_\theta)$ and the monotone convergence theorem, it suffices to show that \added{there exists a constant $C$} such that, for any finite subset $\Theta \subset S^{k-1}$, 
\begin{align}
\prob{\max_{\theta \in \Theta} W_1(\widehat{\gamma_\theta}, \gamma_\theta) \le \added{C k^{7/2}} n^{-1/(4k-2)} \sqrt{\log \frac{1}{\delta}}} \geq 1-\delta.
%\( \frac{\(\log(1/\delta)\)^{2k-1}}{n}\)^{1/(4k-2)}.
\label{eq:w1_sup_bound1}
\end{align}
% Throughout the proof, $C_k$ stands for a constant depending only on $k$ whose value may vary from line to line.

Recall that $X_1, \hdots, X_n \iiddistr P_\gamma$, where $\gamma\in \calG_{k,k}$. Define the empirical process
\begin{align*}
\tilde m_r(\theta) \triangleq \frac{1}{n} \sum_{i =1}^n H_r(\theta^{\top}X_i),
\end{align*}
where $H_r$ is the degree-$r$ Hermite polynomial defined in~\prettyref{eq:hermites_def}. 
Define the centered random process indexed by $\theta$:
\begin{align}
f_r(\theta) &\triangleq \sqrt{n} \( \tilde m_r(\theta) - \mb E \tilde m_r(\theta)\) = \sqrt{n} \( \tilde m_r(\theta) - m_r(\gamma_\theta)\). \label{eq:empirical_process}
\end{align}
Let $r \in [2k-1]$. By \prettyref{lem:moment_increments_sg}, there are positive constants $C, c$ such that $\mb P \{ |f_r(\theta_1) - f_r(\theta_2)| \ge \norm{\theta_1 - \theta_2}_2 \lambda \} \le C\exp \( - c \lambda^{2/r} /kr \)$. 
So we can apply \prettyref{lem:chain_tail_probs} with $\Theta \subseteq S^{k-1}$, $\rho(\theta_1, \theta_2) = \norm{\theta_1 - \theta_2}_2$, $\epsilon_0=2$, $\alpha = 2/r$, \added{$C_\alpha = C$, and $c_\alpha = c/kr$}. Note that the maximal $\epsilon$-packing of $\Theta$ has size $ M(\epsilon, S^{k-1}, \norm{\cdot}_2) \le \(4/\epsilon\)^{k}$. Fix $\theta_0 \in \Theta$. Then, by~\prettyref{lem:chain_tail_probs}, with probability $1 - \added{C\exp(-c t^{2/r}/kr)}$,
\[
\max_{\theta\in\Theta}|f_r(\theta)-f_r(\theta_0)|
\le \added{C' 2^{r/2}} \pth{t+\int_0^1\(\added{\frac{k^2r}{c}} \log(1/u)\)^{r/2} \, du }
= \added{C' 2^{r/2}} \pth{t+ \added{\(\frac{k^2r}{c}\)^{r/2}} \added{G}\pth{1+\frac{r}{2}}},
\]
where $G(\cdot)$ denotes the Gamma function defined after \prettyref{eq:randomsphere}.
\added{Since $r \le 2k-1$ and $G(1 + (2k-1)/2) \le G(1 + k) \le O( (k/e)^k)$, we obtain that with probability at least $1 - C \exp \( -c t^{2/(2k-1)}/k^2\)$, }
\[
\added{\max_{\theta\in\Theta}|f_r(\theta)-f_r(\theta_0)| \le \added{(C'')^k} (t + k^{4k}),}
\]
\added{for a constant $C''$.} And by~\prettyref{eq:single-ub} in \prettyref{lem:moment_increments_sg}, $|f_r(\theta_0)|\le t$ with probability $1- C\exp(-c t^{2/r}/kr)$. Therefore, with probability $1-\frac{\delta}{2k-1}$,
% \nbwu{Pengkun: here need to assume $\delta \leq \delta_k$; otherwise below is not true.}
\[
\max_{\theta\in\Theta}|\tilde m_r(\theta) -  m_r(\gamma_\theta)|
=\frac{1}{\sqrt{n}}\max_{\theta\in\Theta}|f_r(\theta)|
%\le C_k \frac{(\log(C_k/\delta))^{\frac{2k-1}{2}}}{\sqrt{n}}.
\added{\le (C'')^k  \frac{\(\frac{k^2}{c} \log \(\frac{C (2k-1)}{\delta}\)\)^{\frac{2k-1}{2}} + k^{4k}}{\sqrt{n}} .}
\]
We take a union bound over $r\in[2k-1]$ and obtain that, with probability $1 - \delta$, 
\begin{equation}
\max_{\theta\in\Theta, r\in[2k-1]}|\tilde m_r(\theta) - m_r(\gamma_\theta)|
\added{\le (C'')^k \frac{\(\frac{k^2}{c} \log \(\frac{C (2k-1)}{\delta}\)\)^{\frac{2k-1}{2}} + k^{4k}}{\sqrt{n}} .}
%\le C_k \frac{(\log(C_k/\delta))^{\frac{2k-1}{2}}}{\sqrt{n}}.
\label{eq:w1_sup_bound2}
\end{equation}
Recall that for each $\theta$, the DMM estimator results in a $k$-atomic distribution $\widehat{\gamma_\theta}$, such that $m_r(\widehat{\gamma_\theta}) = \hat m_r(\theta)$ for all $r=1,\ldots,2k-1$, where 
$(\hat m_1(\theta), \ldots, \hat m_{2k-1}(\theta))$ is the Euclidean projection of $\tilde m(\theta)=(\tilde m_1(\theta),\dots,\tilde m_{2k-1}(\theta))$ onto the moment space $\calM_{2k-1}$ (see \prettyref{eq:momentspace}).
Thus, 
\begin{equation}
\label{eq:w1-hat-bound}
\max_{r\in[2k-1]}|m_r(\widehat{\gamma_\theta}) - m_r(\gamma_\theta)|
\le 2 \sqrt{2k-1}\max_{r\in[2k-1]}|\tilde m_r(\theta) - m_r(\gamma_\theta)|.
\end{equation}
By the moment comparison inequality in \prettyref{lem:w1_moments_ub}, we have
\begin{align} \label{eq:moments_conversion}
W_1(\widehat{\gamma_\theta},\gamma_\theta)
\added{\lesssim k} \max_{r\in[2k-1]}|\tilde m_r(\theta) - m_r(\gamma_\theta)|^{1/(2k-1)}.
\end{align}
Finally, maximizing both sides over $\theta \in \Theta$ and applying \prettyref{eq:w1_sup_bound2} yields the desired \prettyref{eq:w1_sup_bound1}. 
\end{proof}

\begin{proof} [Proof of~\prettyref{lem:existence_solution}]
% By \cite[Theorem 1]{WY18} and a union bound, 
By \prettyref{lem:w1_sup_bound}, 
\added{there is a positive constant $C$} such that, for any $\delta\in(0,\frac{1}{2})$, with probability $1 - \delta$, 
\[
W_1(\widehat{\gamma_i}, \gamma_i) \le \epsilon \triangleq \added{C k^{7/2}} n^{-1/(4k-2)} \sqrt{\log \frac{1}{\delta}}
% \( \frac{\log(1/\delta)}{n} \)^{1/(4k-2)}
, \quad \forall~i \in [k].
\]
% \nbwu{Pengkun: there is a problem with this. In GM paper we got the above good tail by doing the median trick. Since now the main thing is to do the chaining analysis, we probably should not do that. Thus, we cannot call \cite[Theorem 1]{WY18}. Instead, we should simply call \prettyref{lem:w1_sup_bound}, to get 
% \[
% W_1(\widehat{\gamma_i}, \gamma_i) \le \epsilon \triangleq C_k n^{-1/(4k-2)} \sqrt{\log \frac{1}{\delta}}, \quad \forall~i \in [k].
% \]
% You need to update the statement and check its propagation.

% I noticed this because the statement of \prettyref{lem:w1_sup_bound} and \prettyref{lem:existence_solution} are not the same (the latter is better somehow). This is suspicious. There is indeed a problem. I think both result should be symmetric and look like the current \prettyref{lem:w1_sup_bound}.
% }

Let $\gamma=\sum_{j=1}^k w_j\delta_{\mu_j}$. Fix $j\in[k]$. For any $i\in [k]$, by definition of $W_1$ distance in \prettyref{eq:wq-def}, 
\[
w_j \cdot \min_{x \in \supp(\widehat{\gamma_i})}  | x-e_i^\top \mu_j | \le W_1(\widehat{\gamma_i}, \gamma_i).
\]
Thus there exists $ \mu_{ji} \in \supp(\widehat{\gamma_i})$ such that
\[
w_j| \mu_{ji} - e_i^\top \mu_j|\le W_1(\widehat{\gamma_i}, \gamma_i)\le \epsilon.
\]
Let $\mu_j' =(\mu_{j1},\dots,\mu_{jk})^\top\in \calA$. Then
\[
w_{j} \norm{\mu'_j - \mu_{j} }_2
\le \sqrt{k}w_{j} \norm{\mu'_j - \mu_{j} }_{\infty} 
\le \sqrt{k}\epsilon.
\]
Since $\calW$ is an $n^{-\frac{1}{4k-2}}$-covering of the probability simplex with respect to $\|\cdot\|_1$, there exists a weights vector $w' =(w'_1,\dots,w'_k)\in \calW$ such that $\norm{w' - w}_1 \le \epsilon$.

Consider the distributions $\gamma' \triangleq \sum_{j=1}^k w_j' \delta_{\mu_j'}\in \calS$ and 
$\gamma'' \triangleq \sum_{j=1}^k w_j\delta_{\mu_j'}$.
Note that  $\gamma$ and $\gamma''$ have the same weights. Using their natural coupling we have $W_1(\gamma,\gamma'') \leq \sum_{j=1}^k w_j\Norm{\mu_j-\mu_j'}_2$.
Note that $\gamma''$ and $\gamma'$ have the same support. 
Using the total variation coupling (see~\cite[Theorem 4]{Gibbs_Su_2002}) of their weights $w$ and $w'$ (and the fact that total variation equals half of the $\ell_1$-distance), 
we have $W_1(\gamma'',\gamma') \leq R\norm{w' - w}_1$.
Thus,
\[
W_1(\gamma,\gamma') \le \sum_{j=1}^k w_j\Norm{\mu_j- \mu_j'}_2
+ R\norm{w' - w}_1
\added{\le k^{3/2} \epsilon + R \epsilon = C (k^{3/2} + R) k^{7/2} n^{-1/(4k-2)} \sqrt{\log \frac{1}{\delta}}.} 
%\le C_k\epsilon.\qedhere
\]
 %and the upper bound of $W_1(\gamma'',\gamma')$ follows from the coupling characterizations of Wasserstein and total variation (see, \eg, \cite[Theorem 4]{Gibbs_Su_2002}).
% by Lemmas~\ref{lem:w1_same_weights} and~\ref{lem:w1_same_support}.
\end{proof}

\begin{proof}[Proof of \prettyref{lmm:DMM-kdim}]
	The proof parallels the simple analysis of the estimator $\hat \gamma^{\circ}$ in \prettyref{eq:ideal-estimator-bound}. 
	Throughout this proof we use the abbreviation $\epsilon \equiv \epsilon_{n,k}$.
	Fix an arbitrary $\gamma'\in\calS$.	
	Then $W_1(\hat\gamma,\gamma) \leq W_1(\hat\gamma,\gamma') + W_1(\gamma',\gamma)$.
	Furthermore,	
	\begin{align*}
W_1(\hat\gamma,\gamma')
\stepa{\lesssim_k} & ~ \Wsliced_1(\hat\gamma,\gamma') = \sup_{\theta \in S^{k-1}} W_1(\hat\gamma_\theta,\gamma'_\theta) \\
\stepb{\leq} & ~ 	\max_{\theta \in \calN} W_1(\hat\gamma_\theta,\gamma'_\theta) + 2 R \sqrt{k} \epsilon \\
\leq & ~ 	\max_{\theta \in \calN} W_1(\widehat{\gamma_\theta},\gamma'_\theta) + \max_{\theta \in \calN} W_1(\widehat{\gamma_\theta},\hat\gamma_\theta) + 2 R \sqrt{k} \epsilon \\
\stepc{\leq} & ~ 	2 \max_{\theta \in \calN} W_1(\widehat{\gamma_\theta},\gamma'_\theta) + 2 R \sqrt{k} \epsilon \\ 
\stepd{\leq} & ~ 	2 \sup_{\theta \in S^{k-1}} W_1(\widehat{\gamma_\theta},\gamma_\theta) + 2 W_1(\gamma',\gamma) + 2 R \sqrt{k} \epsilon,
\end{align*}
where 
(a) is due to from the upper bound in \prettyref{lem:w1_dim_reduction} (with $d=k$);
(b) is by the following argument: 
Recall that $\calN$ is an $(\epsilon,\|\cdot\|_2)$-covering of the unit sphere, so that for any $\theta \in S^{k-1}$, $\|\theta-u\|\leq \epsilon$ for some $u\in \calN$. Since by definition, each estimated marginal $\hat{\gamma_j}$ is supported on $[-R,R]$ and hence any $\gamma'\in\calS$ is supported on the hypercube $[-R,R]^k$. Consequently, $W_1(\gamma_{u}, \gamma_\theta) \leq \sqrt{k} R \|\theta-u\|_2$ by Cauchy-Schwarz and the natural coupling between $\gamma_{u}$ and $\gamma_\theta$;
(c) follows from the optimality of $\hat \gamma$ -- see \prettyref{eq:hatgamma-kdim}; 
(d) uses the lower bound in \prettyref{lem:w1_dim_reduction}.

In summary, by the arbitrariness of $\gamma'\in\calS$, we obtained the following deterministic bound:
\[
%W_1(\hat\gamma, \gamma) \lesssim_{k,R} \sup_{\theta\in S^{k-1}} W_1(\widehat{\gamma_\theta},\gamma_\theta)  + \min_{\gamma'\in\calS} W_1(\gamma,\gamma') + \epsilon.
W_1(\hat\gamma, \gamma) \leq 2 \sup_{\theta\in S^{k-1}} W_1(\widehat{\gamma_\theta},\gamma_\theta)  + 3 \min_{\gamma'\in\calS} W_1(\gamma,\gamma') + 2 R \sqrt{k} \epsilon.
\]
The first and the second terms are bounded in probability by \prettyref{lem:w1_sup_bound} and \prettyref{lem:existence_solution}, respectively, completing the proof of the lemma.
\end{proof}
%The comment that $\hat \gamma_j \in [-R, R] is true because the DMM algorithm takes as input the known bound R and uses it as part of the algorithm and guarantees that the resulting estimator has centers in [-R, R]; see the moment space described in (14) of~\cite{WY18} for how it is used. In our implementation, R = 10 actually.

It remains to show Lemmas \ref{lem:perturbation_psd} and \ref{lem:cov_bound} on subspace estimation.

\begin{proof} [Proof of~\prettyref{lem:perturbation_psd}]
Let $\mc V_r'^{\perp}$ be the subspace of $\mb R^d$ that is orthogonal to the space spanned by the top $r$ eigenvectors of $\Sigma'$, and let $y_j = \argmax_{x \in \mc V_r'^{\perp} \cap S^{d-1}} |\mu_j^{\top}x|$. 
Then \added{$\Norm{\mu_j - \Pi_r' \mu_j}_2^2 = (\mu_j^{\top}y_j)^2$}.  
Furthermore, for each $j$, $y_j^{\top} w_j \mu_j \mu_j^{\top} y_j \le y_j^{\top} (\sum_{\ell = 1}^{k} w_\ell \mu_\ell \mu_\ell^{\top} ) y_j = y_j^{\top} \Sigma y_j$. It remains to bound the latter.
Let $\lambda_1'\ge \hdots \ge  \lambda_d'$ be the sorted eigenvalues of $ \Sigma'$.
Now
\begin{align*}
|y_j^{\top}\Sigma y_j | &\le |y_j^{\top}( \Sigma - \Sigma' ) y_j | +  | y_j^{\top}\Sigma' y_j|  \\
&\le \Norm{\Sigma - \Sigma'}_2 + \lambda_{r+1}'  &\text{ since } y_j \in \mc V_r'^{\perp} \\
&\le 2 \Norm{\Sigma - \Sigma'}_2 + \lambda_{r+1},
\end{align*}
where the last step follows from Weyl's inequality \cite{Horn_Johnson_1991}.
Consequently, by the natural coupling between $\Gamma$ and \added{$\Gamma_{\Pi_r'}$}, 
\[
W_2^2(\Gamma, \added{\Gamma_{\Pi_r'}})\le \sum_{j=1}^k w_j \Norm{\mu_j - \added{\Pi_r'}\mu_j}_2^2
\le k(\lambda_{r+1}+2 \Norm{\Sigma - \Sigma'}_2).\qedhere
\]
\end{proof}

% We now show that for any $j \in [k]$, 
% \[
% w_j \Norm{\mu_j - H_r' \mu_j}_2^2
% \le \lambda_{r+1}+2\Norm{\Sigma-\Sigma'}_2.
% \]

% Let $\mc V_r'^{\perp}$ be the subspace of $\mb R^d$ that is orthogonal to the space spanned by the top $r$ eigenvectors of $\Sigma'$, and let $y = \argmax_{x \in \mc V_r'^{\perp} \cap S^{d-1}} |\mu_j^{\top}x|$. 
% Then $\Norm{\mu_j - H_r' \mu_j}_2^2 = (\mu_j^{\top}y)^2$.  
% And $y^{\top} w_j \mu_j \mu_j^{\top} y \le y^{\top} (\sum_{i = 1}^{k} w_i \mu_i \mu_i^{\top} ) y = y^{\top} \Sigma y$. \begin{forme} Since all are psd. \end{forme} It remains to bound the latter.
% Let $\lambda_1'\ge \hdots \ge  \lambda_d'$ be the sorted eigenvalues of $ \Sigma'$.
% Now
% \begin{align*}
% |y^{\top}\Sigma y | &\le |y^{\top}( \Sigma - \Sigma' ) y | +  | y^{\top}\Sigma' y|  \\
% &\le \Norm{\Sigma - \Sigma'}_2 + \lambda_{r+1}'  &\text{ since } y \in \mc V_r'^{\perp} \\
% &\le 2 \Norm{\Sigma - \Sigma'}_2 + \lambda_{r+1},
% \end{align*}
% where the last step follows from Weyl's inequality \cite{Horn_Johnson_1991}.

\begin{proof}[Proof of~\prettyref{lem:cov_bound}]
Write $X_i = U_i + Z_i$ where $U_i \iiddistr \Gamma$ and $Z_i \iiddistr N(0, I_d)$ for $i = 1, \hdots, n$. Then
\[
\hat \Sigma -\Sigma=\(\frac{1}{n} \sum_{i =1}^n U_i U_i^{\top}-\Sigma\) 
+ \(\frac{1}{n} \sum_{i =1}^n Z_i Z_i^{\top} - I_d\) 
+ \(\frac{1}{n} \sum_{i =1}^n  U_i  Z_i^{\top} + Z_i U_i^{\top}\) .
\]
We upper bound the spectral norms of three terms separately.
For the first term, let $\Gamma=\sum_{j=1}^kw_j\delta_{\mu_j}$ and $\hat w_j = \frac{1}{n} \sum_{i = 1}^n \bone \{ U_i = \mu_j \}$. Then $\frac{1}{n} \sum_{i =1}^n U_i U_i^{\top}=\sum_{j=1}^k\hat w_j\mu_j \mu_j^{\top}$. Therefore, by Hoeffding's inequality and the union bound, with probability $1-2ke^{-2t^2}$,
\begin{equation} \label{eq:ub-uu}
\norm{\frac{1}{n} \sum_{i =1}^n U_i U_i^{\top}-\Sigma}_2 \le R^2 \sum_{j = 1}^k |\hat w_j - w_j| \le \frac{R^2k t}{\sqrt{n}}.
\end{equation}
For the second term, by standard results in random matrix theory (see, e.g.,~\cite[Corollary 5.35]{Vershynin_2012_tutorial}), and since $d < n$, there exists a positive constant $C$ such that, with probability at least $1 - e^{-t^2}$, 
\begin{equation}
\label{eq:ub-zz}
\norm{\frac{1}{n} \sum_{i=1}^n Z_i Z_i^{\top} - I_d}_2 \le C\(\sqrt{\frac{d}{n}}+\frac{t}{\sqrt{n}}+\frac{t^2}{n}\).
\end{equation}

To bound the third term, let $A=\frac{1}{n} \sum_{i =1}^n  U_i  Z_i^{\top} + Z_i U_i^{\top}$, and $\calN$ be an $\frac{1}{4}$-covering of $S^{d-1}$ of size $2^{cd}$ for an absolute constant $c$. Then
\[
\Norm{A}_2 = \max_{\theta \in S^{d-1}} |\theta^\top A \theta| \le 2\max_{\theta\in \mc N}|\theta^\top A \theta|
=4\max_{\theta\in \calN}\abs{\frac{1}{n}\sum_{i=1}^n (\theta^\top U_i) (\theta^\top Z_i)}.
\]
For fixed $\theta\in S^{d-1}$, conditioning on $U_i$, we have $\sum_{i=1}^n (\theta^\top U_i) (\theta^\top Z_i)\sim N(0,\sum_{i=1}^n(\theta^\top U_i)^2)$.
Since $\sum_{i=1}^n(\theta^\top U_i)^2\le nR^2$, we have
\[
\Prob\sth{\abs{\frac{1}{n}\sum_{i=1}^n (\theta^\top U_i) (\theta^\top Z_i)}>\frac{R\tau}{\sqrt{n}}}
\le\Prob\sth{|Z_1|\ge \tau}
\le 2 e^{-\frac{\tau^2}{2}}.
\]
Therefore, by a union bound, with probability $1-2e^{cd-\frac{\tau^2}{2}}$
\[
\max_{\theta\in \calN}\abs{\frac{1}{n}\sum_{i=1}^n (\theta^\top U_i) (\theta^\top Z_i)}
\le \frac{R\tau}{\sqrt{n}}.
\]
By taking $\tau=C(\sqrt{d}+t)$ for some absolute constant $C$, we obtain that with probability $1-e^{-t^2}$,
\begin{equation*}
\norm{\frac{1}{n} \sum_{i =1}^n  U_i  Z_i^{\top} + Z_i U_i^{\top}}_2
\le C R\( \sqrt{\frac{d}{n}}+ \frac{t}{\sqrt{n}}\).\qedhere
\end{equation*}
\end{proof}

\section{Proof of Lemma \lowercase{\ref{lem:l2_moments_lb}}}
	\label{sec:l2_moments_lb}

%\prettyref{lem:l2_moments_lb} provides a tighter relationship between mixture density estimation error in Hellinger distance and mixing distribution moment estimation than the one provided by Lemma 11 of~\cite{WY18}, allowing for the sharp connection between Hellinger distance and $1$-Wasserstein distance provided in~\prettyref{lem:k_to_one_dim}. In~\prettyref{lem:l2_moments_lb}, we mention as a byproduct of the result that we obtain a lower bound on the Hellinger distance as well. This relationship between the Hellinger and $L_2$ distances is not necessarily sharp, since it relies on a simple upper bound on the densities in question. For high dimensions, the upper bound on the densities could be large. However, we will deal with $k$-dimensional densities, so this upper bound is good enough. 

We start by recalling the basics of polynomial interpolation in one dimension. 
For any function $h$ and any set of $m+1$ distinct points $\{x_0,\ldots,x_m\}$, there exists a unique polynomial $P$ of degree at most $m$, such that
$P(x_i)=h(x_i)$ for $i=0,\ldots,m$.
For our purpose, it is convenient to express $P$ in the Newton form (as opposed to the more common Lagrange form):
\begin{equation} \label{eq:newton}
\begin{aligned}
P(x) = \sum_{j = 0}^m a_j \prod_{i = 0}^{j-1} (x - x_i), 
\end{aligned}
\end{equation}
where the coefficients are given by the finite differences of $h$, namely, $a_j = h[x_0, \hdots, x_j]$, which in turn are defined recursively via:
\begin{align*}
h[x_i] &= h(x_i) \\
h[x_i, \hdots, x_{i+r}] &= \frac{h[x_{i+1}, \hdots, x_{i+r}] - h[x_i, \hdots, x_{i+r-1}]}{x_{i+r} - x_i}.
\end{align*}

\begin{proof}[Proof of \prettyref{lem:l2_moments_lb}]
Let $U\sim\gamma$ and $U'\sim\gamma'$.
Note that $p_{\gamma}, p_{\gamma'}$ are bounded above by $\frac{1}{\sqrt{2\pi}}$, so we have 
\begin{align*}
H^2(p_{\gamma}, p_{\gamma'}) = \int \( \frac{p_{\gamma} - p_{\gamma'}}{\sqrt{p_{\gamma}} + \sqrt{p_{\gamma'}}} \)^2 \ge \frac{\sqrt{2\pi}}{4} \|p_{\gamma}- p_{\gamma'}\|_2^2.
\end{align*}
Thus to show \prettyref{eq:hellinger_moments}, it suffices to show there is a positive constant $c$ such that
\begin{align}
\|p_{\gamma}-p_{\gamma'}\|_2 \ge c |\Expect[h(U)]-\Expect[h(U')]|. \label{eq:l2_moments}
\end{align}

Next, by suitable orthogonal expansion we can express $\|p_{\gamma}-p_{\gamma'}\|_2$ in terms of ``generalized moments'' of the mixing distributions (see~\cite[Sec.~VI]{Wu_Verdu_2010}).
Let $\alpha_j(y) = \sqrt{\frac{\sqrt{2} \phi(\sqrt{2}y)}{j!}} H_j(\sqrt{2} y)$. Then $\{\alpha_j: j \in\integers_+\}$ form an orthonormal basis on $L^2(\reals, \diff y)$ in view of \prettyref{eq:hermite-orthod}. 
Since $p_{\gamma}$ is square integrable, we have the orthogonal expansion
$p_\gamma(y) = \sum_{j\geq 0} a_j(\gamma) \alpha_j(y)$, with coefficient
\[
a_j(\gamma)=\Iprod{\alpha_j}{p_\gamma}= \Expect[\alpha_j(U+Z)] = \Expect[(\alpha_j * \phi)(U)] = \frac{1}{2^{\frac{j+1}{2}} 	\pi ^{\frac{1}{4}} \sqrt{j!}}\Expect[U^j e^{-\frac{U^2}{4}}]
\]
where the last equality follows from the fact that \cite[7.374.6, p.~803]{GR}
\[
	(\phi * \alpha_j) (y) = \frac{1}{2^{\frac{j+1}{2}} 	\pi ^{\frac{1}{4}} \sqrt{j!}} y^j e^{-\frac{y^2}{4}}.
\]
Therefore
\begin{align}
\|p_{\gamma}- p_{\gamma'}\|_2^2 &= \sum_{j \ge 0} \frac{1}{j! 2^{j+1} \sqrt{\pi} } \( \Expect[U^j e^{-U^2/4}] - \Expect[U'^j e^{-U'^2/4}] \)^2. \label{eq:l2_equals}
\end{align}
In particular, for each $j\geq 0$,
\begin{align}
|\Expect[U^j e^{-U^2/4}] - \Expect[U'^j e^{-U'^2/4}]| \leq  
\sqrt{j! 2^{j+1} \sqrt{\pi}} \|p_{\gamma}- p_{\gamma'}\|_2.
 \label{eq:l2_expmoments}
\end{align}

% Fix $r \in \integers_+$.
In view of \prettyref{eq:l2_expmoments}, to bound the difference %of the $r$th moments
% $|m_r(U)-m_r(U')|= |\Expect[U^j] - \Expect[U'^j]|$ 
$|\Expect[h(U)] - \Expect[h(U')]|$
by means of $\|p_{\gamma}- p_{\gamma'}\|_2$, our strategy is to interpolate $h(y)$ by linear combinations of
$\{y^j e^{-y^2/4}: j=0,\ldots,2k-1\}$ on all the atoms of $U$ and $U'$, a total of at most $2k$ points.
Clearly, this is equivalent to the standard polynomial interpolation of $\tilde h(y) \triangleq  h(y) e^{y^2/4}$ by a degree-($2k-1$) polynomial. 
%Write $h(y) = \tilde h(y) e^{-y^2/4}$, where $\tilde h(y) = h(y) e^{y^2/4}$. Then do a polynomial interpolation of $\tilde h$.
Specifically, let $T \triangleq  \{ t_1, \hdots, t_{2k}\}$ denote the set of atoms of $\gamma$ and $\gamma'$. 
By assumption, $T \subset [-R,R]$.
Denote the interpolating polynomial of $\tilde h$ on $T$ by $P(y) = \sum_{j=0}^{2k-1}  b_j y^j$.
Then
\begin{align}
|\mb E[h(U)] - \mb E[h(U')] | 
&= |\Expect[P(U)e^{-U^2/4}\added{]} - \Expect[P(U')e^{-U'^2/4}\added{]}  | \\
&\le \sum_{j=0}^{2k-1} |b_j| | \Expect[U^j e^{-U^2/4}] - \Expect[U'^j e^{-U'^2/4}] |  \\
&\le \|p_{\gamma}- p_{\gamma'}\|_2 \sum_{j=0}^{2k-1} |b_j| \sqrt{j! 2^{j+1} \sqrt{\pi}}. \label{eq:h_ub}
\end{align}
It remains to bound the coefficient $b_j$ independently of the set $T$. \added{In the notation of~\prettyref{lmm:interp}, with $m = 2k-1$,}
\begin{align*}
P(x) = \sum_{j = 0}^m \tilde h[x_0, \hdots, x_j] \prod_{i = 0}^{j-1} (x - x_i) \le \sum_{j = 0}^m \frac{1}{j!} \sup_{|\xi|\leq R}|\tilde h^{(j)}(\xi)| \sum_{i = 0}^j c_{ij} x^i  \le C_1 \sum_{j = 0}^m \frac{1}{j!} \sup_{|\xi| \leq R}|\tilde h^{(j)}(\xi)| \sum_{i = 0}^j x^i,
\end{align*}
where $C_1 \le (1+R)^m$.
Thus $|b_j| \le C_1 \sum_{i = 0}^j \frac{1}{i!} \sup_{|\xi| \le R} |\tilde h^{(i)}(\xi)|$. 

\added{In the specific case of $\tilde h(y) = h(y) e^{y^2/4}$, by the Leibniz rule \cite[0.42, p.~22]{GR}, $\tilde h^{(i)}(y) = \sum_{l = 0}^i {i \choose l} h^{(i-l)}(y) (e^{y^2/4})^{(l)}$. Furthermore}
\begin{align*}
(e^{y^2/4})^{(l)} 
=  \frac{e^{y^2/4}}{2^{l}}\sum_{i = 0}^{\lfloor{l/2}\rfloor} {l \choose 2i}\frac{(2i)!}{i!}y^{l-2i}
\le \frac{e^{R^2/4}}{2^{l}}l^{l/2}\sum_{i = 0}^{\lfloor{l/2}\rfloor} {l \choose 2i} R^{l-2i}
\le \frac{e^{R^2/4}}{2^{l}}l^{l/2}(1+R)^l
% \frac{e^{y^2/4}}{2^{l/2} (\sqrt{-1})^l} H_l\(\frac{y \sqrt{-1}}{\sqrt{2}} \) &= \frac{e^{y^2/4}}{2^{l/2}} H_l\(\frac{y}{\sqrt{2}}\) \\
% &\le \frac{e^{y^2/4}}{2^{l/2}} \sum_{i = 0}^{\lfloor{l/2}\rfloor}  \frac{(2i)!}{2^i i!} {l \choose 2i} \(\frac{R}{\sqrt{2}}\)^{l-2i} &\text{ using~\prettyref{eq:hermites_def}} \\
% %&\le \frac{e^{R^2/4}}{2^{l/2}} \sum_{i = 0}^{\lfloor{l/2}\rfloor} i^i  {l \choose 2i} \(\frac{R}{\sqrt{2}}\)^{l-2i} \\
% &\le C' \frac{e^{R^2/4}}{2^{l/2}} \(\frac{l}{2}\)^{l/2} \sum_{i = 0}^{\lfloor{l/2}\rfloor} {l \choose 2i} \(\frac{R}{\sqrt{2}}\)^{l-2i} \\
% &\le C' \frac{e^{R^2/4}}{2^{l/2}} \(\frac{l}{2}\)^{l/2} (1 + R/\sqrt{2})^l \\
\le (C' \sqrt{l})^l.
\end{align*}
Here and below, $C',C'',C'''$  are constants depending only on $R$.
% changes from line to line and 
%may depend only on $R$. 
In the special case of $h(y) = y^r$, $h^{(m)}(y) = \frac{r!}{(r-m)!} y^{r-m} \bone\{m \le r \}$. Thus
\begin{align*}
|b_j| 
&\le C_1 \sum_{i = 0}^j \sum_{l = 0}^i \frac{1}{i!} {i \choose l} \frac{r!}{(r-(i-l))!} R^{r-(i-l)} \bone \{i-l \le r\} (C' \sqrt{l})^l \\
&= C_1 \sum_{i = 0}^j \sum_{l = 0}^i \frac{(C' \sqrt{l})^l}{l!} {r \choose i-l} R^{r-(i-l)} \bone \{i-l \le r\} \\
%&\le C_1 \sum_{i = 0}^j \sum_{l = 0}^i \frac{(C' \sqrt{j})^l}{l!}  {r \choose i-l} R^{r-(i-l)} \bone \{i-l \le r\} \\
&\le C_1 C''\sum_{i = 0}^j \sum_{l = 0}^{i\wedge r}{r \choose l} R^{r-l} \\
&\le C_1 C'' j (1+R)^r,
% &\le C_1 e^{C' \sqrt{j}} \sum_{i = 0}^j (j+1-i){r \choose i} R^{r-i} \bone \{i \le r\} \\
% &\le C_1 (C')^j (1+R)^r,
\end{align*}
Since $C_1 \le (1+R)^m, |b_j| \le C'j (1+R)^{r+m} \le C'j (1+R)^{4k}$. Plugging into~\prettyref{eq:h_ub},
 % and using Stirling's upper bound,}
\begin{align*}
|\mb E[h(U)] - \mb E[h(U')] | 
%&\le \|p_{\gamma} - p_{\gamma'}\|_2 (1+R)^{4k} \sum_{j=0}^{2k-1} \sqrt{(C'k)^{2k}}\\
% &\le \|p_{\gamma} - p_{\gamma'}\|_2 (1+R)^{4k} \sum_{j=0}^{2k-1} C'j \sqrt{j^{j+1/2} \(\frac{2}{e}\)^j 2e\sqrt{\pi}} \\
% &\le \|p_{\gamma} - p_{\gamma'}\|_2 (1+R)^{4k} \sum_{j=0}^{2k-1} (C'\sqrt{j})^j  \\
&\le \|p_{\gamma} - p_{\gamma'}\|_2 (C''' k)^k.
\end{align*}
%\added{where again $C'$ changes from line to line and may depend only on $R$.}
\end{proof}

\begin{lemma}
\label{lmm:interp}	
	Let $h$ be an $m$-times differentiable function on the interval $[-R,R]$, whose derivatives are bounded by
	$ |h^{(i)}(x)| \leq M$ for all $0\leq i\leq m$ and all $x\in[-R,R]$. 
	Then for any $m\geq 1$ and $R>0$, there exists a positive constant $C=C(m,R,M)$, such that the following holds.
	For any set of distinct nodes $T=\{x_0,\ldots,x_m\} \subset [-R,R]$, denote by $P(x)=\sum_{j=0}^{m} b_j x^j$ 
	the unique interpolating polynomial of degree at most $m$ of $h$ on $T$.
	Then $\max_{0\leq j \leq m} |b_j| \leq C$.	
\end{lemma}
\begin{proof}
Express $P$ in the Newton form~\prettyref{eq:newton}:
\begin{align*}
P(y) &= \sum_{j = 0}^{m} h[x_0, \hdots, x_j] \prod_{i = 0}^{j-1} (y - x_i)
\end{align*}
By the intermediate value theorem, finite differences can be bounded by derivatives as follows: (c.f.~\cite[(2.1.4.3)]{Stoer_Bulirsch_2002})
\[
|h[x_0, \hdots, x_j]| \leq \frac{1}{j!} \sup_{|\xi|\leq R}|h^{(j)}(\xi)|.
\]
Let $\prod_{i = 0}^{j-1} (y - x_i)  = \sum_{i = 0}^{j} c_{ij} y^i$. Since $|x_i|\leq R$, 
\added{$|c_{ij}| \leq {j \choose i} R^{j-i} \le (1+R)^j \le (1+R)^m$}
% {$|c_{ij}| \leq C_1=C_1(R,m)$} 
for all $i,j$. This completes the proof.
\end{proof}

\section{Proof of Theorem~\lowercase{\ref{thm:mixing_est_unknownsigma}}}
\label{app:mixing_est_unknownsigma}

The proof is completely analogous to that of \prettyref{thm:mixing_est}. So we only point out the major differences.
First of all, let us define the estimator $(\hat\Gamma,\hat\sigma^2)$.
Let $\hat\Gamma$ be obtained from the procedure in \prettyref{sec:dtok}, except that
in \prettyref{algo:gmm_est} we change the grid size in \prettyref{eq:gridsize-mixing} to $n^{-1/(4k)}$ and 
replace each use of the DMM estimator with Lindsay's estimator \cite{Lindsay_1989} for $k$-GM with an unknown common variance (see also \cite[Algorithm 2]{WY18}). Finally, $\hat\sigma^2$ can be obtained by applying Lindsay's estimator to the first coordinate.

The proof of the upper bound follows that of \prettyref{thm:mixing_est} in \prettyref{ssec:proof_main}. Specifically, 
\begin{itemize}
	\item For the subspace estimation error, Lemmas \ref{lem:perturbation_psd}--\ref{lem:cov_bound} continue to hold with $\hat \Sigma$ defined as $\hat \Sigma  = \frac{1}{n}\sum_{i=1}^n X_iX_i^\top - \sigma^2 I_d$, due to the crucial observation that the top $k$ left singular vectors of $[X_1,\ldots,X_n]$ coincide with 
the top $k$ eigenvectors of this $\hat \Sigma$, so that the estimation error of $\sigma^2$ does not contribute to that of the subspace.
(This fact crucially relies on the isotropic assumption.)

%, where $\sigma^2$ is the actual variance parameter.
\item For the estimation error of $k$-dimensional mixing distribution post-projection, 
the conclusions of Lemmas \ref{lem:w1_sup_bound}--\ref{lmm:DMM-kdim} remain valid with 
all $O_k(n^{-1/(4k-2)})$ replaced by $O_k(n^{-1/(4k)})$.
This is because the same bound in \prettyref{eq:w1_sup_bound2} applies to the first $2k$ moments and all projections, 
and the theoretical analysis of Lindsay's estimator in \cite[Theorem 1]{WY18} only relies on the maximal difference of the first $2k$ moments; see \cite[Sec.~4.2]{WY18}, in particular, the moment comparison inequality in Proposition 3 therein.
\end{itemize}

Finally, the minimax lower bound $\Omega((d/n)^{1/4})$ follows from \prettyref{thm:mixing_est} and $\Omega(n^{-1/{4k}})$ follows from \cite[Proposition 9]{WY18} for one dimension; the latter result also contains the lower bound $\Omega(n^{-1/{2k}})$ for estimating $\sigma^2$.

\begin{forme}

\subsection{Why we can't imitate proof of~\cite{WY18} exactly in doing concentration of Hermite process}

In the DMM estimator of~\citet{Wu_Yang_2019}, which is what we rely on in Algorithm~\prettyref{algo:gmm_est}, the first $2k-1$ Hermite empirical processes are used to form a system of equations that are solved to obtain the mixing distribution estimate. The analysis of this estimate relies on the analysis of the empirical Hermites and their closeness to the true moments of $\tilde \gamma_0$. Accordingly, define the centered random process in $\theta$:
\begin{align}
f_r(\theta) &=  \sqrt{n} \( \tilde m_r(\theta^{\top}x) - \mb E_{y, z} \tilde m_r(\theta^{\top}x)\).  \label{eq:empirical_process}
\end{align}
Note that the sub-Gaussian type concentration in~\citet{Wu_Yang_2019} was proved using a variance bound and then the so-called ``median trick'' and Chebyshev. The median trick works as follows. We split the data into $T$ samples, $\{x_t \}_{t \in [T]}$. We run the algorithm on each sample and take our estimator to be the median of these estimators; we use facts about the concentration of the binomial distribution to show concentration of the resulting median estimator. For the chaining technique, we need to have sub-Gaussian type concentration of the increments $|f_r(\theta_1, x) - f_r(\theta_2, x)|$. We can indeed apply the median trick to the increments $|f_r(\theta_1, x_t) - f_r(\theta_2, x_t)|$ and obtain sub-Gaussian type concentration of the increments. But $med_{t \in [T]} |f_r(\theta_1, x_t) - f_r(\theta_2, x_t)| \ne med(f_r(\theta_1, x_t)) - med(f_r(\theta_2, x_t))$. That is, obtaining sub-Gaussian type concentration of the median of the increments would not result in a statement about concentration of the actual process $med_{t \in [T]} f_r(\theta, x_t)$ that we use in our estimation. We instead rely on some concentration results for polynomials from the literature.

We also sharpen the results of~\cite{WY18} by using~\prettyref{lem:hypercontractivity}. We do a similar variance calculation to~\cite{WY18}, which leads to a variance that looks something like $(2k)^k$. So their density estimator was with rate like $k^{k/2} n^{-1/2}$. Their mixing estimator had $k^{3/2} n^{-1/(4k-2)}$; this is because the $k^k$ gets hit by the $1/(4k-2)$. 

But in Hypercontractivity, we then obtain something like a concentration of:
\begin{align*}
\exp \( - \frac{\lambda^2}{Var g}\)^{1/(2k-1)} \approx \exp \( - \frac{n\lambda^2}{k^k}\)^{1/(2k-1)} = \exp \( - \frac{n\lambda^2)^{1/(2k-1)}}{k} \).
\end{align*}
So this sharpens that $k^k$. However, then we do get with probability $1 - \delta$, the concentration of $\frac{k\log (1/\delta)^{2k-1}}{n}$, which is less sharp in $\delta$. But in fact in $W_1$ distance, that we basically have a factor of $k^{1/(4k-2)}$, which is very small.

Note that our way of density estimation doesn't use this stuff; it just uses the MLE, which again sharpens~	\cite{WY18} since we only get $\sqrt{(2k-1)/n}$ where the $2k-1$ is the number of parameters; it comes from the bracketing number.

\end{forme}

\bibliographystyle{alpha}
\bibliography{strings,mine,refs}

\end{document}